\documentclass[reqno,10pt]{amsart}
\numberwithin{equation}{section}
\usepackage{amsmath}
\usepackage{esint} % permette di fare l'integrale tagliato
\usepackage{amsthm}
\usepackage{epsfig}
\usepackage{psfrag}
\usepackage{graphicx}
\usepackage{graphpap,latexsym,epsf}
\usepackage{color}
\usepackage{amssymb,mathrsfs,enumerate}
\usepackage{mathtools}
\definecolor{citegreen}{rgb}{0,0.6,0}
\definecolor{refred}{rgb}{0.8,0,0}
\usepackage[colorlinks, citecolor=citegreen, linkcolor=refred]{hyperref}
\newcommand{\R}{\mathbb{R}}
\newcommand{\N}{\mathbb{N}}
\newcommand{\Sph}{\mathbb{S}}

\def\HHH{{\rm H}}
\def\RRR{{\mathrm R}}

\def\a{\alpha}
\def\b{\beta}

\newcommand{\pa}{\partial}

\newcommand{\ffi}{\varphi}

\newcommand{\ep}{\varepsilon}
\newcommand{\rmd}{{\rm d}}

\newcommand{\go}{g_0}
\newcommand{\cgo}{g^{(0)}}

\newcommand{\Ric}{{\rm Ric}}
\newcommand{\cRic}{{\rm R}^{(0)}}

\newcommand{\D}{{\rm D}}
\newcommand{\DD}{{\rm D}^2}
\newcommand{\De}{\Delta}

\newcommand{\cho}{{\rm h}^{(0)}}
\newcommand{\Ho}{{\rm H}}
\newcommand{\Cr}{{\rm G}}

\newcommand{\g}{g}

\newcommand{\Ricg}{{\rm Ric}_g}
\newcommand{\cRicg}{{\rm R}^{(g)}}
\newcommand{\Rg}{{\rm R}_g}
\newcommand{\na}{\nabla}
\newcommand{\nana}{\nabla^2}
\newcommand{\Deg}{\Delta_g}
\newcommand{\hg}{{\rm h}_g}
\newcommand{\chg}{{\rm h}^{(g)}}
\newcommand{\Hg}{{\rm H}_g}

\mathchardef\emptyset="001F

\definecolor{vgreen}{rgb}{0.1,0.5,0.2}
\definecolor{viola}{RGB}{85,26,139}
\renewcommand{\theequation}{\thesection.\arabic{equation}}

\newtheorem{theorem}{Theorem}[section]
\newtheorem{remark}{Remark}
\newtheorem{corollary}[theorem]{Corollary}
\newtheorem{definition}{Definition}
\newtheorem{proposition}[theorem]{Proposition}

\newtheorem{assumption}{Assumption}
\newtheorem{notation}{Notation}
\newtheorem{normalization}{Normalization}
\newtheorem{lemma}[theorem]{Lemma}

\newtheorem{innercustomass}{Assumption}
\newenvironment{customass}[1]
  {\renewcommand\theinnercustomass{#1}\innercustomass}
  {\endinnercustomass}
\newtheorem{theoremapp}{Theorem}[subsection]
\newtheorem{corollaryapp}[theoremapp]{Corollary}
\newtheorem{propositionapp}[theoremapp]{Proposition}
\newtheorem{lemmaapp}[theoremapp]{Lemma}
\newtheorem{remarkapp}{Remark}[subsection]
\begin{document}

\hyphenation{ma-ni-fold}

\title[Monotonicity formulas for static metrics with non-zero cosmological constant
]{Monotonicity formulas for static metrics \\ with non-zero cosmological constant}

\author[S.~Borghini]{Stefano Borghini}
\address{S.~Borghini, Scuola Normale Superiore di Pisa,
Piazza Cavalieri 7, 56126 Pisa, Italy}
\email{stefano.borghini@sns.it}

%\author[P.~Chru\'sciel]{Piotr Chru\'sciel}
%\address{P.~Chru\'sciel, University of Vienna, 
%Boltzmanngasse 5, A 1090 Vienna, Austria}
%\email{Piotr.Chrusciel@univie.ac.at}

\author[L.~Mazzieri]{Lorenzo Mazzieri}
\address{L.~Mazzieri, Universit\`a degli Studi di Trento,
via Sommarive 14, 38123 Povo (TN), Italy}
\email{lorenzo.mazzieri@unitn.it}

% \thanks{}

\begin{abstract} 

In this paper we adopt the approach presented in~\cite{Ago_Maz_1,Ago_Maz_2} to study non-singular vacuum static space-times with non-zero cosmological constant. We introduce new integral quantities, and under suitable assumptions we prove their monotonicity along the level set flow of the static potential. We then show how to use these properties to derive a number of sharp geometric and analytic inequalities, whose equality case can be used to characterize the rotational symmetry of the underlying static solutions. As a consequence, we are able to prove some new uniqueness statements for the de Sitter and the anti-de Sitter metrics. In particular, we show that the de Sitter solution has the least possible surface gravity among three-dimensional static metrics with connected boundary and positive cosmological constant.

\end{abstract}

\maketitle

\noindent\textsc{MSC (2010): 
35B06,
%PDE - symmetries, invariants of pdes
\!53C21,
%methods of Riem Geom (including pdes method)
\!83C57,
%black holes
\!35N25.
%overdetermined bvp
}

\smallskip
\noindent 

\date{\today}

\maketitle

%%%%%%%%%%%%%%%%%%%%%%%%%%%%%%%%%
%%%%%%%%%%%%%%%%%%%%%%%%%%%%%%%%%

\section{Introduction}

Throughout this paper we let $(M,\go)$ be an $n$-dimensional Riemannian manifold, $n \geq 3$, with (possibly empty) smooth compact boundary $\pa M$.

\subsection{Static Einstein system.}

Consider positive functions $u \in {\mathscr C}^\infty  (M)$ such that the triple $(M, g_0, u)$ satisfies the {\em static Einstein system} 
\smallskip
\begin{equation}
\label{eq:SES}
\begin{dcases}
u\,\Ric=\DD u+ \frac{2\Lambda}{n-1}\,u\,\go, & \mbox{in } M\\
\ \;\, \De u=-\frac{2\Lambda}{n-1}\, u, & \mbox{in } M
\end{dcases}
\end{equation}

\noindent where $\Ric$, $\D$, and $\De$ represent the Ricci tensor, the Levi-Civita connection, and the Laplace-Beltrami operator of the metric $g_0$, respectively, and $\Lambda\in\R$ is a constant called {\em cosmological constant}.
Note that a consequence of the above equations is that the scalar curvature is
$$
\RRR=2\Lambda \, .
$$

We notice that the equations in~\eqref{eq:SES} are assumed to be satisfied in the whole $M$ in the sense that they hold in $M \setminus \pa M$ in the classical sense and if we take the limits of both the left hand side and the right hand side, they coincide at the boundary. In the rest of the paper the metric $g_0$ and the function $u$ will be referred to as {\em static metric} and {\em static potential}, respectively, whereas the triple $(M,g_0, u)$ will be called a {\em static solution}. A classical computation shows that if $(M,g_0,u)$ satisfies~\eqref{eq:SES}, then the Lorentzian metric $\gamma = -u^2 \, dt \otimes dt + g_0$ satisfies the {\em vacuum Einstein equations}
\smallskip
\begin{equation*}
\Ric_\gamma \, = \, \frac{2\Lambda}{n-1}\,\gamma \,  \quad \hbox{ in \,\, $\R \times (M \setminus \pa M)$} \, .
\end{equation*}

Throughout this work we will be interested to the case $\Lambda\neq 0$ (see~\cite{Ago_Maz_2} for the case $\Lambda=0$). If $\Lambda>0$ (respectively $\Lambda<0$) we can rescale the metric to obtain $\Lambda=\frac{1}{2}n(n-1)$ (respectively $\Lambda=-\frac{1}{2}n(n-1)$). We recall that the simplest solutions of the rescaled problem~\eqref{eq:SES} are given by the {\em de Sitter solution} 
\begin{equation}\label{eq:D}
(M,g_0,u)=\left(\mathbb{D}^n\, , \,g_D=\frac{d|x|\otimes d|x|}{1-|x|^2}+|x|^2 g_{\Sph^{n-1}}\, , \, u_D=\sqrt{1-|x|^2}\right) \, ,
\end{equation}
where $\mathbb{D}^n:=\{x\in\R^n\,:\,|x|<1\}$ is the $n$-disc, when the cosmological constant is {\em positive}, and by the {\em anti-de Sitter solution}
\begin{equation}\label{eq:A}
(M,g_0,u)=\left(\,\R^n\, , \,g_A=\frac{d|x|\otimes d|x|}{1+|x|^2}+|x|^2 g_{\Sph^{n-1}}\, , \, u_A=\sqrt{1+|x|^2}\,\right) \, ,
\end{equation}
when the cosmological constant is {\em negative}.

\subsection{Setting of the problem and statement of the main results (case $\Lambda>0$).}
\label{sub:settingandstatement_D}

In the case $\Lambda>0$ it seems physically reasonable (see for instance~\cite{Ambrozio,Hij_Mon_Rau}) to suppose that $M$ is compact with non-empty boundary, and that $u\in\mathscr{C}^{\infty}(M)$ is a nonnegative function (strictly positive in ${\rm int}(M)$) which solves the problem
%we will restrict our analysis to the physically reasonable assumption that $M$ is closed and compact. With this hypotesis, it is well known that the set $\{u=0\}$ is necessarily non-empty, and the quantity $|\D u|$ is a striclty positive constant on each connected component of $\{u=0\}$. Therefore, it makes sense to restrict our analysis to one of the connected component of $M\setminus\{u=0\}$.
%For this reason, in what follows we will suppose $M$ to be a compact manifold with non-empty boundary and 

\begin{equation}
\label{eq:pb_D}
\begin{dcases}
u\,\Ric=\DD u+n\,u\,g_0, & \mbox{in } M\\
\ \;\,\De u=-n\, u, & \mbox{in } M\\
\ \ \ \ \; u=0, & \mbox{on } \pa M
\end{dcases}
\end{equation}

We notice that the first two equations coincide with the equations of the rescaled problem~\eqref{eq:SES} in the case of a positive cosmological constant. 
\begin{normalization}
\label{norm:D}
Since the problem is invariant under a multiplication of $u$ by a positive constant, without loss of generality we will suppose from now on $\max_M(u)=1$. We also let 
\begin{equation*}
{\rm MAX}(u)=\{p\in M \, : \, u(p)=1\}
\end{equation*}
be the set of the points that realize the maximum.
\end{normalization} 

Recall that, since $u=0$ on $\pa M$, the first equation of problem~\eqref{eq:pb_D} implies that $\DD u=0$ on $\pa M$. Therefore, $|\D u|$ is constant (and different from zero, see~\cite[Lemma~3]{Ambrozio}) on each connected component of $\pa M$. The positive constant value of $|\D u|$ on a connected component of $\pa M$ is known in the literature as the {\em surface gravity} of the connected component. It is easily seen that the surface gravity of the boundary of the {\em de Sitter solution}~\eqref{eq:D} is equal to $1$. Thus, it makes sense to consider the following hypotesis, that will play a fundamental role in what follows.

\begin{assumption}
\label{ass:D}
The surface gravity on each connected component of the boundary is less than or equal to $1$, namely, $|\D u|\leq 1$ on $\pa M$.
\end{assumption}

We notice that the de Sitter triple $(\mathbb{D}^n,g_D,u_D)$ defined by~\eqref{eq:D} is still a static solution of the rescaled problem~\eqref{eq:pb_D} and satisfies Normalization~\ref{norm:D} and Assumption~\ref{ass:D}.
On the other hand, Assumption~\ref{ass:D} rules out other known solutions of~\eqref{eq:pb_D}, such as the de Sitter-Schwarzschild triple
\begin{equation}
\label{eq:deS-Schw}
\Big(M=[r_1(m),r_2(m)]\times\Sph^{n-1}\,,\,
\go=\frac{dr\otimes dr}{1-r^2-2mr^{2-n}}+r^2 g_{\Sph^{n-1}}\,,\,
u=\sqrt{1-r^2-2mr^{2-n}}\Big)\,,
\end{equation}
where $m\in\Big(0,\sqrt{\frac{(n-2)^{n-2}}{n^n}}\,\Big)$ and $r_1(m),r_2(m)$ are the two positive solutions of $1-r^2-2mr^{2-n}=0$ (once $u$ is rescaled according to Normalization~\ref{norm:D}, it can be seen that the surface gravity of the event horizon $r=r_1(m)$ is greater than 1 for all $m$), and the Nariai solution
\begin{equation}
\label{eq:cylsol_D}
\Big(M=[0,\pi]\times\Sph^{n-1}\,,\,
\go=\frac{1}{n}\,\big[dr\otimes dr+(n-2)\,g_{\Sph^{n-1}}\big]\,,\,
u=\sin(r)\Big)
\end{equation}
which has $|\D u|=\sqrt{n}$ at both its boundaries.

Proceeding in analogy with~\cite{Ago_Maz_2}, we are now ready to introduce, for all $p\geq 0$, the functions $U_p:[0,1)\rightarrow\R$ given by
\begin{equation}
\label{eq:Up_D}
t \,\, \longmapsto \,\, U_p(t) \, = \, \Big(\frac{1}{1-t^2}\Big)^{\!\!\frac{n+p-1}{2}}\!\!\!\!\!\! \int\limits_{ \{ u = t \}} \!\!\!\!  |\D u|^p \, \rmd \sigma .
\end{equation}
It is worth noticing that the functions $t \mapsto U_p(t)$ are well defined, since the integrands are globally bounded and the level sets of $u$ have finite hypersurface area. In fact, since $u$ is analytic (see for example~\cite{Chr}), the level sets of $u$ have locally finite $\mathscr{H}^{n-1}$-measure by the results in~\cite{Federer}. Moreover, they are compact and thus their hypersurface area is finite. To give further insights about the definition of the functions $t \mapsto U_p(t)$, we note that, using the explicit formul\ae~\eqref{eq:D}, one easily realizes that the quantities
\begin{equation}
\label{eq:Pfunct_confvol_D}
M \, \ni \,x \, \longmapsto \,\,  \frac{|\D u|}{\sqrt{1-u^2}}(x) \quad \hbox{and} \quad  [0, 1) \, \ni \, t \, \longmapsto \,\, U_0(t) \,= \!\!\!\!\int\limits_{\{u = t\}}\!\!
 \Big(\frac{1}{1-u^2}\Big)^{\!\frac{n-1}{2}}
 \rmd\sigma 
\end{equation}
are constant on the de Sitter solution. In the following, via a conformal reformulation of problem~\eqref{eq:pb_D}, we will be able to give a more geometric interpretation of this fact. On the other hand, we notice that the function $t \mapsto U_p(t)$ can be rewritten in terms of the above quantities as
\begin{equation}
\label{eq:Up2_D}
{ U}_p(t)\,\,=\!\!\!
\int\limits_{\{u = t\}}\!\!\!
\left( \frac{|\D u|}{\sqrt{1-u^2}} \right)^{p}  \Big(\frac{1}{1-u^2}\Big)^{\!\frac{n-1}{2}}
\,\, \rmd\sigma .
\end{equation}
Hence, thanks to~\eqref{eq:Pfunct_confvol_D}, we have that for every $p \geq 0$ the function $t \mapsto U_p(t)$ is constant on the de Sitter solution. Our main result illustrates how  the functions $t \mapsto U_p(t)$ can be also used to detect the rotational symmetry of the {\em static solution} $(M, g_0, u)$. In fact, for $p\geq 3$, they are nonincreasing and the monotonicity is strict unless $(M, g_0, u)$ is isometric to the de Sitter solution.

\begin{theorem}[Monotonicity-Rigidity Theorem, case $\Lambda>0$]
\label{thm:main_D}
Let $(M,g_0,u)$ be a static solution to problem~\eqref{eq:pb_D} satisfying Normalization~\ref{norm:D} and Assumption~\ref{ass:D}. 
Then
\begin{equation}
\label{eq:Dusign_D}
{|\D u|^2}\,\,\leq\,\, {1-u^2}\,, \quad \hbox{in $M$.}
\end{equation}
%on the whole manifold $M$.
Moreover, the functions $U_p : [0, 1) \rightarrow \R$  defined  in~\eqref{eq:Up_D} satisfy the following properties.

\begin{itemize}
\item[(i)] For every $p\geq 1$, the function $U_p$ is continuous.

\smallskip

\item[(ii)] The function $U_1$ is monotonically nonincreasing. Moreover, if $U_1(t_1)=U_1(t_2)$ for some $t_1\neq t_2$, then $(M,\go,u)$ is isometric to the de Sitter solution.

\smallskip

\item[(iii)] For every $p \geq 3$, the function $U_p$ is differentiable and the derivative satisfies, for every $t \in [0,1)$,
\begin{align}
\notag
\,\,\,\,\,\,U_p'(t) \,\,& = \,\, - \, (p-1) \,\, t \,\, \Big( \frac{1}{1-t^2} \Big)^{\!\!\frac{n+p-1}{2}} \!\!\!\!\!
\int\limits_{\{u=t\}}
\!\!\!\!
|\D u|^{p-2}
\!\left[\,\, \bigg| \frac{\D u}{u} \bigg| \, \HHH \, + \, \bigg( \!\frac{np}{p-1}\! \bigg) \, - \, \bigg(\!\frac{n+p-1}{p-1} \!\bigg) \bigg( \frac{|\D u|^2}{1-u^2} \bigg)
\, \right]  \rmd\sigma \,
\\
\notag
\phantom{\qquad\quad\,\,}&= \,\, - \, (p-1) \,\, t \,\, \Big( \frac{1}{1-t^2} \Big)^{\!\!\frac{n+p-1}{2}} \!\!\!\!\!
\int\limits_{\{u=t\}}
\!\!\!\!
|\D u|^{p-2}
\!\left[\, (n-1) \, - \, \Ric (\nu, \nu) \, +   \bigg(\!\frac{n+p-1}{p-1} \!\bigg) \bigg(\!1- \frac{|\D u|^2}{1-u^2} \!\bigg)
\, \right]  \rmd\sigma \,
\\
%\!\!\!\!\!\left(1-t^2\right)^{\!\frac{n+p-1}{2}} U_p'(t)
%=
%-\,(p-1)\!\!\!
%\int\limits_{\{u=t\}}
%\!\!\!
%|\D u|^{p-1}
%\!\left[\, \HHH+
%(n-1)
%\, \frac{u \,|\D u|}{1-u^2}
%\, \right]  \rmd\sigma \,
%-\,np\!\!\!
%\int\limits_{\{u=t\}}
%\!\!\!
%u\,|\D u|^{p-2}
%\!\left(\, 1-
%\frac{|\D u|^2}{1-u^2}
%\, \right)  \rmd\sigma 
%\\
\label{eq:derup_D}
&\leq \, 
-\, (p\!-\!1) \,\, t \,\, \Big( \frac{1}{1-t^2} \Big)^{\!\!\frac{n+p-1}{2}}
\!\!\!\!\!
\int\limits_{\{u=t\}}
\!\!\!\!
|\D u|^{p-2} \bigg(\! \frac{n}{p-1} \!\bigg)
\!\left(\! 1-
\frac{|\D u|^2}{1-u^2}
\! \right)  \rmd\sigma 
\,\, \leq \,\, 0\, , \phantom{\qquad\quad\qquad\qquad \,\,\,\,\,}
\end{align}
where $\HHH$ is the mean curvature of the level set $\{u=t\}$ and $\nu=\D u/|\D u|$ is the unit normal to the set $\{u=t\}$.
Moreover, if there exists $t \in (0,1)$ such that $U_{p}'(t) = 0$ for some $p \geq 3$, then the static solution $(M,g_0,u)$ is isometric to the de Sitter solution. 

\smallskip

\item[(iv)] For every $p \geq 3$, we have $U_p'(0) := \lim_{t \to 0^+} U_p'(t)=0$ and, setting $U''_p(0) := \lim_{t \to 0^+}U_p'(t)/t$, 
%we have that for every $p\geq 3$, 
it holds
\begin{align}
\notag
\qquad U_p''(0) \,\, &= \,\,
- \, ({p\!-\!1}) \int\limits_{\pa M}
|\D u|^{p-2}\left[\,\, \frac{ \RRR^{\pa M}\! - 
(n-1)(n-2)}{2} \, + \bigg(\!\frac{n+p-1}{p-1} \!\bigg)\left(1-|\D u|^2\right)
\, \right]  \rmd\sigma \, 
\\
\label{eq:der2up_D}
&\leq \,  
- \, (p\!-\!1)  \int\limits_{\pa M}|\D u|^{p-2}\,\bigg(\! \frac{n}{p-1} \!\bigg) \left(1-|\D u|^2\right)\rmd\sigma
\,\, \leq \,\, 0 \,  ,  \phantom{\qquad\qquad\qquad\qquad\quad \,\,\,}
\end{align}
where $\RRR^{\pa M}$ is the scalar curvature of the metric $g_{\pa M}$ induced by $g_0$ on $\pa M$.
Moreover, if $U_{p}''(0) =0$ for some $p \geq 3$, then the static solution $(M,g_0,u)$ is isometric to the de Sitter solution. 
%\begin{comm} 
%Probabilmente vale per $p \geq 2$ riflettere...
%\end{comm}
\end{itemize}

\end{theorem}
\begin{remark}
\label{rem:uno_D}
Notice that formula~\eqref{eq:derup_D} is well-posed also in the case where $\{u = t\}$ is not a regular level set of $u$. In fact, one has from~\cite{Chr} that $u$ is analytic, hence we can use the results from~\cite{Federer} to conclude that the $(n-1)$-dimensional Hausdorff measure of the level sets of $u$ is finite. Moreover, from~\cite{Krantz} we know that the set ${\rm Crit}(\ffi)=\{x\in M\,:\,\na\ffi(x)=0\}$ contains an open $(n-1)$-submanifold $N$ such that $\mathscr{H}^{n-1}({\rm Crit}(\ffi)\setminus N)=0$. In particular, the unit normal vector field to the level set is well defined $\mathscr{H}^{n-1}$-almost everywhere and so does the
mean curvature $\HHH$. In turn, the integrand in~\eqref{eq:derup_D} is well defined $\mathscr{H}^{n-1}$-almost everywhere. Finally, we observe that where $|\D u| \neq 0$ it holds
\begin{equation*}
|\D u|^{p-1} \HHH \,\, = \,\, |\D u|^{p-2}\, \De u- \, |\D u|^{p-4} \,\DD u (\D u , \D u) \,\, = \,\, - \, u \, |\D u|^{p-2} \,\Ric (\nu, \nu) \, ,
\end{equation*}
where $\nu=\D u/|\D u|$ as usual. It is also clear that $|\D u|^{p-1} \HHH=- \, u \, |\D u|^{p-2} \,\Ric (\nu, \nu)=0$ on the whole $N$ for every $p> 2$.
Since $|\Ric|$ is uniformly bounded on $M$, this shows that the integrand in~\eqref{eq:derup_D} is essentially bounded and thus summable on every level set of $u$, provided $p > 2$. 
%
%To conclude, we notice that also the hypersurface area element $\rmd\sigma$ is a priori well defined only on the regular portion of the level set. However, by the above arguments one can deduce that the density that relates $\rmd\sigma$ to the everywhere defined volume element $\rmd \mathscr{H}^{n-1}$ is well defined and bounded $\mathscr{H}^{n-1}$-almost everywhere on every level set of $u$. Hence, the integral in~\eqref{eq:derup_D} is well defined.
\end{remark}

The analytic and geometric implications of Theorem~\ref{thm:main_D} will be discussed in full details in Section~\ref{sec:conseq}. However, we have decided to collect  the more significant among them in Theorem~\ref{thm:geom_ineq_D} below. Before giving the statement, it is worth noticing that,
combining Theorem~\ref{thm:main_D} with some approximations near the extremal points of $u$, we are able to characterize the set ${\rm MAX}(u)$ and to estimate the behavior of the $U_p(t)$'s as $t$ approaches 1. 

\begin{theorem}
\label{thm:estimate_D}
Let $(M,\go,u)$ be a solution of~\eqref{eq:pb_D} satisfying Assumption~\ref{ass:D}.
The set ${\rm MAX}(u)$ is discrete (and finite) and, for every $p\leq n-1$, it holds
\begin{equation}
\label{eq:limup_D}
\liminf_{t\to 1^-} \,U_p(t)\,\,\geq\,\, |{\rm MAX}(u)|\,|\Sph^{n-1}|\,,
\end{equation}
where $|{\rm MAX}(u)|$ is the cardinality of the set ${\rm MAX}(u)$. 
\end{theorem}

For the detailed proof of this result, we refer the reader to Theorem~\ref{thm:MAX} in the appendix.

\begin{remark}
The above result is false without Assumption~\ref{ass:D}. In fact, we can easily find solutions (that does not satisfy our assumption) such that the set ${\rm MAX}(u)$ is very large. For instance, the set of the maximum points of the de Sitter-Schwarzschild solution~\eqref{eq:deS-Schw} has non-zero $\mathscr{H}^{n-1}$-measure, and the same holds for the maximum points of the Nariai solution~\eqref{eq:cylsol_D}.
\end{remark}

Now we are ready to state the main consequences of Theorem~\ref{thm:main_D} on the geometry of the boundary of $M$.

\begin{theorem}[Geometric Inequalities, case $\Lambda>0$]
\label{thm:geom_ineq_D}
Let $(M,g_0,u)$ be a static solution to problem~\eqref{eq:pb_D} satisfying Normalization~\ref{norm:D} and Assumption~\ref{ass:D}. Then the following properties are satisfied.
\begin{itemize}
\item[(i)] {\rm(Area bound)} The inequality 
\begin{equation}
|{\rm MAX}(u)|\,|\Sph^{n-1}|\,\,\leq\,\,|\pa M|\,,
\end{equation}
holds true. Moreover, the equality is fulfilled if and only if the static solution $(M,\go,u)$ is isometric to the de Sitter solution.
\item[(ii)] {\rm(Willmore-type inequality)} The inequality
\begin{equation}
 { |{\rm MAX}(u)|  \,   |\Sph^{n-1}| }   \, \,\,\leq \,\, \int\limits_{\pa M}    \bigg|  \, \frac{ \RRR^{\pa M}\,-\,n\,(n-3)}{2} \,  \bigg|^{n-1} \, \rmd \sigma
\end{equation}
holds true. Moreover, the equality is fulfilled  if and only if the static solution $(M,g_0,u)$ is isometric to the de Sitter solution.
\item[(iii)] The inequality 
\begin{equation}
{ |{\rm MAX}(u)|  \,   |\Sph^{n-1}| }   \, \,\,\leq \,\, \int\limits_{\pa M}    \frac{ \RRR^{\pa M}}{(n-1)(n-2)} \, \, \rmd \sigma
\end{equation}
holds true. Moreover, the equality is fulfilled if and only if the static solution $(M,g_0,u)$ is isometric to the de Sitter solution. 
\item[(iv)] {\rm(Uniqueness Theorem)} Let $n=3$. If $\pa M$ is connected, then $(M,\go,u)$ is isometric to the de Sitter solution. If $\pa M$ is not connected, then $3\,|{\rm MAX}(u)|<\,\pi_0(\pa M)$, in particular, $\pa M$ must have at least four connected components.
%More generally, let $\pa M=\sqcup_{i=1}^r \Sigma_i$, where $\Sigma_1,\dots,\Sigma_r$ are connected components. Then
%$$
%2\,|{\rm MAX}(u)|\,\,\leq\,\, 
%\sum_{i=1}^r k_i\,\chi(\Sigma_i)
%\,,
%$$
%where $k_i$ is the surface gravity of $\Sigma_i$, that is, the constant value of $|\D u|$ on $\Sigma_i$.
%Moreover, the equality holds if and only if $\pa M$ is connected and $(M,\go,u)$ is isometric to the de Sitter solution.
\end{itemize}
\end{theorem}

We conclude this subsection observing that point (iv) in the above statement can be rephrased by saying that (after normalization) the de Sitter solution has the least possible surface gravity among three-dimensional solutions to problem~\eqref{eq:pb_D} with connected boundary.

\subsection{Setting of the problem and statement of the main results (case $\Lambda<0$).} 
\label{sub:settingandstatement_A}

Suppose that $M$ has empty boundary and at least one end, and consider positive functions $u\in\mathscr{C}^{\infty}(M)$ such that the triple $(M,\go,u)$ satisfies the system

\begin{equation}
\label{eq:pb_A}
\begin{dcases}
u\,\Ric\,=\,\DD u-n\,u\,\go, & \mbox{in } M\\
\ \,\;\De u\,=\,n\, u, & \mbox{in } M\\
\  u(x)\rightarrow +\infty & \mbox{as } x\rightarrow \infty
\end{dcases}
\end{equation}

We notice that the first two equations coincide with the equations of the rescaled problem~\eqref{eq:SES} in the case of a negative cosmological constant.

\begin{normalization}
\label{norm:A}
Since the problem is invariant under a multiplication of $u$ by a positive constant, without loss of generality we will suppose from now on $\min_M(u)=1$. 
We also let 
$$
{\rm MIN}(u)=\{p\in M \, : \, u(p)=1\}
$$ 
be the set of the points that realize the minimum.
\end{normalization}

For future convenience, we introduce the following classical definition, originally introduced by Penrose in~\cite{Penrose} (see also~\cite{Hij_Mon} and the references therein).

\begin{definition}[Conformally compact static solution]
\label{def:CC_A}
A static solution $(M,\go,u)$ of problem~\eqref{eq:pb_A} is said to be {\em conformally compact} if the following conditions are satisfied:

\begin{itemize}
\item[(i)] The manifold $M$ is diffeomorphic to the interior of a compact manifold with boundary $\overline{M}$.

\item[(ii)] There exists a compact $K\subset M$ and a function $r\in\mathscr{C}^{\infty}(\overline{M}\setminus K)$ such that $r\neq 0$ on $M$, $r=0$ on $\pa \overline{M}$, $dr\neq 0$ on $\pa \overline{M}$ and the metric $\bar g=r^2 \go$ extends smoothly to a metric on $\overline{M}\setminus K$. 
\end{itemize}
\end{definition}

In the following, we will call $\pa \overline{M}$ the {\em conformal boundary} of $M$ and, in order to simplify the notation, we will set
$$
\pa M\,:=\,\pa\overline{M}\, .
$$
We will refer to a function with the same properties of $r$ in $(ii)$ as to a {\em defining function} for $\pa M$.

%\begin{lemma}
%\label{le:CC}
%Let $(M,\go,u)$ be a conformally compact static solution. Then $1/u$ is a defining function and
%\begin{equation}
%\frac{|\D u|}{u} \, = \, 1, \quad \mbox{on } \pa M
%\end{equation}
%\end{lemma}
%
%\begin{proof}
%Let $r$ be a defining function of $\pa M$. Proceeding as in {\color{viola} [Graham, volume and area renormalization ...]}, we see that necessarily $|dr|_{r^2\go}=1$ and, in a collar of $\pa M$, the metric assumes the form
%$$
%\go=r^{-2}\Big(dr\otimes dr+\left(1+h_{ij}\right)\tilde g_{ij} d\theta^i\otimes d\theta^j\Big)\, ,  \quad \mbox{where } h_{ij}=o_2(r) \mbox{ as } r\to 0\,.
%$$
%with respect to coordinates $(r,\theta^1,\dots,\theta^{n-1})$, where $h_{ij}=o_2(r)$ as $r\to 0$ and $\tilde g_{ij}$ does not depends on $r$. Now, evaluating the second equation of problem~\eqref{eq:pb_A} at a point $p=(r_p,\theta_p^1,\dots,\theta_p^{n-1})$ we obtain the equation
%$$
%r^2\frac{\pa^2 u}{\pa r^2}-r \left(n-2+\frac{r\, \tilde g^{ij}}{2(1+h_{ij})}\frac{\pa h_{ij}}{\pa r}\right)\frac{\pa u}{\pa r}-nu+r^2\widetilde\De u=0\, ,
%$$
%where $\widetilde\De$ is the Laplace-Beltrami operator of the metric $\tilde g$ on the slice $\{r=r_p\}$. From this equation one finds
%$$
%u=\frac{1}{r}+w\, , \quad \mbox{where } w=o_2(1) \mbox{ as } r\to 0\,.
%$$
%Now a simple computation proves the thesis. {\color{viola} (rivedere conti e se si pu\`o fare di meglio)}
%\end{proof}

We are now ready to introduce the analogous of Assumption~\ref{ass:D} in the case of a negative cosmological constant.

\begin{assumption}
\label{ass:A}
The triple $(M,\go,u)$ is conformally compact, the function $1/\sqrt{u^2-1}$ is a defining function for $\pa M$ and $\lim_{x\to \bar x} \left(u^2-1-|\D u|^2\right)\geq 0$ for every $\bar x\in\pa M$.
\end{assumption}

Some comments are in order to justify these requirements. 
First we notice that, if $1/\sqrt{u^2-1}$ is a defining function, then the limit in Assumption~\ref{ass:A} exists and is finite (see Lemma~\ref{le:tot_geod_BGH_A}-(i) in the appendix).
%This fact is a consequence of the computations made in Section~\ref{sec:conf_reform}, that we briefly recall here. Since the function $1/\sqrt{u^2-1}$ is a defining function, the metric $g=\go/(u^2-1)$ extends to the conformal boundary $\pa M$, thus, in particular, its scalar curvature $\Rg$ is a smooth finite function on $\pa M$. Therefore, one can deduce from formula~\eqref{eq:tilde_R} that the quantity $\lim_{x\to\bar x}u^2(1-|\na\ffi|_g^2)$ is well defined and finite. Here $\ffi$ is the function defined by~\eqref{eq:ffi_A}, and satisfies
%$$
%|\na\ffi|_g^2\,=\,\frac{|\D u|^2}{u^2-1}\,.
%$$
%It is easily checked that $\lim_{x\to \bar x} \left(u^2-1-|\D u|^2\right)=\lim_{x\to\bar x}u^2(1-|\na\ffi|_g^2)$, hence our limit exists and is finite as claimed.

We also remark that Assumption~\ref{ass:A} is not unusual, in the sense that it is a weakening of a more classical hypotesis, appeared in similar forms in various articles like~\cite{Qing,Wang} and~\cite{Wang2}. In these works, it is supposed that $\pa M$ is diffeomorphic to a sphere and $\go$ is conformally compact, with respect to a defining function $r$ such that $r^2{\go}_{|_{\pa M}}$ is conformal to the spherical metric. One can prove that these hypoteses are stronger than our assumption. In fact, it follows from them that $1/\sqrt{u^2-1}$ is a defining function and the function $u^2-1-|\D u|^2$ goes to zero as $x\to\infty$ (see formul\ae~(3.2),~(3.3) in~\cite{Qing}).

Finally, we observe that the anti-de Sitter triple $(\R^n,g_A,u_A)$ defined by~\eqref{eq:A} indeed verifies all our hypotesis, namely it is a conformally compact static solution of problem~\eqref{eq:pb_A} satisfying Normalization~\ref{norm:A} and Assumption~\ref{ass:A}. 

Proceeding in analogy with~\cite{Ago_Maz_2}, for all $p\geq 0$ we introduce the functions $U_p:(1,+\infty)\rightarrow\R$ defined as
\begin{equation}
\label{eq:Up_A}
t \,\, \longmapsto \,\, U_p(t) \, = \, \Big(\frac{1}{t^2-1}\Big)^{\!\!\frac{n+p-1}{2}}\!\!\!\!\!\! \int\limits_{ \{ u = t \}} \!\!\!\!  |\D u|^p \, \rmd \sigma .
\end{equation}

It is worth noticing that the functions $t \mapsto U_p(t)$ are well defined, since the integrands are globally bounded and the level sets of $u$ have finite hypersurface area. In fact, since $u$ is analytic (see~\cite{Chr}), the level sets of $u$ have locally finite $\mathscr{H}^{n-1}$-measure by the results in~\cite{Federer}. Moreover, they are compact and thus their hypersurface area is finite. Another important observation comes from the fact that, using the explicit formul\ae~\eqref{eq:A}, one easily realizes that the quantities
\begin{equation}
\label{eq:Pfunct_confvol_A}
M \, \ni \,x \, \longmapsto \,\,   \frac{|\D u|}{\sqrt{u^2-1}}  \,(x) \quad \hbox{and} \quad  [0, 1) \, \ni \, t \, \longmapsto \,\, U_0(t) \,= \!\!\!\!\int\limits_{\{u = t\}}\!\!
 \Big(\frac{1}{u^2-1}\Big)^{\!\frac{n-1}{2}}
 \rmd\sigma 
\end{equation}
are constant on the anti-de Sitter solution. In the following, via a conformal reformulation of problem~\eqref{eq:pb_A}, we will be able to give a more geometric interpretation of this fact. On the other hand, we notice that the function $t \mapsto U_p(t)$ can be rewritten in terms of the above quantities as
\begin{equation}
\label{eq:Up2_A}
{ U}_p(t)\,\,=\!\!\!
\int\limits_{\{u = t\}}\!\!\!
\left(   \frac{|\D u|}{\sqrt{u^2-1}}  \right)^{p}  \Big(\frac{1}{u^2-1}\Big)^{\!\frac{n-1}{2}}
\,\, \rmd\sigma .
\end{equation}
Hence, thanks to~\eqref{eq:Pfunct_confvol_A}, we have that for every $p \geq 0$ the function $t \mapsto U_p(t)$ is constant on the anti-de Sitter solution. Our main result illustrates how  the functions $t \mapsto U_p(t)$ can be used to detect the rotational symmetry of the {\em static solution} $(M, g_0, u)$. In fact, for $p\geq 3$, they are nondecreasing and the monotonicity is strict unless $(M, g_0, u)$ is isometric to the anti-de Sitter solution.

\begin{theorem}[Monotonicity-Rigidity Theorem, case $\Lambda<0$]
\label{thm:main_A}

Let $(M,g_0,u)$ be a conformally compact static solution to problem~\eqref{eq:pb_A} in the sense of Definition~\ref{def:CC_A}. Suppose moreover that $(M,\go,u)$ satisfies Normalization~\ref{norm:A} and Assumption~\ref{ass:A}. 
Then
\begin{equation}
\label{eq:Dusign_A}
|\D u|^2\,\,\leq\,\,u^2-1\,,
\end{equation}
on the whole manifold $M$.
Moreover, for every $p \geq 1$ let $U_p : (1, +\infty) \rightarrow \R$ be the function defined  in~\eqref{eq:Up_A}. Then, the following properties hold true.
\begin{itemize}
\item[(i)] For every $p\geq 1$, the function $U_p$ is continuous. 

\smallskip

\item[(ii)] The function $U_1$ is monotonically nondecreasing. Moreover, if $U_1(t_1)=U_1(t_2)$ for some $t_1\neq t_2$, then $(M,\go,u)$ is isometric to the anti-de Sitter solution.

\smallskip

\item[(iii)] For every $p \geq 3$, the function $U_p$ is differentiable and the derivative satisfies, for every $t \in (1,+\infty)$,
\begin{align}
\notag
U_p'(t) \,\, &= \,\, (p-1) \,\, t \,\, \Big( \frac{1}{t^2-1} \Big)^{\!\!\frac{n+p-1}{2}} \!\!\!\!\!
\int\limits_{\{u=t\}}
\!\!\!\!
|\D u|^{p-2}
\!\left[\, -\,\bigg| \frac{\D u}{u} \bigg| \, \HHH \, + \, \bigg( \!\frac{np}{p-1}\! \bigg) \, - \, \bigg(\!\frac{n+p-1}{p-1} \!\bigg) \bigg( \frac{|\D u|^2}{u^2-1} \bigg)
\, \right]  \rmd\sigma \,
\\
\notag
&= \,\, (p-1) \,\, t \,\, \Big( \frac{1}{t^2-1} \Big)^{\!\!\frac{n+p-1}{2}} \!\!\!\!\!
\int\limits_{\{u=t\}}
\!\!\!\!
|\D u|^{p-2}
\!\left[\, (n-1) \,+ \, \Ric (\nu, \nu) \, +   \bigg(\!\frac{n+p-1}{p-1} \!\bigg) \bigg(\!1- \frac{|\D u|^2}{u^2-1} \!\bigg)
\, \right]  \rmd\sigma \,
\\
\label{eq:derup_A}
&\geq \,\, 
(p-1) \,\, t \,\, \Big( \frac{1}{t^2-1} \Big)^{\!\!\frac{n+p-1}{2}}
\!\!\!\!\!
\int\limits_{\{u=t\}}
\!\!\!\!
|\D u|^{p-2} \bigg(\! \frac{n}{p-1} \!\bigg)
\!\left(\! 1-
\frac{|\D u|^2}{u^2-1}
\! \right)  \rmd\sigma 
\,\, \geq \,\, 0\, , 
\end{align}
where $\HHH$ is the mean curvature of the level set $\{u=t\}$ and $\nu=\D u/|\D u|$ is the unit normal to the level set $\{u=t\}$.
Moreover, if there exists $t \in (1,+\infty)$ such that $U_{p}'(t) = 0$ for some $p \geq 3$, then the static solution $(M,g_0,u)$ is isometric to the anti-de Sitter solution. 

\smallskip

\item[(iv)] For our next result it is convenient to see $U_p(t)$ as a function of the defining function $r=1/\sqrt{u^2-1}$, that is, we consider the function $V_p(r)=U_p(\sqrt{1+1/r^2})$. We have that, for every $p\geq 3$, it holds
\begin{align}
\notag
\lim_{r \to 0^+} V_p''(r)
\,\, &= \,\,
-\,(p-1)
\!\int\limits_{\pa M}
\!
\bigg[\, \frac{ (n-1)(n-2)\, - \,\RRR_g^{\pa M}}{2\,(n-1)}\,+\,\frac{n\,(p+1)}{2\,(p-1)}\left(u^2-1-|\D u|^2\right)
 \bigg] \, \rmd\sigma_g 
 \\
 \label{eq:der2up_A}
 &\leq\,\,
 -\,(p-1)\int\limits_{\pa M}\bigg(\frac{n}{p-1}\bigg)\left(u^2-1-|\D u|^2\right)\,\rmd\sigma_g
 \,\,\leq \,\,0 \, ,
\end{align}
where $g=\go/(u^2-1)$ and $\RRR_g^{\pa M}$ is the scalar curvature of the metric $g_{\pa M}$ induced by $g$ on $\pa M$. The integrands in~\eqref{eq:der2up_A} have to be thought as the limits of the corresponding functions as $x\to\bar x$, with $\bar x\in\pa M$.
%Moreover, if $ \lim_{t \to +\infty} t^{n+4}\, U_p''(t) =0$ for some $p \geq 3$, then the static solution $(M,g_0,u)$ is isometric to the anti-de Sitter solution.
%\begin{comm} 
%Probabilmente vale per $p \geq 2$ riflettere...
%\end{comm}
\end{itemize}

\end{theorem}
\begin{remark}
\label{rem:uno_A}
Using the same arguments of Remark~\ref{rem:uno_D}, one observes that formula~\eqref{eq:derup_A} is well posed even when the set $\{u=t\}$ is not a regular level set of $u$. Notice that the integrands in~\eqref{eq:der2up_A} are finite functions, as it has been stated in the discussion below Assumption~\ref{ass:A} (see also Lemma~\ref{le:tot_geod_BGH_A}-(i)).
\end{remark}

\begin{remark}
Note that, unlike the case $\Lambda>0$, the rigidity statement does not hold for point {\rm (iii)} of Theorem~\ref{thm:main_A}. The reason for this will be clear later (see the discussion at the end of Subsection~\ref{sub:second derivative}). 

In general, as it will become apparent in Section~\ref{sec:conseq}, the analysis of the static solutions is more delicate in the case $\Lambda<0$. 
In particular, we will see that, in the case $\Lambda<0$, in order to obtain results that are comparable with the ones for $\Lambda>0$, it will be useful to require some extra hypoteses on the behavior of the static solution near the conformal boundary (namely Assumption~\ref{ass:A2} in Subsection~\ref{sub:further_A}). Still, some of the consequences for $\Lambda>0$ will have no analogue in the case $\Lambda<0$ (compare Theorem~\ref{thm:geom_ineq_D} with Theorem~\ref{thm:geom_ineq_A} below).
\end{remark}

%\begin{remark}
%Notice that under the hypothesis of the above theorem, formula~\eqref%{eq:derup_D} implies that the only possible minimal level set is the one where $u$ vanishes, if present. 
%\end{remark}

The analytic and geometric implications of Theorem~\ref{thm:main_A} will be discussed in full details in Section~\ref{sec:conseq}. However, we have decided to collect the more significant among them in Theorem~\ref{thm:geom_ineq_A} below. Before giving the statement, it is worth noticing that, 
combining Theorem~\ref{thm:main_A} with some approximations near the extremal points of the static potential $u$, we are able to characterize the set ${\rm MIN}(u)$ and to estimate the behavior of the $U_p(t)$'s as $t$ approaches 1.

\begin{theorem}
\label{thm:estimate_A}
Let $(M,\go,u)$ be a conformally compact solution of \eqref{eq:pb_A} satisfying Normalization~\ref{norm:A} and Assumption~\ref{ass:A}.
Then the set ${\rm MAX}(u)$ is discrete (and finite) and, for every $p\leq n-1$, it holds
\begin{equation}
\label{eq:limup_A}
\liminf_{t\to 1^+} \,U_p(t)\,\,\geq\,\, |{\rm MIN}(u)|\,|\Sph^{n-1}|\,,
\end{equation}
where $|{\rm MIN}(u)|$ is the cardinality of the set ${\rm MIN}(u)$. 
\end{theorem}

For the detailed proof of this result, we refer the reader to Theorem~\ref{thm:MIN} in the appendix.

\begin{theorem}[Geometric Inequalities, case $\Lambda<0$]
\label{thm:geom_ineq_A}
Let $(M,g_0,u)$ be a conformally compact static solution to problem~\eqref{eq:pb_A} in the sense of Definition~\ref{def:CC_A}. Suppose moreover that $(M,\go,u)$ satisfies Normalization~\ref{norm:A} and Assumption~\ref{ass:A}. Then the metric $g=\go/(u^2-1)$ extends to the conformal boundary and the following properties are satisfied.
\begin{itemize}
\item[(i)] {\rm(Area bound)} The inequality 
\begin{equation}
|{\rm MIN}(u)|\,|\Sph^{n-1}|\,\,\leq\,\,|\pa M|_g\,,
\end{equation}
holds true. Moreover, the equality is fulfilled if and only if the static solution $(M,\go,u)$ is isometric to the anti-de Sitter solution.
\item[(ii)] {\rm(Willmore-type inequality)} Suppose that $\lim_{t\to+\infty} (u^2-1-|\D u|^2)=0$. Then the inequality
\begin{equation}
 { |{\rm MIN}(u)|  \,   |\Sph^{n-1}| }   \, \,\,\leq \,\, \int\limits_{\pa M}    \bigg|  \, \frac{ \Rg^{\pa M}\,-\,(n+1)\,(n-2)}{2(n-2)} \,  \bigg|^{n-1} \, \rmd \sigma_g
\end{equation}
holds true. Moreover, the equality is fulfilled  if and only if the static solution $(M,g_0,u)$ is isometric to the anti-de Sitter solution.
%\item[(iii)] In the same hypotesis of point (ii), the inequality 
%\begin{equation}
%{ |{\rm MIN}(u)|  \,   |\Sph^{n-1}| }   \, \,\,\geq \,\, \int\limits_{\pa M}    \frac{ -\,\Rg^{\pa M}}{(n-1)(n-2)} \, \, \rmd \sigma
%\end{equation}
%holds true. Moreover, the equality is fulfilled if and only if the static solution $(M,g_0,u)$ is isometric to the anti-de Sitter solution. 
%\item[(iv)] {\rm(Uniqueness Theorem)} If $\pa M$ is connected and $n=3$, then $(M,\go,u)$ is isometric to the de Sitter solution.
\end{itemize}
\end{theorem}

We underline the similarity between this result and statements~(i),~(ii) of Theorem~\ref{thm:geom_ineq_D}. Unfortunately, we are not able to provide analogues of points~(iii),~(iv).

\subsection{Strategy of the proof. }
\label{sub:strategy}

To describe the strategy of the proof, we focus our attention on the rigidity statements in Theorems~\ref{thm:main_D}-(iii),~\ref{thm:main_A}-(iii) and for simplicity, we let $p=3$. At the same time, we provide an heuristic for the monotonicity statement.
In this introductory section, we treat the two cases $\Lambda>0$ and $\Lambda<0$ at the same time, in order to emphasize the similarities between them. For a more specific and precise analysis we address the reader to Section~\ref{sec:conf_reform} and following.

The method employed 
%to prove Theorem~\ref{thm:main} and Theorem~\ref{thm:refined} 
is based on the conformal splitting technique introduced in~\cite{Ago_Maz_1}, which consists of two main steps. 
The first step is the construction of the so called {\em cylindrical ansatz} and amounts to find an appropriate conformal deformation $g$ of the {\em static metric} $g_0$ in terms of the {\em static potential} $u$. In the case under consideration, the natural deformation is given by
\begin{align*}
%\label{eq:g_schw}
g \, &= \,  \frac{\go}{1-u^2} \qquad (\mbox{case }\Lambda>0)\,,
\\
g \, &= \,  \frac{\go}{u^2-1} \qquad (\mbox{case }\Lambda<0)\,,
\end{align*}
defined on $M^*:=M\setminus {\rm MAX}(u)$ (respectively $M^*:=M\setminus{\rm MIN}(u)$) if $\Lambda>0$ (respectively $\Lambda<0$). The manifold $M^*$ has the same boundary as $M$ and each point of ${\rm MAX}(u)$ (respectively ${\rm MIN}(u)$) corresponds to an end of $M^*$. 
When $(M,g_0,u)$ is the {\em de Sitter solution} (respectively the {\em anti-de Sitter solution}), the metric $g$ obtained through the above formula is immediately seen to be the cylindrical one. In general, the {\em cylindrical ansatz} leads to a conformal reformulation of  problems~\eqref{eq:pb_D},~\eqref{eq:pb_A} in which the conformally related metric $g$ obeys the quasi-Einstein type equation
\begin{align*}
\Ricg-\left[\frac{1-(n-1)\tanh^2(\ffi)}{\tanh(\ffi)}\right]\nana\ffi+(n-2)d\ffi\otimes d\ffi=\left(n-2+2(1-|\na \ffi|_g^2)\right)g \,,  \qquad \hbox{(case $\Lambda>0$)}
\\
\Ricg-\left[\frac{1-(n-1)\coth^2(\ffi)}{\coth(\ffi)}\right]\nana\ffi+(n-2)d\ffi\otimes d\ffi=\left(n-2+2(1-|\na \ffi|_g^2)\right)g \,,  \qquad \hbox{(case $\Lambda<0$)}
\end{align*}
where $\na$ is the Levi-Civita connection of $g$ and the function $\ffi$ is defined by 
\begin{align*}
\ffi \, =\, \frac{1}{2}\log\left(\frac{1+u}{1-u} \right)\,,\qquad \hbox{(case $\Lambda>0$)}
\\
\ffi \, =\, \frac{1}{2}\log\left(\frac{u+1}{u-1} \right)\,,\qquad \hbox{(case $\Lambda<0$)}
\end{align*}
and satisfies
\begin{align*}
\Deg \, \ffi \, = \, -n\tanh(\ffi) \left(1-|\na\ffi|_g^2\right) \, , \quad \hbox{(case $\Lambda>0$)}
\\
\Deg \, \ffi \, = \, -n\coth(\ffi) \left(1-|\na\ffi|_g^2\right) \, , \quad \hbox{(case $\Lambda>0$)}
\end{align*}
where $\Deg$ is the Laplace-Beltrami operator of the metric $g$.
Before proceeding, it is worth pointing out that taking the trace of the quasi-Einstein type equation gives
\begin{align*}
\frac{\RRR_g}{n-1} \,= \, (n-2) + \left(n\,\tanh^2(\ffi)+2\right)\left(1-|\na \ffi|_g^2\right) \, ,\qquad \hbox{(case $\Lambda>0$)}
\\
\frac{\RRR_g}{n-1} \,= \, (n-2) + \left(n\,\coth^2(\ffi)+2\right)\left(1-|\na \ffi|_g^2\right) \, ,\qquad \hbox{(case $\Lambda<0$)}
\end{align*}
where $\RRR_g$ is the scalar curvature of the conformal metric $g$. On the other hand, it is easy to see that $|\na \ffi|_g^2$ is proportional to the first term in~\eqref{eq:Pfunct_confvol_D} (respectively~\eqref{eq:Pfunct_confvol_A}).
%, namely
%\begin{equation*}
% \Big( \frac{2m}{1-u^2} \Big)^{\!\frac{n-1}{n-2}} \, |\D u| \,(x) \,.
%\end{equation*}
In fact, if $(M,\go)$ is the {\em de Sitter solution} (respectively {\em anti-de Sitter solution}), then $(M^*,g)$ is a round cylinder with constant scalar curvature. Furthermore, the second term appearing in~\eqref{eq:Pfunct_confvol_D} (respectively~\eqref{eq:Pfunct_confvol_A}) is (proportional to) the hypersurface area of the level sets of $\ffi$ computed with respect to the metric induced on them by $g$. Again, in the cylindrical situation such a function is expected to be constant.

The second step of our strategy consists in proving via a splitting principle that the metric $g$ has indeed a product structure, provided the hypotheses of the Rigidity statement are satisfied. More precisely, we use the above conformal reformulation of the original system combined with the Bochner identity to deduce the equation
\begin{multline*}
\Deg |\na\ffi|^2_g\, -\left(\frac{1+(n+1)\tanh^2(\ffi)}{\tanh(\ffi)}\right)\langle \na |\na\ffi|_g^2 \,|\, \na\ffi\rangle_g \,=
\\
=\, 2\, |\nana\ffi|_g^2  \, +\, 2\, n\, \tanh^2(\ffi)\, |\na\ffi|_g^2\, \left(1-|\na \ffi|_g^2\right)\, ,\qquad \hbox{(case $\Lambda>0$)}
\end{multline*}
\begin{multline*}
\Deg |\na\ffi|^2_g\, -\left(\frac{1+(n+1)\coth^2(\ffi)}{\coth(\ffi)}\right)\langle \na |\na\ffi|_g^2 \,|\, \na\ffi\rangle_g\,=
\\
=\, 2\, |\nana\ffi|_g^2  \, +\, 2\, n\, \coth^2(\ffi)\, |\na\ffi|_g^2\, \left(1-|\na \ffi|_g^2\right)\, .\qquad \hbox{(case $\Lambda<0$)}
\end{multline*}
Observing that the drifted Laplacian appearing on the left hand side is formally self-adjoint with respect to the weighted measure 
\begin{align*}
\frac{\rmd \mu_g}{\sinh(\ffi)\cosh^{n+1}(\ffi)} \,,\qquad \hbox{(case $\Lambda>0$)}
\\
\frac{\rmd \mu_g}{\sinh^{n+1}(\ffi)\cosh(\ffi)} \,,\qquad \hbox{(case $\Lambda<0$)}
\end{align*} 
we integrate by parts and we obtain, for every $s\geq 0$, the integral identity
\begin{align*}
\int\limits_{\{\ffi=s\}}\!\!\! \left[\frac{|\na \ffi|_g^2 \,\Hg -\,|\na \ffi|_g\,\Deg\ffi}{\sinh(\ffi)\cosh^{n+1}(\ffi)}\right] \!\rmd\sigma_g =\!\! \int\limits_{\{\ffi>s\}}\!\!\! \left[\frac{|\nana\ffi|^2_g 
\,+ n\, \tanh^2(\ffi)\, |\na\ffi|^2_g \left(1-|\na \ffi|^2_g\right)}{\sinh(\ffi)\cosh^{n+1}(\ffi)}\right]\rmd\mu_g,\ \hbox{($\Lambda>0$)}
\\
\int\limits_{\{\ffi=s\}}\!\!\! \left[\frac{|\na \ffi|_g^2 \, \Hg -\,|\na \ffi|_g\,\Deg\ffi}{\sinh^{n+1}(\ffi)\cosh(\ffi)}\right] \!\rmd\sigma_g =\!\! \int\limits_{\{\ffi>s\}}\!\!\! \left[\frac{|\nana\ffi|^2_g 
\,+ n\, \coth^2(\ffi)\, |\na\ffi|^2_g \left(1-|\na \ffi|^2_g\right)}{\sinh^{n+1}(\ffi)\cosh(\ffi)}\right]\rmd\mu_g,\ \hbox{($\Lambda>0$)}
\end{align*}
where $\Hg$ is the mean curvature of the level set $\{ \ffi = s \}$ inside the ambient $(M^*,g)$ (notice that the same considerations as in Remark~\ref{rem:uno_D} apply here). We then observe that, up to a negative function of $s$, the left hand side is closely related to $U_3'$ (see formul\ae~\eqref{eq:upfip2} and~\eqref{eq:der_fip}). On the other hand, we will prove that, under suitable assumptions, the right hand side is always nonnegative. This will easily imply the Monotonicity statement.
Also, under the hypotheses of the Rigidity statement, the left hand side of the  above identity vanishes and thus the Hessian of $\ffi$ must be zero in an open region of $M$. In turn, by analyticity, it vanishes everywhere. 
%On the other hand, we easily deduce that $\ffi\to +\infty$ along the ends of $M^*$. In particular, $\na \ffi$ is a nontrivial parallel vector field. Hence, it provides a natural splitting direction for the metric $g$. Finally, studying the behavior of $u$ near ${\rm MAX}(u)$ (respectively ${\rm MIN}(u)$), it is easy to realize that the asymptotics of $g$ are the ones of a round cylinder, so that the product structure forces the Riemannian manifold $(M^*,g)$ to be isometric to a round cylinder.
Translating this information back in terms of the hessian of $u$, we are able to conclude using Obata's theorem.

\subsection{Summary.}
The paper is organized as follows. In Section~\ref{sec:conseq} we describe the geometric consequences of Theorems~\ref{thm:main_D},~\ref{thm:main_A}, obtaining several sharp inequalities for which the equality is satisfied if and only if the solution to system~\eqref{eq:pb_D} or~\eqref{eq:pb_A}  is rotationally symmetric. We distinguish the consequences of Theorems~\ref{thm:main_D}-(iii),~\ref{thm:main_A}-(iii) on the geometry of a generic level set of $u$ (see Subsections~\ref{sub:rot_sym_D} and~\ref{sub:rot_sym_A}), from the consequences of Theorems~\ref{thm:main_D}-(iv),~\ref{thm:main_A}-(iv) on the geometry of the boundary of $M$ (see Subsections~\ref{sub:further_D} and~\ref{sub:further_A}). 

In Subsection~\ref{sub:further_D}, we deduce some sharp inequalities for static solutions of problem~\eqref{eq:pb_D} and we use them to obtain some corollaries on the uniqueness of the de Sitter metric (see Theorem~\ref{thm:ul_bounds_bound_D} and the discussion below). 
In particular, we show that, if Assumption~\ref{ass:D} holds, then the only 3-dimensional static solution of problem~\eqref{eq:pb_D} with a connected boundary is the de Sitter solution. 
%Such a general result is not available in dimensions greater than 3. Indeed, in~\cite{Gib_Har_Pop} non-trivial static solutions are provided for $4\leq n\leq 8$. Nevertheless, we discuss some geometric conditions under which the uniqueness statement holds in every dimension.
For $n\geq 4$, we are not able to prove such a general result.  Nevertheless, we discuss some geometric conditions under which the uniqueness statement holds in every dimension.
The analogous consequences in Subsection~\ref{sub:further_A} are less strong. In any case, we are still able to state a result (Theorem~\ref{thm:mon_glob_A}) that extends the classical Uniqueness Theorems of the anti-de Sitter metric proved in~\cite{Bou_Gib_Hor,Qing,Wang}.

In Section~\ref{sec:conf_reform}, we reformulate problem~\eqref{eq:pb_D} and~\eqref{eq:pb_A} in terms of a quasi-Einstein type metric $g$ and a function $\ffi$ satisfying system~\eqref{eq:pb_conf} ({\em cylindrical ansatz}), according to the strategy described in Subsection~\ref{sub:strategy}. In this new framework, both Theorem~\ref{thm:main_D} and Theorem~\ref{thm:main_A} results to be equivalent to Theorem~\ref{thm:main_conf} in Subsection~\ref{sub:reform} below, as we will prove in detail in Section~\ref{sec:translation}. Theorem~\ref{thm:main_conf} will be proven in Section~\ref{sec:proofs} with the help of the integral identities proved in Section~\ref{sec:integral}.

Finally, in Appendix~\ref{sec:app_BGH} we discuss a different approach to the study of problems~\eqref{eq:pb_D} and~\eqref{eq:pb_A}, that does not rely on the machinery of Sections~\ref{sec:conf_reform}-\ref{sec:proofs} and provides some consequences that are comparable with the ones discussed in Section~\ref{sec:conseq}. In the case $\Lambda>0$, the results that we show in this section are known (see~\cite{Bou_Gib_Hor,Chr}), but in the case $\Lambda<0$ they appear to be new.

\section{Consequences}
\label{sec:conseq}

In this section we discuss some consequences of Theorems~\ref{thm:main_D} and~\ref{thm:main_A}, distinguishing the two cases $\Lambda>0$ and $\Lambda<0$.
 
\subsection{Consequences on a generic level set of $u$ (case $\Lambda>0$).}
\label{sub:rot_sym_D}
Since, as already observed, the functions $t \mapsto U_p(t)$ defined in~\eqref{eq:Up_D} are constant on the de Sitter solution, we obtain, as an immediate consequence of Theorem~\ref{thm:main_D} and formula~\eqref{eq:derup_D}, the following characterizations of the rotationally symmetric solutions to  system~\eqref{eq:pb_D}.

\begin{theorem}
%\label{thm:main2_geom_D}
Let $(M,\go,u)$ be a solution to problem~\eqref{eq:pb_D} satisfying Normalization~\ref{norm:D} and Assumption~\ref{ass:D}. Then, for every $p \geq 3$ and every $t \in [0, 1)$, it holds 
\begin{equation*}
%\label{eq:int_ineq_D}
\int\limits_{\{u=t\}}
\!\!\!\!
|\D u|^{p-2}
\!\left[\, (n-1) \, - \, \Ric (\nu, \nu) \, +    \bigg(\!1- \frac{|\D u|^2}{1-u^2} \!\bigg)
\, \right]  \rmd\sigma  \,\, \geq \,\, 0 \, .
%t\!\!\int\limits_{\{u=t\}} \!\!\!\frac{|\D u|^p}{1-t^2}\, \rmd \sigma \,\,\, \leq  \!\int\limits_{\{u=t\}} \!\!\! |\D u|^{p-1} \Big(\HHH+\frac{nu}{|\D u|}\Big) \, \rmd \sigma 
\end{equation*}
%holds true, where $\HHH$ is the mean curvature of the level set $\{u = t \}$. 
Moreover, the equality is fulfilled for some $p\geq 3$ and some $t \in [0,1)$ if and only if the static solution $(M,\go,u)$ is isometric to the de Sitter solution.
\end{theorem}

Setting $t=0$ in the above formula, and using the Gauss-Codazzi equation, one gets 
\begin{equation*}
\label{eq:int_ineq_D}
\int\limits_{\{u=t\}}
\!\!\!\!
|\D u|^{p-2}
\!\left[\, \frac{ \RRR^{\pa M}\! - 
(n-1)(n-2)}{2}  \, +    \bigg(\!1- \frac{|\D u|^2}{1-u^2} \!\bigg)
\, \right]  \rmd\sigma  \,\, \geq \,\, 0 \, ,
%t\!\!\int\limits_{\{u=t\}} \!\!\!\frac{|\D u|^p}{1-t^2}\, \rmd \sigma \,\,\, \leq  \!\int\limits_{\{u=t\}} \!\!\! |\D u|^{p-1} \Big(\HHH+\frac{nu}{|\D u|}\Big) \, \rmd \sigma 
\end{equation*}
This inequality is just a rewriting of formula~\eqref{eq:der2up_A}, whose consequences will be discussed in Subsection~\ref{sub:further_D} (see Theorem~\ref{thm:rig_gen_D} and below). Another way to rewrite formula~\eqref{eq:derup_D} is the following.

\begin{theorem}
\label{thm:main2_geom_D}
Let $(M,\go,u)$ be a solution to problem~\eqref{eq:pb_D} satisfying Normalization~\ref{norm:D} and Assumption~\ref{ass:D}. Then, for every $p \geq 3$ and every $t \in [0, 1)$, the inequality 
\begin{equation}
\label{eq:int_ineq_D}
\int\limits_{\{u=t\}} \!\!\!\! \bigg( \! \frac{|\D u|}{\sqrt{1-u^2}} \! \bigg)^{\!p} \rmd \sigma \,\,\, \leq  \!\int\limits_{\{u=t\}} \!\!\! \bigg( \! \frac{|\D u|}{\sqrt{1-u^2}} \! \bigg)^{\!p-2} \Big[ \, \HHH \, |\D\log u| \, +   {n }  \, \Big] \, \rmd \sigma 
\end{equation}
holds true, where $\HHH$ is the mean curvature of the level set $\{u = t \}$. Moreover, the equality is fulfilled for some $p\geq 3$ and some $t \in (0,1)$ if and only if the static solution $(M,\go,u)$ is isometric to the de Sitter solution.
\end{theorem}

To give an interpretation of Theorem~\ref{thm:main2_geom_D} in the framework of overdetermined boundary value problems, we observe that the equality is achieved in~\eqref{eq:int_ineq_D} as soon as the integrands on the left and right hand side coincide $\mathscr{H}^{n-1}$-almost everywhere on some level set of $u$. This amounts to ask that 
\begin{equation*}
\frac{u\,|\D u|^2}{1-u^2}\,\, = \,\, \HHH\,|\D u|\,+\,n\,u \, .
\end{equation*}
On the other hand, with the help of the first equation in~\eqref{eq:pb_D}, one can rewrite the right hand side as 
\begin{equation*}
\HHH\,|\D u|\,+\,n\,u \, = \,  - u \, \Ric\, (\nu, \nu)\,+\,n\,u  \, = \, - \D^2 u \,(\nu, \nu) \, = \, - \frac{\pa |\D u|}{ \pa \nu} \, ,
\end{equation*}
where $\nu = \D u / |\D u|$. This easily implies the following corollary.
\begin{corollary}
\label{cor:over_D}
Let $(M,\go,u)$ be a solution to problem~\eqref{eq:pb_D} satisfying Normalization~\ref{norm:D} and Assumption~\ref{ass:D}. Assume in addition that the identity
\begin{equation}
\label{eq:overd_D}
\frac{\pa (1/|\D u|)}{\pa \nu} \, = \, \frac{t}{1-t^2}
\end{equation}
holds $\mathscr{H}^{n-1}$-almost everywhere on some level set $\{ u = t\}$, with $t \in (0,1)$. Then, the static solution $(M,\go,u)$ is isometric to the de Sitter solution.
\end{corollary}
In other words, assumption~\eqref{eq:overd_D} in the previous corollary can be seen as a condition that makes system~\eqref{eq:pb_D} overdetermined and forces the solution to be rotationally symmetric. Observe that~\eqref{eq:overd_D} is satisfied on the de Sitter solution and thus it is also a necessary condition for $(M,\go,u)$ being rotationally symmetric.

To illustrate other implications of Theorem~\ref{thm:main2_geom_D}, let us observe 
that, applying H\"older inequality to the right hand side of~\eqref{eq:int_ineq_D} with conjugate exponents $p/(p-2)$ and $p/2$, one gets
\begin{equation*}
\int\limits_{\{u=t\}} \!\!\! \bigg( \! \frac{|\D u|}{\sqrt{1-u^2}} \! \bigg)^{\!p-2} \Big[ \, \HHH \, |\D\log u|\, +  \, {n }  \, \Big] \, \rmd \sigma 
 \,\, \leq \,\, \Bigg(  \int\limits_{\{u=t\}} \!\!\! 
\bigg( \! \frac{|\D u|}{\sqrt{1-u^2}} \! \bigg)^{\!p}  \, \rmd \sigma \Bigg)^{\!\!\frac{p-2}{p}} 
\Bigg(  \int\limits_{\{u=t\}} \!\!\! \Big| \,  \, \HHH \, |\D\log u| \, +   \,{n } \,\Big|^{\frac{p}{2}}   \rmd \sigma \Bigg)^{\!\!\frac{2}{p}} \,.
\end{equation*}                             
This implies on every level set of $u$ the following sharp $L^p$-bound for the gradient of the {\em static potential}.
\begin{corollary}
\label{cor:Lp_D}
Let $(M,g_0,u)$ be a solution to problem~\eqref{eq:pb_D} satisfying Normalization~\ref{norm:D} and Assumption~\ref{ass:D}. Then, for every $p \geq 3$ and every $t \in [0, 1)$ the inequality
\begin{equation}
\label{eq:Lp_ineq_D}
  \left\|  \frac{\D u}{\sqrt{1-u^2}} \right\|_{L^p(\{ u=t \})}   \,\,\,\leq \,\,\,\, \sqrt{ \,\,\, \bigg\|  \,\HHH \, |\D\log u|\,+\,{n}\, \bigg\|_{L^{p/2}(\{ u=t \})} }  , 
\end{equation}
holds true, where $\HHH$ is the mean curvature of the level set $\{u = t \}$. Moreover, the equality is fulfilled for some $p \geq 3$ and some $t \in (0,1)$ if and only if the static solution $(M,g_0,u)$ is isometric to the de Sitter solution.
\end{corollary}

It is worth pointing out that the right hand side in~\eqref{eq:Lp_ineq_D} may possibly be unbounded. However, for regular level sets of the {\em static potential} the $L^p$-norm of the mean curvature is well defined and finite (see Remark~\ref{rem:uno_D}). We also observe that letting $p \to + \infty$, we deduce, under the same hypothesis of Corollary~\ref{cor:Lp_D}, the following $L^\infty$-bound
\begin{equation}
\label{eq:Linf_ineq_D}
  \left\|  \frac{\D u}{\sqrt{1-u^2}} \right\|_{L^{\infty}(\{ u=t \})}   \,\,\,\leq \,\,\,\, \sqrt{ \,\,\, \Big\|  \,\HHH \, |\D\log u|\,+\,{n}\, \Big\|_{L^{\infty}(\{ u=t \})} } ,
\end{equation}
for every $t \in [0, 1)$. Unfortunately, in this case we do not know whether the rigidity statement holds true or not. However, the equality is satisfied on the de Sitter solution and this makes the inequality sharp.

Now we will combine inequality~\eqref{eq:Lp_ineq_D} in Corollary~\ref{cor:Lp_D} with the observation that for every $t \in [0,1)$ and every $3\leq p \leq n-1$ (we need to take $n\geq 4$) it holds
\begin{equation}
\label{eq:mon_limup_D}
U_p(t) \, \geq \, |{\rm MAX} (u)|\, |\Sph^{n-1}| \,,  
\end{equation}
where the latter estimate follows from estimate~\eqref{eq:limup_D} in Theorem~\ref{thm:estimate_D} and from the monotonicity of the $U_p$'s stated in Theorem~\ref{thm:main_D}-(iii).
Recalling the explicit expression~\eqref{eq:Up_D} of the $U_p$'s, we obtain the following.

\begin{theorem}
\label{thm:ul_bounds_D}
Let $n\geq 4$. Let $(M,g_0,u)$ be a solution to problem~\eqref{eq:pb_D} satisfying Normalization~\ref{norm:D} and Assumption~\ref{ass:D}. Then, for every $3\leq p \leq n-1$ and every $t \in [0, 1)$, the inequalities 
\begin{equation}
\label{eq:pi_gen_D}
(1-t^2)^{\frac{n-1}{2}} \, |{\rm MAX}(u)|\,|\Sph^{n-1}|\,\, \leq \,\  \Big\| \, \HHH\, |\D\log u| \,+\, n \, \Big\|_{L^{p/2}(\{ u=t \})}^{\frac{p}{2}} \, 
\end{equation}
hold true.
Moreover, the equality is fulfilled for some $t \in(0,1)$ and some $3\leq p \leq n-1$, if and only if the static solution $(M,g_0,u)$ is isometric to the de Sitter solution.
\end{theorem}

To give a geometric interpretation of the above theorem, we recall the identity
\begin{equation*}
\HHH\, |\D\log u|+{n} \, = \, n- \Ric(\nu, \nu) \, ,
\end{equation*}
and we observe that the quantity $({1-t^2})^{\!\frac{n-1}{2}} \,|\Sph^{n-1}|$ corresponds to the hypersurface area of the level set $\{u_D = t\}$ in the de Sitter solution~\eqref{eq:D}. Combining together these two facts, we arrive at the following corollary.
\begin{corollary}
\label{cor:geom_bound_D}
Let $n\geq 4$. Let $(M,g_0,u)$ be a solution to problem~\eqref{eq:pb_D} satisfying Normalization~\ref{norm:D} and Assumption~\ref{ass:D}. Then, for every $3\leq p \leq n-1$ and every $t \in [0, 1)$, the inequality
\begin{equation}
 \frac{ |{\rm MAX}(u)|  \,   \left| \{ u_D= t\} \right|}{\left| \{ u = t\} \right|}   \, \leq  \fint\limits_{\{ u= t\}}   \!\!\! \big|  \, n - \Ric(\nu, \nu) \,  \big|^{\frac{p}{2}} \, \rmd \sigma
\end{equation}
holds true. Moreover, the equality is fulfilled for some $t \in(0,1)$ and some $3\leq p\leq n-1$, if and only if the static solution $(M,g_0,u)$ is isometric to the de Sitter solution. In particular, for every $t \in (0,1)$, it holds
\begin{equation}
1 \,\, \leq \,\, \big\|  \, n - \Ric(\nu, \nu) \,  \big\|_{L^\infty(\{u=t\})} \, .
\end{equation}
\end{corollary}

\subsection{The geometry of $\pa M$ (case $\Lambda>0$). }
\label{sub:further_D}

We pass now to describe some consequences of the behaviour of the {\em static solution} $(M,\go,u)$ at the boundary $\pa M$, as prescribed by Theorem~\ref{thm:main_D}-(iv). We remark that $|\D u|$ is constant on every connected component of $\pa M$, and that $\pa M$ is a totally geodesic hypersurface inside $(M,\go)$. In particular, also the mean curvature $\HHH$ vanishes at $\pa M$. Hence, formula~\eqref{eq:derup_D} implies that
%\begin{comm}
%for every $p \geq 3$,
%\end{comm}
$U_p'(0) = 0$. 

The following theorem is a rephrasing of formula~\eqref{eq:der2up_D} and is the analog of Theorem~\ref{thm:main2_geom_D}.
\begin{theorem}
\label{thm:rig_gen_D}
Let $(M,\go,u)$ be a solution to problem~\eqref{eq:pb_D} satisfying Normalization~\ref{norm:D} and Assumption~\ref{ass:D}.
Then, for every $p \geq 3$, it holds
\begin{equation}
\label{eq:cond_1_D}
\int\limits_{\pa M}
\!\!
|\D u|^{p-2}
\!\left[\, 
(n-1)(n-2) - \RRR^{\pa M} \, -   2 \big(1- {|\D u|^2} \big)
\, \right]  \rmd\sigma  \,\, \leq \,\, 0 \, ,
\end{equation}
where $\RRR^{\pa M}$ denotes the scalar curvature of the metric induced by $\go$ on $\pa M$. Moreover, the equality holds for some $p \geq 3$ if and only if $(M,\go,u)$ is isometric to the de Sitter solution. In particular, the boundary of $M$ has only one connected component and it is isometric to a $(n-1)$-dimensional sphere.
\end{theorem}

Since the quantity $|\D u|$ is constant on each connected component of $\pa M$ (because $\DD u=0$ on $\pa M$, as it follows from the first equation in problem~\eqref{eq:pb_D}), if we assume that the boundary is connected, formula~\eqref{eq:cond_1_D} can be replaced by 
\begin{equation}
\label{eq:cond_12_D}
\int\limits_{\pa M}
\!
\left[\,(n-1)(n-2) \,-\RRR^{\pa M}  \, - 2  \big(1- {|\D u|^2} \big)
\, \right]  \rmd\sigma  \,\, \leq \,\, 0 \, .
\end{equation}

%\begin{corollary} 
%Let $(M,\go,u)$ be a solution to problem~\eqref{eq:pb_D} satisfying Normalization~\ref{norm:D} and Assumptions~\ref{ass:AE_D},~\ref{ass:D}.
%If $\pa M$ is connected, then it holds
%\begin{equation}
%\label{eq:cond_12_D}
%\int\limits_{\pa M}
%\left[\, \frac{ \RRR^{\pa M}\! - 
%(n-1)(n-2)}{2}  \, +    \big(1- {|\D u|^2} \big)
%\, \right]  \rmd\sigma  \, \geq \,\, 0 \, ,
%\end{equation}
%where $\RRR^{\pa M}$ denotes the scalar curvature of the metric induced by $g_0$ on $\pa M$. Moreover, the equality holds if and only if $(M,\go,u)$ is isometric to the de Sitter solution.
%\end{corollary}

\begin{remark}
\label{rmk:BGH_ref_D}
In the case $p=3$, Theorem~\ref{thm:rig_gen_D} is a weaker version of Corollary~\ref{cor:bound_BGH_D} in the appendix. This corollary is not new, but it has been proved in~\cite{Chr} generalizing some early computations in~\cite{Bou_Gib_Hor} and~\cite{Lindblom}.
In particular, in the case of a connected boundary, from Corollary~\ref{cor:bound_BGH_D} it follows the inequality
\begin{equation}
\label{eq:cond_2_D}
\int\limits_{\pa M} \!\left[(n-1)(n-2)\,-\,\RRR^{\pa M}\right]\,  \rmd\sigma \,\, \leq \,\, 0\,,
\end{equation}
that is strictly better than our formula~\eqref{eq:cond_12_D}, and is proved without the need of Assumption~\ref{ass:D}. 
Note that from inequality~\eqref{eq:cond_2_D} it follows the remarkable result that the only static solution of~\eqref{eq:pb_D} whose boundary is isometric to a sphere with its standard metric, is the de Sitter solution. This is not a direct consequence of our Theorem~\ref{thm:rig_gen_D}.
\end{remark}

%In the framework of overdetermined boundary value problems, we have the following corollary, which should be compared with Corollary~\ref{cor:over_D}.
%\begin{corollary}
%\label{cor:over2_D}
%Let $(M,\go,u)$ be a solution to problem~\eqref{eq:pb_D} satisfying Normalization~\ref{norm:D} and Assumptions~\ref{ass:AE_D},~\ref{ass:D}. Assume in addition that the identity
%\begin{equation}
%\label{eq:over_D}
%1\,-\,\frac{2\,\left(1\,-\,|\D u|^2\right)}{(n-1)(n-2)} \,= \,\frac{\quad\RRR^{\pa M}}{(n-1)(n-2)}	
%\end{equation}
%holds $\mathscr{H}^{n-1}$-almost everywhere on $\pa M$. Then, the static solution $(M,g_0,u)$ is isometric to the de Sitter solution. In particular, the boundary of $M$ has only one connected component and it is isometric to a $(n-1)$-dimensional sphere.
%\end{corollary}

To illustrate some other consequences of Theorem~\ref{thm:rig_gen_D}, we rewrite formula~\eqref{eq:cond_1_D} as
\begin{equation}
\label{eq:cond_1_rew_D}
\int\limits_{\pa M}|\D u|^p \, \rmd\sigma \,\, \leq \,\int\limits_{\pa M}\! |\D u|^{p-2}\bigg[\, \frac{\RRR^{\pa M}-n\,(n-3)}{2} \, \bigg]\rmd\sigma\,.
\end{equation}
Then we apply H\"older inequality to the right hand side with conjugate exponents $p/(p-2)$ and $p/2$, obtaining
\begin{equation*}
\int\limits_{\pa M} \!|\D u|^{p-2} \bigg[\, \frac{\RRR^{\pa M}-n\,(n-3)}{2} \, \bigg] \,  \rmd \sigma  \,\, \leq \,\, \Bigg(  \int\limits_{\pa M} \! |\D u|^{p}  \, \rmd \sigma \Bigg)^{\!\!{(p-2)}/{p}} 
\Bigg(  \int\limits_{\pa M} \! 
\bigg|\, \frac{\RRR^{\pa M}-n\,(n-3)}{2} \, \bigg|^{p/2}  \, \rmd \sigma \Bigg)^{\!\!{2}/{p}} \,.
\end{equation*}
This immediately implies the following corollary, that should be compared with Corollary~\ref{cor:Lp_D}.
\begin{corollary}
\label{cor:R_Lp_D}
Let $(M,\go,u)$ be a solution to problem~\eqref{eq:pb_D} satisfying Normalization~\ref{norm:D} and Assumption~\ref{ass:D}.
Then, for every $p \geq 3$, the inequality 
\begin{equation}
\label{eq:R_Lp_D}
\left|\left|    \, \D u \,  \right|\right|_{L^p(\pa M)} \,\,  \leq \,\, \sqrt{ \,\,\, \left|\left| \, \frac{\RRR^{\pa M}\,-\,n\,(n-3)}{2} \, \right|\right|^{\phantom{1}}_{L^{p/2}(\pa M)} }
\end{equation}
holds true, where $\RRR^{\pa M}$ denotes the scalar curvature of the metric induced by $g_0$ on $\pa M$. Moreover, the equality holds for some $p \geq 3$ if and only if $(M,g_0,u)$ is isometric to the de Sitter solution. In particular, the boundary of $M$ has only one connected component and it is isometric to a $(n-1)$-dimensional sphere.
\end{corollary}

Letting $p \to + \infty$ in formula~\eqref{eq:R_Lp_D}, we obtain, under the hypotheses of the above corollary, the $L^{\infty}$-bound
\begin{equation}
\label{eq:R_Linf_D}
\left|\left|   \, \D u \,  \right|\right|_{L^{\infty}(\pa M)}  \!\! \leq \,\,\,\, \sqrt{ \,\, \left|\left| \, \frac{ \RRR^{\pa M}\,-\,n\,(n-3) }{2} \, \right|\right|^{\phantom{1}}_{L^{\infty}(\pa M)} }\,.
\end{equation}

For our next result, we are going to combine the monotonicity of the $U_p$'s, as stated by Theorem~\ref{thm:main_D}, together with the estimate~\eqref{eq:limup_D} given in Theorem~\ref{thm:estimate_D}.

\begin{theorem}
\label{thm:mon_glob_D}
Let $(M,\go,u)$ be a solution to problem~\eqref{eq:pb_D} satisfying Normalization~\ref{norm:D} and Assumption~\ref{ass:D}.
Then, it holds
\begin{equation}
\label{eq:mon_glob_D}
|{\rm MAX}(u)|\,|\Sph^{n-1}|\,\,\leq\,\,\int\limits_{\pa M}|\D u|^p\rmd\sigma\,\,\leq\,\,|\pa M|\,.
\end{equation}
for $0\leq p \leq 1$ if $n=3$ and for $0\leq p \leq n-1$ if $n\geq 4$.
Moreover, the equality $|{\rm MAX}(u)|\,|\Sph^{n-1}|\,=\,|\pa M|$ holds if and only if $(M,\go,u)$ is isometric to the de Sitter solution.
\end{theorem}

\begin{proof}
First, recalling Assumption~\ref{ass:D}, it is clear that 
$$
U_p(0)\,=\!\int\limits_{\pa M}\!\!|\D u|^p\,\rmd\sigma\,\leq\,|\pa M|\,,
$$
for every $p\geq 0$. 

Now consider the case $n\geq 4$ and let $3\leq p\leq n-1$. From formula~\eqref{eq:limup_D} and Theorem~\ref{thm:main_D}-(iii), we obtain
$$
|{\rm MAX}(u)|\,\, |\Sph^{n-1}| \,\,\leq\,\,U_p(0)\,,
$$
and the equality holds if and only if $U_p(t)$ is constant, that is, if and only if $(M,\go,u)$ is isometric to the de Sitter solution. Combining this with the inequality above, we obtain the thesis for $3\leq p\leq n-1$. If $0\leq p\leq 3$ instead, to conclude it is enough to observe that $\int_{\pa M}|\D u|^p\rmd\sigma\geq \int_{\pa M}|\D u|^3\rmd\sigma$, thanks to Assumption~\ref{ass:D}.

In the case $n=3$, we can repeat the argument above using $U_1(t)$, that we know to be monotonic thanks to Theorem~\ref{thm:main_D}-(ii).
\end{proof}

\begin{remark}
The result above is particularly effective in dimension $n=3$. In that case, it is known from~\cite{Bou_Gib_Hor} that any solution $(M,\go,u)$ of problem~\eqref{eq:pb_D} with a connected boundary satisfies $|\pa M|\,\leq\, 4\pi$. Since formula~\eqref{eq:mon_glob_D} gives the opposite inequality, we  conclude that the only $3$-dimensional static solution to problem~\eqref{eq:pb_D} with $\pa M$ connected and satisfying Normalization~\ref{norm:D} and Assumption~\ref{ass:D} is the de Sitter solution. A direct proof of this fact will be given later (see Theorem~\ref{thm:uniqueness_n3_D}).
Note that the same thesis does not hold without Assumption~\ref{ass:D}. An explicit example of a non-trivial $3$-dimensional static solution with a connected boundary diffeomorphic to $\Sph^2$ (which does not satisfy Assumption~\ref{ass:D}) can be constructed via a quotient of the Nariai solution~\eqref{eq:cylsol_D} (see~\cite[Section~7]{Ambrozio}).

In the case $n\geq 4$ we are not able to provide such a general result, and the situation seems much wilder. For instance, for any $4\leq n\leq 8$, one can prove the existence of a countable family of non-trivial static solutions of~\eqref{eq:pb_D} with $\pa M$ connected and diffeomorphic to a sphere or to a product of spheres (see~\cite{Gib_Har_Pop}). However, looking at the numerical approximations of some of these solutions, it appears that they do not satisfy our hypoteses, thus the question of the uniqueness of the de Sitter solution under our assumptions seems still open.
\end{remark}

Using Corollary~\ref{cor:R_Lp_D} in place of Corollary~\ref{cor:Lp_D} we obtain the following analog of Corollary~\ref{cor:geom_bound_D}.

\begin{theorem}
\label{thm:ul_bounds_bound_D}
Let $(M,\go,u)$ be a solution to problem~\eqref{eq:pb_D} satisfying Normalization~\ref{norm:D} and Assumption~\ref{ass:D}. Then, the following statements hold true.
\begin{itemize}
\item[(i)]
For every $p \geq 2$, the inequality
\begin{equation}
\label{eq:pi_gen_bound_D}
 \frac{ |{\rm MAX}(u)|  \,   |\Sph^{n-1}| }{ |\pa M|}   \, \,\,\leq \,\, \fint\limits_{\pa M}    \bigg|  \, \frac{ \RRR^{\pa M}\,-\,n\,(n-3)}{2} \,  \bigg|^{\frac{p}{2}} \, \rmd \sigma
\end{equation}
holds true. Moreover, the equality is fulfilled for  some $p \geq 2$, if and only if the static solution $(M,g_0,u)$ is isometric to the de Sitter solution. In particular, it holds
\begin{equation}
2 \,\, \leq \,\, \big\|  \, \RRR^{\pa M}\,-\,n\,(n-3) \,  \big\|_{L^\infty(\{u=t\})} \, .
\end{equation}
\item[(ii)] The inequality 
\begin{equation}
\label{eq:piip_D}
\frac{ |{\rm MAX}(u)|  \,   |\Sph^{n-1}| }{ |\pa M|}   \, \,\,\leq \,\, \fint\limits_{\pa M}    \frac{ \RRR^{\pa M}}{(n-1)(n-2)} \, \, \rmd \sigma
\end{equation}
holds true. Moreover, the equality is fulfilled if and only if the static solution $(M,g_0,u)$ is isometric to the de Sitter solution. 

\smallskip
%
%\item[(iii)] Under the same assumptions as in {\rm (ii)}, if the scalar curvature $\RRR^{\pa M}$ of the boundary satisfies the bound
%\begin{equation}
%\label{eq:scal_invrad_D}
%\frac{\!\RRR^{\pa M}}{\,\,\RRR^{\Sph^{n-1}}} 
%\,\, \leq \,\, 
%|{\rm MAX}(u)|\,\frac{|\Sph^{n-1}|}{|\pa M| }\, ,
%\end{equation}
%%where $\RRR^{\Sph^{n-1}}\!\! = (n-1)(n-2)$ denotes the scalar curvature of the standard $(n-1)$-dimensional sphere, 
%then $(M,g_0,u)$ is isometric to the de Sitter solution. 
\end{itemize}
\end{theorem}

\begin{proof}
For $n\geq 4$ and $3\leq p\leq n-1$, statements (i) and (ii) can be derived from inequality~\eqref{eq:R_Lp_D} in Corollary~\ref{cor:R_Lp_D} and formula~\eqref{eq:mon_glob_D} in Theorem~\ref{thm:mon_glob_D}. 
In general, we need to use inequality~\eqref{eq:bound_BGH_D} in Corollary~\ref{cor:bound_BGH_D}, proved in the appendix.

To prove statement (i), we rewrite formula~\eqref{eq:bound_BGH_D} as
\begin{equation}
\label{eq:cond_2_rew_D}
\int\limits_{\pa M} |\D u|\,\rmd\sigma\,\leq\,\int\limits_{\pa M} |\D u|\bigg[\frac{\RRR^{\pa M}-n(n-3)}{2}\bigg]\rmd\sigma\,.
\end{equation}
Compare this inequality with~\eqref{eq:cond_1_rew_D}, which holds for every $p\geq 3$ but is weaker than~\eqref{eq:cond_2_rew_D} in the case $p=3$.

Using H\"older inequality, we have
$$
\int\limits_{\pa M} |\D u|\bigg[\frac{\RRR^{\pa M}-n(n-3)}{2}\bigg]\rmd\sigma
\,\leq\,
\bigg[\,\int\limits_{\pa M} \bigg|\frac{\RRR^{\pa M}-n(n-3)}{2}\bigg|^{\frac{p}{2}}\rmd\sigma\,\bigg]^{\frac{2}{p}}\,\bigg(\,\int\limits_{\pa M} |\D u|^{\frac{p}{p-2}}\rmd\sigma\bigg)^{\frac{p-2}{p}}\,.
$$
Moreover, since $|\D u|\leq 1$ on $\pa M$ thanks to Assumption~\ref{ass:D}, we have $|\D u|^{\frac{p}{p-2}}\leq |\D u|$ for every $p\geq 2$. Substituting in~\eqref{eq:cond_2_rew_D}, with some easy computations we find
$$
 \frac{\|\D u\|_{L^1(\pa M)}}{ |\pa M|}   \, \,\,\leq \,\, \fint\limits_{\pa M}    \bigg|  \, \frac{ \RRR^{\pa M}\,-\,n\,(n-3)}{2} \,  \bigg|^{\frac{p}{2}} \, \rmd \sigma
$$
Now using inequality~\eqref{eq:mon_glob_D} we obtain~\eqref{eq:pi_gen_bound_D}. 
Moreover, if the equality holds in~\eqref{eq:pi_gen_bound_D}, then also~\eqref{eq:cond_2_rew_D} is an equality, thus by Corollary~\ref{cor:bound_BGH_D} we have that $(M,\go,u)$ is the de Sitter solution.

Statement (ii), is an immediate consequence of formula~\eqref{eq:cond_2_rew_D} and inequality~\eqref{eq:mon_glob_D} in Theorem~\ref{thm:mon_glob_D}.
%observe that, when $\pa M$ is contained in a level set of $|\D u|$, say $\{|\D u|=c\}$, formula~\eqref{eq:cond_1_D} becomes
%$$
%\left[\,n\,(n-3)\,+\,2\,c^2\,\right]\,|\pa M|\,\leq\,\int\limits_{\pa M} \RRR^{\pa M}\rmd\sigma.
%$$
%On the other hand, from~\eqref{eq:glob_mon_D} we deduce, for every $p\geq 3$,
%$$
%|{\rm MAX}(u)|\, |\Sph^{n-1}| \,\,\leq\,\,c^p\,|\pa M|\,.
%$$
%Combining the two inequalities together we obtain
%$$
%\left[\,\frac{n(n-3)}{c^2}\,+\,2\,\right]|{\rm MAX}(u)|\,|\Sph^{n-1}|\,\leq\,c^{p-2}\int\limits_{\pa M} \RRR^{\pa M}\rmd\sigma\,,
%$$
%and we conclude using the fact that $c\leq 1$, as it is clear from Assumption~\ref{ass:D}.
%Statement (iii) follows immediately from (ii), since inequality~\eqref{eq:scal_invrad_D} implies that the equality is satisfied in~\eqref{eq:piip_D}.
\end{proof}

If we set $p=2(n-1)$ in Theorem~\ref{thm:ul_bounds_bound_D}-(i), we obtain the following nicer statement.

\begin{corollary}[Willmore-type inequality]
\label{cor:will_D}
Let $(M,\go,u)$ be a static solution to problem~\eqref{eq:pb_D}, satisfying Normalization~\ref{norm:D} and Assumption~\ref{ass:D}. Then, it holds
\begin{equation*}
\left(|{\rm MAX}(u)|\,\, |\Sph^{n-1}|\right)^{\frac{1}{n-1}} 
\,\,\leq\,\,\
\left\| \,\frac{ \RRR^{\pa M}\,-\,n\,(n-3)}{2} \,\right\|^{\phantom{1}}_{L^{n-1}(\pa M)} \,,
\end{equation*}
where $\RRR^{\pa M}$ denotes the scalar curvature of the metric induced by $\go$ on $\pa M$. Moreover, the equality holds if and only if $(M,\go,u)$ is isometric to the de Sitter solution. In particular, the boundary of $M$ has only one connected component and it is isometric to a $(n-1)$-dimensional sphere.
\end{corollary}

The result above should be compared with~\cite[Theorem~2.11-(ii)]{Ago_Maz_2} where a similar inequality is provided for the Schwarzschild metric. 

For our next result we restrict to dimension $n=3$, and we use the Gauss-Bonnet Formula to prove that, in the hypotesis of a connected boundary, the equality is achieved in formula~\eqref{eq:piip_D}.

\begin{theorem}[Uniqueness Theorem]
\label{thm:uniqueness_n3_D}
Let $(M,\go,u)$ be a $3$-dimensional static solution to problem~\eqref{eq:pb_D}, satisfying Normalization~\ref{norm:D} and Assumption~\ref{ass:D}. If $\pa M$ is connected, then $(M,\go,u)$ is isometric to the de Sitter solution.
More generally, let $\pa M=\sqcup_{i=1}^r\Sigma_i$, where $\Sigma_1,\dots,\Sigma_r$ are connected surfaces. Then
\begin{equation}
\label{eq:uniqueness_n3_D}
2\,|{\rm MAX}(u)|\,\,\leq\,\, 
\sum_{i=1}^r k_i\,\chi(\Sigma_i)\,,
\end{equation}
where $k_i$ is the surface gravity of $\Sigma_i$, that is, the constant value of $|\D u|$ on $\Sigma_i$.
Moreover, the equality holds if and only if $\pa M$ is connected and $(M,\go,u)$ is isometric to the de Sitter solution.
\end{theorem}

\begin{proof}
Again, it is useful to use Corollary~\ref{cor:bound_BGH_D}, proved in the appendix. Setting $n=3$ in formula~\eqref{eq:bound_BGH_D}, we obtain
$$
2\,\|\D u\|_{L^1(\pa M)}
\,\,\leq\,\,
\int\limits_{\pa M} |\D u|\,\RRR^{\pa M}\,\rmd\sigma
\,,
$$
where the equality holds if and only if $(M,\go,u)$ is isometric to the de Sitter solution.
Recalling formula~\eqref{eq:mon_glob_D}, we obtain
$$
8\pi\,|{\rm MAX}(u)|\,\,\leq\,\,
\sum_{i=1}^r\,k_i\int\limits_{\Sigma_i} \RRR^{\Sigma_i}\,\rmd\sigma
\,.
$$
The thesis is now a consequence of the equalities
$$
\int\limits_{\Sigma_i} \RRR^{\Sigma_i}\,\rmd\sigma\,\,=\,\,
4\pi\,\chi(\Sigma_i)\,,\quad \hbox{for all } i=1,\dots,r\,,
$$
which follow from the Gauss-Bonnet theorem.
\end{proof}

Combining the theorem above with the results in~\cite{Ambrozio}, we obtain the following strenghtening of formula~\eqref{eq:uniqueness_n3_D}.

\begin{corollary}
Let $(M,\go,u)$ be a $3$-dimensional static solution to problem~\eqref{eq:pb_D}, satisfying Normalization~\ref{norm:D} and Assumption~\ref{ass:D}. If $(M,\go,u)$ is not isometric to the de Sitter solution, then 
$$
3\,|{\rm MAX}(u)|\,<\,\sum_{i=1}^r\, k_i\,\leq\,\pi_0(\pa M)\,,
$$
where $k_1,\dots,k_r$ are the surface gravities of the connected components $\Sigma_1,\dots,\Sigma_r$ of $\pa M$.
In particular, a non-trivial $3$-dimensional static solution satisfying Normalization~\ref{norm:D} and Assumption~\ref{ass:D}, must have a boundary with at least four connected components.
\end{corollary}

\begin{proof}
Let $(\tilde M,\tilde \go)\xrightarrow{\pi}(M,\go)$ be the universal covering. Clearly the triple $(\tilde M,\tilde \go,\tilde u=u\circ\pi)$ is still a solution of problem~\eqref{eq:pb_D} and satisfies Assumption~\ref{ass:D} and Normalization~\ref{norm:D}.
From~\cite[Theorem~B]{Ambrozio}, we know that $(\tilde M ,\tilde \go)$ is compact. In particular, the degree $d$ of the covering $\pi$ is a finite number and $|{\rm MAX}(\tilde u)|\,=\,d\,|{\rm MAX}(u)|$. 

Let $\pa \tilde M=\sqcup_{i=1}^s\tilde\Sigma_i$, where $\tilde\Sigma_1,\dots,\tilde\Sigma_s$ are connected.
From~\cite[Theorem~C]{Ambrozio}, 
we have that $(\tilde M,\tilde\go,\tilde u)$ is isometric to the de Sitter triple or 
\begin{equation}
\label{eq:surf_grav_SC}
\sum_{i=1}^s \tilde k_i\,|\tilde{\Sigma}_i|\,<\,\frac{4\pi}{3}\,\sum_{i=1}^s\tilde k_i\,,
\end{equation}
where $\tilde k_i$ is the surface gravity of $\tilde\Sigma_i$ for all $i=1,\dots,s$.

If $(\tilde M,\tilde\go,\tilde u)$ is isometric to the de Sitter triple, then $\pa\tilde M$ is connected, hence also $\pa M$ is connected. Recalling Theorem~\ref{thm:uniqueness_n3_D}, we deduce that $(M,\go,u)$ is isometric to the de Sitter solution, against our hypoteses. 

Therefore, formula~\eqref{eq:surf_grav_SC} must hold. Recalling Theorem~\ref{eq:mon_glob_D}, we obtain the following chain of inequalities
$$
4\pi\,d\,|{\rm MAX}(u)|\,=\,
4\pi\,|{\rm MAX}(\tilde u)|\,\leq\,
\int\limits_{\pa\tilde M}|\tilde\D\tilde u|\,\rmd\tilde{\sigma}\,=\,
\sum_{i=1}^s \tilde k_i\,|\tilde{\Sigma}_i|\,<\,
\frac{4\pi}{3}\,\sum_{i=1}^s\tilde k_i\,.
%\,\leq\,
%\frac{4\pi}{3}\,\pi_0(\pa \tilde M)\,\leq\,
%\frac{4\pi}{3}\,d\,\pi_0(\pa M)\,.
$$ 
Since each connected component of $\pa M$ lifts to at most $d$ connected components of $\pa\tilde M$, we have $$
\sum_{i=1}^s\tilde k_i\,\leq\, d\, \sum_{i=1}^r k_i\,.
$$
This proves the first part of the statement. The inequality $\sum_{i=1}^r k_i\leq\pi_0(\pa M)$ is a consequence of Assumption~\ref{ass:D}.
\end{proof}

%Theorem~\ref{thm:uniqueness_n3_D} can be further elaborated in the case where we have some additional topological information on $M$.
%For example, if $M$ is oriented, from~\cite[Theorem~B]{Ambrozio} we have that $\pa M$ is diffeomorphic to a collection of spheres, hence equation~\eqref{eq:uniqueness_n3_D} is equivalent to $\sum_{i=1}^r k_i\,\geq\,|{\rm MAX}(u)|\,$, where the $k_i$'s are again the surface gravities of the surfaces $\Sigma_i\subseteq\pa M$.

%If we add the hypotesis that the manifold $M$ is simply connected then, combining \cite[Theorem~C]{Ambrozio} with formula~\eqref{eq:mon_glob_D}, we obtain that, if $(M,\go,u)$ satisfies our assumptions but is not the de Sitter solution, then it must hold $\sum_{i=1}^r k_i\,>\,3\,|{\rm MAX}(u)|$. In particular, a non-trivial simply connected $3$-dimensional static solution satisfying Normalization~\ref{norm:D} and Assumption~\ref{ass:D}, must have a boundary with at least four connected components.

\subsection{Consequences on a generic level set of $u$ (case $\Lambda<0$).}
\label{sub:rot_sym_A}

Now we start to discuss the consequences in the case of a negative cosmological constant.
Since, as already observed, the functions $t \mapsto U_p(t)$ defined in~\eqref{eq:Up_A} are constant on the anti-de Sitter solution, we obtain from Theorem~\ref{thm:main_A} and formula~\eqref{eq:derup_A}, the following characterizations of the rotationally symmetric solutions to  system~\eqref{eq:pb_A}.

\begin{theorem}
%\label{thm:main2_geom_A}
Let $(M,\go,u)$ be a solution to problem~\eqref{eq:pb_A} satisfying Normalization~\ref{norm:A} and Assumption~\ref{ass:A}. Then, for every $p \geq 3$ and every $t \in (1, +\infty)$, it holds 
\begin{equation*}
%\label{eq:int_ineq_A}
\int\limits_{\{u=t\}}
\!\!\!\!
|\D u|^{p-2}
\!\left[\, (n-1) \, + \, \Ric (\nu, \nu) \, +    \bigg(\!1- \frac{|\D u|^2}{1-u^2} \!\bigg)
\, \right]  \rmd\sigma  \,\, \geq \,\, 0 \, ,
%t\!\!\int\limits_{\{u=t\}} \!\!\!\frac{|\D u|^p}{1-t^2}\, \rmd \sigma \,\,\, \leq  \!\int\limits_{\{u=t\}} \!\!\! |\D u|^{p-1} \Big(\HHH+\frac{nu}{|\D u|}\Big) \, \rmd \sigma 
\end{equation*}
%holds true, where $\HHH$ is the mean curvature of the level set $\{u = t \}$. 
where $\nu=\D u/|\D u|$.
Moreover, the equality is fulfilled for some $p\geq 3$ and some $t \in (1,+\infty)$ if and only if the static solution $(M,\go,u)$ is isometric to the anti-de Sitter solution.
\end{theorem}

\begin{theorem}
\label{thm:main2_geom_A}
Let $(M,\go,u)$ be a solution to problem~\eqref{eq:pb_A} satisfying Normalization~\ref{norm:A} and Assumption~\ref{ass:A}. Then, for every $p \geq 3$ and every $t \in (1,+\infty)$, the inequality 
\begin{equation}
\label{eq:int_ineq_A}
\int\limits_{\{u=t\}} \!\!\!\! \bigg( \! \frac{|\D u|}{\sqrt{u^2-1}} \! \bigg)^{\!p} \rmd \sigma \,\,\, \leq  \!\int\limits_{\{u=t\}} \!\!\! \bigg( \! \frac{|\D u|}{\sqrt{u^2-1}} \! \bigg)^{\!p-2} \Big[ \,- \HHH \, |\D\log u| \, +   {n }  \, \Big] \, \rmd \sigma
\end{equation}
holds true, where $\HHH$ is the mean curvature of the level set $\{u = t \}$. Moreover, the equality is fulfilled for some $p\geq 3$ and some $t \in (1,+\infty)$ if and only if the static solution $(M,\go,u)$ is isometric to the anti-de Sitter solution.
\end{theorem}

To give an interpretation of Theorem~\ref{thm:main_A} in the framework of overdetermined boundary value problems, we observe that the equality is achieved in~\eqref{eq:int_ineq_A} as soon as the integrands on the left and right hand side coincide $\mathscr{H}^{n-1}$-almost everywhere on some level set of $u$. 
This amounts to ask that 
\begin{equation*}
\frac{u\,|\D u|^2}{u^2-1}\,\, = \,\, -\HHH\,|\D u|\,+\,n\,u \, .
\end{equation*}
On the other hand, with the help of the first equation in~\eqref{eq:pb_A}, one can rewrite the right hand side as 
\begin{equation*}
-\HHH\,|\D u|\,+\,n\,u \, = \,  u \, \Ric\, (\nu, \nu)\,+\,n\,u  \, = \, \D^2 u \,(\nu, \nu) \, = \, \frac{\pa |\D u|}{ \pa \nu} \, ,
\end{equation*}
where $\nu = \D u / |\D u|$. This easily implies the following corollary (compare it with Corollary~\ref{cor:over_D}).
\begin{corollary}
\label{cor:over_A}
Let $(M,\go,u)$ be a solution to problem~\eqref{eq:pb_A} satisfying Normalization~\ref{norm:A} and Assumption~\ref{ass:A}. Assume in addition that the identity
\begin{equation}
\label{eq:overd_A}
\frac{\pa (1/|\D u|)}{\pa \nu} \, = \, -\,\frac{t}{t^2-1}
\end{equation}
holds $\mathscr{H}^{n-1}$-almost everywhere on some level set $\{ u = t\}$, with $t \in (1,+\infty)$. Then, the static solution $(M,\go,u)$ is isometric to the anti-de Sitter solution.
\end{corollary}

In other words, assumption~\eqref{eq:overd_A} in the previous corollary can be seen as a condition that makes system~\eqref{eq:pb_A} overdetermined and forces the solution to be rotationally symmetric. Observe that~\eqref{eq:overd_A} is always satisfied on the anti-de Sitter solution and thus it is also a necessary condition for $(M,\go,u)$ being rotationally symmetric.

To illustrate other implications of Theorem~\ref{thm:main2_geom_A}, let us observe 
that, applying H\"older inequality to the right hand side of~\eqref{eq:int_ineq_A} with conjugate exponents $p/(p-2)$ and $p/2$, one gets
\begin{equation*}
\int\limits_{\{u=t\}} \!\!\! \bigg( \! \frac{|\D u|}{\sqrt{u^2-1}} \! \bigg)^{\!p-2} \Big[ \, -\HHH \, |\D\log u|\, +  \, {n }  \, \Big] \, \rmd \sigma 
\, \leq \, \Bigg(  \int\limits_{\{u=t\}} \!\!\! 
\bigg( \! \frac{|\D u|}{\sqrt{u^2-1}} \! \bigg)^{\!p}  \, \rmd \sigma \Bigg)^{\!\!\frac{p-2}{p}} 
\Bigg(  \int\limits_{\{u=t\}} \!\!\! \Big| \, - \HHH \, |\D\log u| \, +   \,{n } \,\Big|^{\frac{p}{2}}   \rmd \sigma \Bigg)^{\!\!\frac{2}{p}} \,.
\end{equation*}                             
This implies on every level set of $u$ the following sharp $L^p$-bound for the gradient of the {\em static potential} in terms of the $L^p$-norm of the mean curvature of the level set.
\begin{corollary}
\label{cor:Lp_A}
Let $(M,g_0,u)$ be a solution to problem~\eqref{eq:pb_A} satisfying Normalization~\ref{norm:A} and Assumption~\ref{ass:A}. Then, for every $p \geq 3$ and every $t \in (1,+\infty)$ the inequality
\begin{equation}
\label{eq:Lp_ineq_A}
  \bigg\| \, \frac{\D u}{\sqrt{u^2-1}} \,\bigg\|_{L^p(\{ u=t \})}   \,\,\,\leq \,\,\,\,\sqrt{\,\,\,\Big\|  \,-\,\HHH\,|\D \log u|\, + \,n\,\Big\|_{L^{p/2}(\{ u=t \})}} , 
\end{equation}
holds true, where $\HHH$ is the mean curvature of the level set $\{u = t \}$. Moreover, the equality is fulfilled for some $p \geq 3$ and some $t \in (1,+\infty)$ if and only if the static solution $(M,g_0,u)$ is isometric to the anti-de Sitter solution.
\end{corollary}
It is worth pointing out that the right hand side in~\eqref{eq:Lp_ineq_A} may possibly be unbounded. However, for regular level sets of the {\em static potential} the $L^p$-norm of the mean curvature is well defined and finite (see Remark~\ref{rem:uno_A}). We also observe that letting $p \to + \infty$, we deduce, under the same hypothesis of Corollary~\ref{cor:Lp_A}, the following $L^\infty$-bound
\begin{equation}
\label{eq:Linf_ineq_A}
  \bigg\| \, \frac{\D u}{\sqrt{u^2-1}} \,\bigg\|_{L^{\infty}(\{ u=t \})}   \,\,\,\leq \,\,\,\,\sqrt{\,\,\,\Big\|  \,-\HHH\,|\D \log u|\, + \,n\,\Big\|_{L^{\infty}(\{ u=t \})}} ,
\end{equation}
for every $t \in ( 1,+\infty)$. Unfortunately, in this case we do not know whether the rigidity statement holds true or not. However, the equality is satisfied on the anti-de Sitter solution and this makes the inequality sharp.

Now we will combine inequality~\eqref{eq:Lp_ineq_A} in Corollary~\ref{cor:Lp_A} with the observation that for every $t \in (1,+\infty)$ and every $3\leq p \leq n-1$ (we need to take $n\geq 4$) it holds
\begin{equation}
\label{eq:mon_limup_A}
U_p(t) \, \geq \, |{\rm MIN} (u)|\, |\Sph^{n-1}| \,,  
\end{equation}
where the latter estimate follows from estimate~\eqref{eq:limup_A} in Theorem~\ref{thm:estimate_A} and the monotonicity of the $U_p$'s stated in Theorem~\ref{thm:main_A}-(iii).
Recalling the explicit expression~\eqref{eq:Up_A} of the $U_p$'s, we obtain the following analogue of Theorem~\ref{thm:ul_bounds_D}.

\begin{theorem}
\label{thm:ul_bounds_A}
Let $n\geq 4$. Let $(M,g_0,u)$ be a solution to problem~\eqref{eq:pb_A} satisfying Normalization~\ref{norm:A} and Assumption~\ref{ass:A}. Then, for every $3\leq p\leq n-1$ and every $t \in (1,+\infty)$, the inequalities 
\begin{equation}
\label{eq:pi_gen_A}
(t^2-1)^{\frac{n-1}{2}}|{\rm MIN}(u)|\,|\Sph^{n-1}|\,\, \leq \,\, \,  \Big\| \,-\HHH\,|\D \log u|\, + \,n\,  \Big\|_{L^{p/2}(\{ u=t \})}^{\frac{p}{2}} \, 
\end{equation}
hold true.
Moreover, the equality is fulfilled for some $t \in(1,+\infty)$ and some $3\leq p\leq n-1$, if and only if the static solution $(M,g_0,u)$ is isometric to the anti-de Sitter solution.
\end{theorem}

To give a geometric interpretation of the above theorem, we recall the identity
\begin{equation*}
-\HHH\, |\D\log u|+{n} \, = \, n + \Ric(\nu, \nu) \, ,
\end{equation*}
and we observe that the quantity $({t^2-1})^{\!\frac{n-1}{2}} \,|\Sph^{n-1}|$ corresponds to the hypersurface area of the level set $\{u_A = t\}$ in the anti-de Sitter solution~\eqref{eq:A}. Combining together these two facts, we arrive at the following corollary, that should be compared with Corollary~\ref{cor:geom_bound_D}.
\begin{corollary}
Let $n\geq 4$. Let $(M,g_0,u)$ be a solution to problem~\eqref{eq:pb_A} satisfying Normalization~\ref{norm:A} and Assumption~\ref{ass:A}. Then, for every $p \geq 3$ and every $t \in (1, +\infty)$, the inequality
\begin{equation}
 \frac{ |{\rm MIN}(u)|  \,   \left| \{ u_A= t\} \right|}{\left| \{ u = t\} \right|}   \, \leq  \fint\limits_{\{ u= t\}}   \!\!\! \big|  \, n + \Ric(\nu, \nu) \,  \big|^{\frac{p}{2}} \, \rmd \sigma
\end{equation}
holds true. Moreover, the equality is fulfilled for some $t \in(1,+\infty)$ and some $p \geq 3$, if and only if the static solution $(M,g_0,u)$ is isometric to the anti-de Sitter solution. In particular, for every $t \in (1,+\infty)$, it holds
\begin{equation}
1 \,\, \leq \,\, \big\|  \, n + \Ric(\nu, \nu) \,  \big\|_{L^\infty(\{u=t\})} \, .
\end{equation}
\end{corollary}

\subsection{The geometry of $\pa M$ (case $\Lambda<0$).}
\label{sub:further_A}

We pass now to describe some consequences of the behaviour of the {\em static solution} $(M,\go,u)$ at the conformal boundary $\pa M$, as prescribed by Theorem~\ref{thm:main_A}-(iv). We remark that the conformal boundary of $M$ is a totally geodesic hypersurface inside $(\overline{M},g)$ (see Lemma~\ref{le:tot_geod_BGH_A}-(ii) in the appendix). 

%This will be clear later, as a consequence of the first equation of system~\eqref{eq:pb_conf} and equality~\eqref{eq:formula_fundamentalform}. 

The following theorem is a rephrasing of formula~\eqref{eq:der2up_A} and is the analogue of Theorem~\ref{thm:rig_gen_D}.
\begin{theorem}
\label{thm:rig_gen_A}
Let $(M,\go,u)$ be a solution to problem~\eqref{eq:pb_A} satisfying Normalization~\ref{norm:A} and Assumption~\ref{ass:A}, and let $g=g_0/(u^2-1)$.
Then it holds
\begin{equation}
\label{eq:cond_1_A}
\int\limits_{\pa M} \! \left[(n-1)(n-2)-\Rg^{\pa M}+n(n-1)\left(u^2-1-|\D u|^2\right)\right]\,  \rmd\sigma_g \,\, \geq \,\, 0 \, ,
\end{equation}
where $\RRR_g^{\pa M}$ denotes the scalar curvature of the metric induced by $g$ on $\pa M$. 
\end{theorem}

Note that inequality~\eqref{eq:cond_1_A} is sharp, but the rigidity statement does not hold for Theorem~\ref{thm:rig_gen_A}. 
%In particular, this means that we cannot achieve a result in the framework of overdetermined boundary value problems analogous to Corollary~\ref{cor:over2_D}. 
Moreover, unlike Theorem~\ref{thm:rig_gen_D}, formula~\eqref{eq:cond_1_A} does not depend on $p$ and we are not able to find an analogue of Corollary~\ref{cor:R_Lp_D} for the case $\Lambda<0$.

We can still provide the following result, that should be compared with Theorem~\ref{thm:mon_glob_D}.

\begin{theorem}
\label{thm:mon_glob_A}
Let $(M,\go,u)$ be a static solution to problem~\eqref{eq:pb_A}, satisfying Normalization~\ref{norm:A} and Assumption~\ref{ass:A}. Let $|{\rm MIN}(u)|$ be the cardinality of the set ${\rm MIN}(u)$ of the points where $u$ attains its minimum and let $g=g_0/(u^2-1)$. Then
\begin{equation}
\label{eq:mon_glob_A}
|{\rm MIN}(u)|\,\, |\Sph^{n-1}| 
\,\,\leq\,\,
|\pa M|_g\,,
\end{equation}
and the equality holds if and only if $(M,\go,u)$ is isometric to the anti-de Sitter solution. 
\end{theorem}

\begin{proof}
From formula~\eqref{eq:limup_A} and Theorem~\ref{thm:main_A}-(ii) we obtain
$$
|{\rm MIN}(u)|\,\, |\Sph^{n-1}| \,\leq\,\lim_{t\to+\infty} U_1(t)\,,
$$
and the equality holds if and only if $U_1(t)$ is constant, that is, if and only if $(M,\go,u)$ is isometric to the anti-de Sitter solution.
On the other hand
$$
\lim_{t\to+\infty} U_p(t)\,=\,\lim_{t\to+\infty} \!\int\limits_{\{u=t\}}\frac{|\D u|\,\,\,}{(u^2-1)^{\frac{n}{2}}}\,\rmd\sigma\,=\,\lim_{t\to+\infty} \int\limits_{\{u=t\}}\!\sqrt{\frac{|\D u|^2}{u^2-1}}\rmd\sigma_g\,\leq\,|\pa M|_g\,,
$$
where the last inequality follows from Assumption~\ref{ass:A}.
\end{proof}

An immediate corollary of Theorem~\ref{thm:mon_glob_A} above is the following uniqueness result.

\begin{corollary}
\label{cor:uniqueness_A}
Let $(M,\go,u)$ be a conformally compact static solution to problem~\eqref{eq:pb_A}. If the conformal boundary is isometric to the sphere $(\Sph^{n-1},g_{\Sph^{n-1}})$, then $(M,\go,u)$ is isometric to the anti-de Sitter solution. 
\end{corollary}

\begin{proof}
As already commented below Assumption~\ref{ass:A}, it is known (see~\cite{Qing}) that, in the case in which the conformal boundary is a sphere, then Assumption~\ref{ass:A} is automatically satisfied. 
Therefore Theorem~\ref{thm:mon_glob_A} is in charge and we have the thesis.
\end{proof}

The result above extends the uniqueness theorems in~\cite{Qing} and~\cite{Wang}, where the same thesis is obtained for $n\leq 7$ or $M$ spin.

In order to have a clearer exposition, and to highlight the analogies between the results in this section and the ones in Subsection~\ref{sub:further_D}, for the rest of this section we will assume the following stronger version of Assumption~\ref{ass:A}. 

\begin{customass}{2-bis}
\label{ass:A2}
The triple $(M,\go,u)$ is conformally compact, the function $1/\sqrt{u^2-1}$ is a defining function for $\pa M$ and $\lim_{x\to \bar x} \left(u^2-1-|\D u|^2\right)= 0$ for every $\bar x\in\pa M$.
\end{customass}

First, we observe that with this additional hypotesis, formula~\eqref{eq:mon_glob_A} in Theorem~\ref{thm:rig_gen_A} becomes
\begin{equation}
\label{eq:cond_1_bis_A}
\int\limits_{\pa M} \! \left[(n-1)(n-2)-\RRR_g^{\pa M}\right]\,  \rmd\sigma_g \,\, \geq \,\, 0 \, ,
\end{equation}

%The same inequality can be derived in another way (see Corollary~\ref{cor:bound_BGH_A} in the appendix).
Now we use formula~\eqref{eq:cond_1_bis_A} to prove the analogue of Theorem~\ref{thm:ul_bounds_bound_D}-(i). 

\begin{theorem}
\label{thm:ul_bounds_bound_A}
Let $(M,\go,u)$ be a solution to problem~\eqref{eq:pb_A} satisfying Normalization~\ref{norm:A} and Assumption~\ref{ass:A2}, and let $g=g_0/(u^2-1)$.
Then for every $p\geq 2$ it holds
\begin{equation}
\frac{|{\rm MIN}(u)|\,|\Sph^{n-1}|}{|\pa M|_g} \,\, \leq \,\, \fint\limits_{\pa M}  \bigg|\,\frac{\RRR_g^{\pa M}\,-\,(n+1)(n-2)}{2\,(n-2)}\,\bigg|^{\frac{p}{2}} \,\, \rmd\sigma_g \, ,
\end{equation}
where $\RRR_g^{\pa M}$ denotes the scalar curvature of the metric induced by $g$ on $\pa M$. 
\end{theorem}

\begin{proof}
First, we rearrange formula~\eqref{eq:cond_1_bis_A} in the following way
$$
|\pa M|_g\,\,\leq\,\,
\int\limits_{\pa M}  \bigg[\,\frac{-\,\RRR_g^{\pa M}\,+\,(n+1)(n-2)}{2(n-2)}\,\bigg] \,\, \rmd\sigma_g\,.
$$ 
Now we rewrite the right hand side of the above formula, using Jensen Inequality. We obtain
$$
|\pa M|_g\,\,\leq\,\,|\pa M|_g^{\frac{p-2}{p}}
\Bigg[\,\,\int\limits_{\pa M}  \bigg|\,\frac{\,\RRR_g^{\pa M}\,-\,(n+1)(n-2)}{2(n-2)}\,\bigg|^{\frac{p}{2}} \,\, \rmd\sigma_g\,\Bigg]^{\frac{2}{p}}\,,
$$ 
that may be rewritten as
$$
|\pa M|_g\,\,\leq\,\,\int\limits_{\pa M}  \bigg|\,\frac{\,\RRR_g^{\pa M}\,-\,(n+1)(n-2)}{2(n-2)}\,\bigg|^{\frac{p}{2}} \,\, \rmd\sigma_g\,.
$$ 
Now the thesis is an immediate consequence of Theorem~\ref{thm:mon_glob_A}.
\end{proof}

Finally, setting $p=2(n-1)$ in Theorem~\ref{thm:ul_bounds_bound_A} above, we obtain the analogue of Corollary~\ref{cor:will_D}.

\begin{corollary}[Willmore-type inequality]
\label{cor:will_A}
Let $(M,\go,u)$ be a solution to problem~\eqref{eq:pb_A} satisfying Normalization~\ref{norm:A} and Assumption~\ref{ass:A2}, and let $g=g_0/(u^2-1)$.
Then it holds
\begin{equation}
\left(|{\rm MIN}(u)|\,|\Sph^{n-1}|\right)^{\frac{1}{n-1}} \,\, \leq \,\,  \bigg\|\,\frac{\RRR_g^{\pa M}\,-\,(n+1)(n-2)}{2\,(n-2)}\,\bigg\|_{L^{n-1}(\pa M)}  \, ,
\end{equation}
where $\RRR_g^{\pa M}$ denotes the scalar curvature of the metric induced by $g$ on $\pa M$. 
\end{corollary}

\section{A conformally equivalent formulation of the problem}
\label{sec:conf_reform}
The aim of this section is to reformulate system~\eqref{eq:pb_D} and  system~\eqref{eq:pb_A} in a conformally equivalent setting.

\subsection{A conformal change of metric (case $\Lambda>0$). }

First of all, we notice that if $(M,g_0, u)$ is a solution of problem~\eqref{eq:pb_D} and satisfies Normalization~\ref{norm:D}, then one has that $1-u^2 > 0$ everywhere in $M^*=M\setminus {\rm MAX}(u)$.

Motivated by the explicit formul\ae~\eqref{eq:D} of the de Sitter solution, we are led to consider the following conformal change of metric
\begin{equation}\label{eq:g_D}
g\,=\,\frac{g_0}{1-u^2}\,.
\end{equation}
on the manifold $M^*$. It is immediately seen that when $u$ and $g_0$ are as in~\eqref{eq:D} then $g$ is a cylindrical metric. Hence, we will refer to the conformal change~\eqref{eq:g_D} as to a {\em cylindrical ansatz}.

Our next task is to reformulate problem~\eqref{eq:pb_D} in terms of $g$. To this aim we fix local coordinates $\{y^{\alpha}\}_{\alpha=1}^n$ in $M^*$
and using standard formul\ae\ for conformal changes of metrics, 
we deduce that the Christoffel symbols $\Gamma_{\alpha\beta}^\gamma$ and $\Cr_{\alpha\beta}^\gamma$, of the metric $g$ and
$g_0$ respectively, are related to each other via the identity
\begin{equation}
\label{eq:christoffels_D}
\Gamma_{\alpha\beta}^{\gamma}=\Cr_{\alpha\beta}^{\gamma}+\frac{u}{1-u^2}\left(\delta_{\alpha}^{\gamma}\pa_{\beta}u+\delta_{\beta}^{\gamma}\pa_{\alpha}u-(g_0)_{\alpha\beta}(g_0)^{\gamma\eta}\pa_{\eta}u\right)
\end{equation}
Comparing the local expressions for the Hessians of a given function $w \in {\mathscr C}^2(M^*)$ with respect to the metrics $g$ and $g_0$, namely
$\nana_{\alpha\beta}w=\pa^{\,2}_{\alpha\beta}w-
\Gamma_{\alpha\beta}^{\gamma}\pa_{\gamma}w$ and $\DD_{\alpha\beta}w=\pa^{\,2}_{\alpha\beta}w-
\Cr_{\alpha\beta}^{\gamma}\pa_{\gamma}w$,
one gets 
\begin{align*}
\nana_{\alpha\beta} w &=\DD_{\alpha\beta} w-\frac{u}{1-u^2}\left(\pa_{\alpha}u\, \pa_{\beta}w+\pa_{\alpha}w\, \pa_{\beta}u-\langle \D u\,|\,\D w\rangle\, g^{(0)}_{\alpha\beta}\right)
\\
\Deg w &=(1-u^2)\De w+(n-2)u\,\langle \D u\,|\,\D w\rangle_{g_0}
\end{align*}

\noindent We note that in the above expressions as well as in the following ones, the notations $\na$ and $\Deg$ represent the Levi-Cita connection and the Laplace-Beltrami operator of the metric $g$. In particular, letting $w=u$ and using $\De u =- n\,u$, one has
\begin{align}
\label{eq:hess_D}
\nana_{\alpha\beta}u
     &\,=\,\DD_{\alpha\beta}u \, - \, \frac{u}{1-u^2}
            \,\Big( \, 2 \,\pa_{\alpha}u\,\pa_{\beta}u
             \, - \, |\D u|^2\,
              g^{(0)}_{\alpha\beta} \, \Big) \, ,\\         
\label{eq:lapl_D}
\Deg u
     &\,=\, - n\, u\, (1-u^2) + \, (n-2) \,\, |\D u |^2 
\, .                
\end{align}
To continue, we observe that the Ricci tensor $\Ricg=\cRicg_{\alpha\beta}\,dy^{\alpha}\!\otimes dy^{\beta}$ of the metric $g$ can be expressed\vspace{-0.1cm}\\
in terms of the Ricci tensor $\Ric ={\rm R}^{(0)}_{\alpha\beta}\,dy^{\alpha}\!\otimes dy^{\beta}$ of the metric $g_0$ as 
\begin{equation}
\label{eq:ricci_2_D}
\cRicg_{\alpha\beta}=\cRic_{\alpha\beta}-\frac{(n-2)u}{1-u^2}\,\DD_{\alpha\beta} u-\frac{n-2}{(1-u^2)^2}\,\pa_{\alpha} u\,\pa_{\beta} u-\left(\frac{u\,\De u}{1-u^2}+\frac{(n-1)u^2+1}{(1-u^2)^2}\,|\D u|^2\right) g^{(0)}_{\alpha\beta}.
\end{equation}

If we plug equations $\De u= -n\, u$ and $u\,\Ric=\DD u + nu\go$ in the above formula we obtain:

\begin{equation}\label{eq:ricci_D}
\mbox{R}^{(g)}_{\alpha\beta}=\frac{1-(n-1)u^2}{u(1-u^2)}\DD_{	\alpha\beta} u-\frac{n-2}{(1-u^2)^2}\,\pa_{\alpha} u\,\pa_{\beta} u
+\left(\frac{n}{1-u^2}-\frac{(n-1)u^2+1}{(1-u^2)^2}\,|\D u|^2\right)g^{(0)}_{\alpha\beta}
\end{equation}
In order to obtain nicer formul\ae, it is convenient to introduce the new variable
\begin{equation}\label{eq:ffi_D}
\ffi=\frac{1}{2}\log\left(\frac{1+u}{1-u}\right)\qquad\iff\quad u=\tanh(\ffi).
\end{equation}
As a consequence, we have that 
\begin{align}
\label{eq:defideu_D}
\pa_{\alpha}\ffi&=\frac{1}{1-u^2}\pa_{\alpha}u
\\
\label{eq:dedefidedeu_D}
\nana_{\alpha\beta}\ffi&=\frac{1}{1-u^2}\,\DD_{\alpha\beta} u+\frac{u}{(1-u^2)^2}\,|\D u|^2\, g^{(0)}_{\alpha\beta}
\end{align}

For future convenience, we report the relation between $|\na \ffi|^2_\g$ and $|\D u|^2$ as well as the one between $|\nana \ffi|_\g^2$ and $|\DD u|^2$, namely
\begin{align}
\notag
|\na \ffi|^2_\g &\, = \, \frac{ |\D u|^2 }{1-u^2} \, , 
\\
\label{eq:|nanaffi|_D}
|\nana \ffi |_\g^2 
& \, = \, |\DD u|^2 \, + \, n\, u^2 \, \frac{ |\D u|^2 }{1-u^2} \left( \frac{ |\D u|^2 }{1-u^2} - 2\right)\, .
\end{align}

Combining expressions~\eqref{eq:hess_D},~\eqref{eq:lapl_D},~\eqref{eq:ricci_D} together with~\eqref{eq:defideu_D},~\eqref{eq:dedefidedeu_D}, we are now in the position to reformulate problem~\eqref{eq:pb_D} as

\begin{equation}
\label{eq:pb_conf_D}
\begin{dcases}
\Ricg=\left(\coth(\ffi)-(n-1)\tanh(\ffi)\right)\nana\ffi-(n-2)d\ffi\otimes d\ffi+\left(n-2|\na \ffi|_g^2\right)g, & \mbox{in } M^*
\\
\Deg\ffi= - n\,\tanh(\ffi)\,\left(1- |\na\ffi|^2_g\right), & \mbox{in } M^*
\\
\ \ \ \ \ffi=0, & \mbox{on } \pa M^*
\\
\ \ \ \ \ffi\to +\infty & \mbox{as }x\to *
\end{dcases}
\end{equation}

\smallskip

\noindent Here we recall that $M^*$ is the manifold $M\setminus {\rm MAX}(u)$. The notation $x\to *$ means that $x\to p$, where $p$ is a point of ${\rm MAX}(u)$, with respect to the topology induced by $M$ on $M^*$.

\subsection{A conformal change of metric (case $\Lambda<0$). }

First of all, we notice that if $(M,g_0, u)$ is a solution of problem~\eqref{eq:pb_A} and satisfies Normalization~\ref{norm:A}, then one has that $u^2-1 > 0$ everywhere in $M^*=M\setminus {\rm MIN}(u)$.
Motivated by the explicit formul\ae~\eqref{eq:A} of the anti-de Sitter solution, we are led to consider the following conformal change of metric
\begin{equation}\label{eq:g_A}
g=\frac{g_0}{u^2-1}.
\end{equation}
on the manifold $M^*$. Notice that, if Assumption~\ref{ass:A} holds, the function $1/\sqrt{u^2-1}$ is a defining function, hence the metric $g$ extends to the conformal boundary. In particular the volume of $\pa M$ with respect to $g$ is finite, that is
$$
|\pa M|_g \, = \, \lim_{t\to+\infty}\int\limits_{\{u=t\}} \!\! \rmd\sigma_g \, < \, +\infty\,.
$$ 

It is immediately seen that when $u$ and $g_0$ are as in~\eqref{eq:A} then $g$ is a cylindrical metric. Hence, we will refer to the conformal change~\eqref{eq:g_A} as to a {\em cylindrical ansatz}.

Our next task is to reformulate problem~\eqref{eq:pb_A} in terms of $g$. To this aim we fix local coordinates $\{y^{\alpha}\}_{\alpha=1}^n$ in $M^*$
and using standard formul\ae\ for conformal changes of metrics, 
we deduce that the Christoffel symbols $\Gamma_{\alpha\beta}^\gamma$ and $\Cr_{\alpha\beta}^\gamma$, of the metric $g$ and
$g_0$ respectively, are related to each other via the identity
\begin{equation}
\Gamma_{\alpha\beta}^{\gamma}=\Cr_{\alpha\beta}^{\gamma}-\frac{u}{u^2-1}\left(\delta_{\alpha}^{\gamma}\pa_{\beta}u+\delta_{\beta}^{\gamma}\pa_{\alpha}u-(g_0)_{\alpha\beta}(g_0)^{\gamma\eta}\pa_{\eta}u\right)
\end{equation}
Comparing the local expressions for the Hessians of a given function $w \in {\mathscr C}^2(M^*)$ with respect to the metrics $g$ and $g_0$, namely
$\nana_{\alpha\beta}w=\pa^{\,2}_{\alpha\beta}w-
\Gamma_{\alpha\beta}^{\gamma}\pa_{\gamma}w$ and $\DD_{\alpha\beta}w=\pa^{\,2}_{\alpha\beta}w-
\Cr_{\alpha\beta}^{\gamma}\pa_{\gamma}w$,
one gets 
\begin{align*}
\nana_{\alpha\beta} w &=\DD_{\alpha\beta} w+\frac{u}{u^2-1}\left(\pa_{\alpha}u\, \pa_{\beta}w+\pa_{\alpha}w\, \pa_{\beta}u-\langle \D u\,|\,\D w\rangle\, g^{(0)}_{\alpha\beta}\right)
\\
\Deg w &=(u^2-1)\De w-(n-2)u\,\langle \D u\,|\,\D w\rangle_{g_0}
\end{align*}
\smallskip
We note that in the above expressions as well as in the following ones, the notations $\na$ and $\Deg$ represent the Levi-Cita connection and the Laplace-Beltrami operator of the metric $g$. In particular, letting $w=u$ and using $\De u = n\,u$, one has
\begin{align}
\label{eq:hess_A}
\nana_{\alpha\beta}u
     &\,=\,\DD_{\alpha\beta}u \, + \, \frac{u}{u^2-1}
            \,\Big( \, 2 \,\pa_{\alpha}u\,\pa_{\beta}u
             \, - \, |\D u|^2\,
              g^{(0)}_{\alpha\beta} \, \Big) \, ,\\         
\label{eq:lapl_A}
\Deg u
     &\,=\, n\, u\, (u^2-1) - \, (n-2) \,\, |\D u |^2 
\, .                
\end{align}
To continue, we observe that the Ricci tensor $\Ricg=\cRicg_{\alpha\beta}\,dy^{\alpha}\!\otimes dy^{\beta}$ of the metric $g$ can be expressed\vspace{-0.1cm}\\
in terms of the Ricci tensor $\Ric = \R^{(0)}_{\alpha\beta}\,dy^{\alpha}\!\otimes dy^{\beta}$ of the metric $g_0$ as 
\begin{equation}
\label{eq:ricci_2_A}
\cRicg_{\alpha\beta}=\cRic_{\alpha\beta}+\frac{(n-2)u}{u^2-1}\,\DD_{\alpha\beta} u-\frac{n-2}{(u^2-1)^2}\,\pa_{\alpha} u\,\pa_{\beta} u+\left(\frac{u\,\De u}{u^2-1}-\frac{(n-1)u^2+1}{(u^2-1)^2}\,|\D u|^2\right) g^{(0)}_{\alpha\beta}.
\end{equation}

If we plug equations $\De u= n\, u$ and $u\,\Ric=\DD u-nu\go$ in the above formula we obtain
\begin{equation}\label{eq:ricci_A}
\mbox{R}^{(g)}_{\alpha\beta}=\frac{(n-1)u^2-1}{u(u^2-1)}\DD_{	\alpha\beta} u-\frac{n-2}{(u^2-1)^2}\,\pa_{\alpha} u\,\pa_{\beta} u
+\left(\frac{n}{u^2-1}-\frac{(n-1)u^2+1}{(u^2-1)^2}\,|\D u|^2\right)g^{(0)}_{\alpha\beta}
\end{equation}
In order to obtain nicer formul\ae, it is convenient to introduce the new variable
\begin{equation}\label{eq:ffi_A}
\ffi=\frac{1}{2}\log\left(\frac{u+1}{u-1}\right)\qquad\iff\quad u=\coth(\ffi).
\end{equation}
As a consequence, we have that 
\begin{align}
\label{eq:defideu_A}
\pa_{\alpha}\ffi&=-\frac{1}{u^2-1}\pa_{\alpha}u
\\
\label{eq:dedefidedeu_A}
\nana_{\alpha\beta}\ffi&=-\frac{1}{u^2-1}\,\DD_{\alpha\beta} u+\frac{u}{(u^2-1)^2}\,|\D u|^2\, g^{(0)}_{\alpha\beta}
\end{align}
For future convenience, we report the relation between $|\na \ffi|^2_\g$ and $|\D u|^2$ as well as the one between $|\nana \ffi|_\g^2$ and $|\DD u|^2$, namely
\begin{align}
\notag
|\na \ffi|^2_\g &\, = \, \frac{ |\D u|^2 }{u^2-1} \, , 
\\
\label{eq:|nanaffi|_A}
|\nana \ffi |_\g^2 
& \, = \, |\DD u|^2 \, + \, n\, u^2 \, \frac{ |\D u|^2 }{u^2-1} \left( \frac{ |\D u|^2 }{u^2-1} - 2\right)\, .
\end{align}

Combining expressions~\eqref{eq:hess_A},~\eqref{eq:lapl_A},~\eqref{eq:ricci_A} together with~\eqref{eq:defideu_A},~\eqref{eq:dedefidedeu_A}, we are now in the position to reformulate problem~\eqref{eq:pb_A} as

\begin{equation}\label{eq:pb_conf_A}
\begin{dcases}
\Ricg=\left(\tanh(\ffi)-(n-1)\coth(\ffi)\right)\nana\ffi-(n-2)d\ffi\otimes d\ffi+\left(n-2|\na \ffi|_g^2\right)g \, , & \mbox{in } M^*
\\
\Deg\ffi=-n\,\coth(\ffi)\,\left(1-|\na\ffi|^2_g\right)\, , & \mbox{in } M^*
\\
\ \ \ \ \ffi = 0\,, & \mbox{on } \pa M^*
\\
\ \ \ \ \ffi\rightarrow +\infty & \mbox{as }x\rightarrow * \, .
\end{dcases}
\end{equation}

\smallskip

\noindent Here we recall that $M^*$ is the manifold $M\setminus {\rm MIN}(u)$ and that $\pa M^*$ is the conformal boundary of $M^*$. The notation $x\to *$, means that $x\to p$, where $p$ is a point of ${\rm MIN}(u)$, with respect to the topology induced by $M$ on $M^*$.

\subsection{A unifying formalism. }

We recall that the relation between $u$ and $\ffi$ is given by~\eqref{eq:ffi_D} if $\Lambda>0$ and by~\eqref{eq:ffi_A} if $\Lambda<0$. In both cases, $u = u(\ffi)$ obeys the equation
\begin{equation*}
\frac{du}{d\ffi}=1-u^2 \, .
\end{equation*}
Since this is the only formal property of $u$ that will be needed in the following, we proceed by noticing that both systems~\eqref{eq:pb_conf_D} and~\eqref{eq:pb_conf_A} can be rewritten in the form
\begin{equation}\label{eq:pb_conf}
\begin{dcases}
\Ricg= \Big(\frac{1}{u}-(n-1)u \Big)\nana\ffi-(n-2)d\ffi\otimes d\ffi+\left(n-2+2(1-|\na \ffi|_g^2)\right)g \, , & \mbox{in } M^*
\\
\Deg\ffi=-n\,u\,\left(1-|\na\ffi|^2_g\right) \, , & \mbox{in } M^*
\\
\ \ \ \ \ffi=0\, , & \mbox{on } \pa M^*
\\
\ \ \ \; \ffi\rightarrow +\infty \, , & \mbox{as } x\to * \, ,
\end{dcases}
\end{equation}
where eventually $u = \tanh (\ffi)$ or $\coth(\ffi)$.

To describe the idea that will lead us throughout the analysis of system~\eqref{eq:pb_conf}, we note that taking the trace of the first equation one gets
\begin{equation}
\label{eq:tilde_R}
\frac{\Rg}{n-1} \, = \, (n-2) + (nu^2 + 2)\left( 1 - |\na\ffi|^2_g \right)\, ,
\end{equation}
where $\Rg$ is the scalar curvature of the metric $\g$. It is important to observe that in the cylindrical situation, which is the conformal counterpart of the (anti-)de Sitter solution, $\Rg$ has to be constant. In this case, the above formula implies that also $|\na \ffi|_g$ has to be constant and equal to $1$. 
%Plugging this information into the Bochner formula, it is then immediate to conclude that $\ffi$ has to be an affine function for the metric $\g$. 
For these reasons, also in the  situation, where we do not know a priori if $g$ is cylindrical, it is natural to think of $\na \ffi$ as to a candidate splitting direction and to investigate under which conditions this is actually the case. 

Now we rephrase Assumptions~\ref{ass:D} and Assumption~\ref{ass:A} in terms of $\ffi$.

\begin{assumption}
\label{ass:conf}
We require the following
\begin{itemize}
\item[(i)] In the case $\Lambda>0$, we assume $1-|\na\ffi|_g^2\geq 0$ on $\pa M$. 
\item[(ii)] In the case $\Lambda<0$, we suppose that $\lim_{x\to\bar x}u^2(1-|\na\ffi|_g^2)\geq 0$ for every point $\bar x\in\pa M$.
\end{itemize}
\end{assumption}

This assumption allows to estimate the behavior of $|\na\ffi|_g$ on the whole manifold $M^*$.

\begin{lemma}
\label{le:grad_ffi}
Let $(M^*,g,\ffi)$ be a solution of problem~\eqref{eq:pb_conf} satisfying Assumption~\ref{ass:conf}. Then the following condition holds on the whole manifold $M^*$
$$
1-|\na\ffi|_g^2\,\geq\, 0\,.
$$
\end{lemma}

\begin{proof}
From the Bochner formula and the equations in~\eqref{eq:pb_conf}, we get
\begin{align}
\Deg |\na\ffi|_g^2 \, &= \, 2 \, \big|\nana\ffi\big|^2_\g
+2\,\Ricg(\na\ffi,\na\ffi)
+2\,\big\langle\na\Deg\ffi \, \big| \,\na\ffi\big\rangle_{\g} \, \notag
\\
\label{eq:DegW}
&= \, 2\, \big|\nana\ffi\big|_g^2 +\left(\frac{1}{u}+(n+1) u\right)\langle \na |\na\ffi|_g^2 \,|\, \na\ffi  \rangle_g \, +\, 2\, n\, u^2\, |\na\ffi|_g^2\, \left(1-|\na\ffi|_g^2\right).
\end{align}
Now we turn to the computation of the gradient and laplacian of the function
$$
w\,\,=\,\,\beta\, \left(\,1-|\na\ffi|_g^2\,\right)\,,
$$
where $\beta=\beta(\ffi)$ is an arbitrary $\mathscr{C}^1$ function.
Using~\eqref{eq:DegW} and~\eqref{eq:pb_conf} again, we get 
$$
\na w \, = \,\frac{\dot \beta}{\beta}\,\, w\, \na \ffi \,-\,\beta\,\na |\na\ffi|_g^2\,.
$$
\begin{align*}
\Deg w \,&=\,-\,\beta\, \Deg |\na\ffi|_g^2\, - \,\frac{\dot \beta}{\beta} \langle\, \beta\,\na |\na\ffi|_g^2 \,|\, \na \ffi\,\rangle_g\, + \,\frac{\dot\beta}{\beta} \langle\, \na w \,|\, \na \ffi\, \rangle_g \,+\, \left(\frac{\ddot\beta}{\beta}-\bigg(\frac{\dot\beta}{\beta}\bigg)^2\right) w \, |\na\ffi|_g^2\, +\, \frac{\dot \beta}{\beta}\, w\, \Deg\ffi  
\\
&=\,-\, 2\beta\, |\nana\ffi|_g^2 \, + \left( 2\,\frac{\dot\beta}{\beta} +\frac{1}{u}+(n+1)u\right) \langle \na w \,|\, \na \ffi\rangle_g
\\
&\phantom{==========} + \, \left(\frac{\ddot\beta}{\beta}-2\bigg(\frac{\dot\beta}{\beta}\bigg)^2-\frac{\dot\beta}{\beta}\bigg(\frac{1}{u}+(n+1) u \bigg)-2\,n\,u^2\right) w\, |\na\ffi|_g^2\, -\dot\beta\, \frac{(\Deg \ffi)^2}{nu}\,.
\end{align*}

We find that the right choice in order to simplify the expression above is to define the function $\beta$ as the solution of the differential equation 
$$
\frac{\dot\beta}{\beta} \, + \, 2 u \, = \,0\,.
$$
More explicitly:
\begin{equation*}
\beta(\ffi)=
\begin{dcases}
\frac{1}{\cosh^2(\ffi)} & \mbox{(case $\Lambda>0$)}\,,
\\
\frac{1}{\sinh^2(\ffi)} & \mbox{(case $\Lambda<0$)}\,.
\end{dcases}
\end{equation*}

\smallskip

\noindent With this choice of $\beta$, the equation above may be rewritten in the simplified form:
\begin{equation}
\label{eq:lapl_w}
\Deg\, w - \left(\frac{1}{u}+(n-3)u\right) \langle \na w \,|\, \na \ffi\rangle_g \, = \,- 2\,\beta \, \left(\big|\nana\ffi\big|^2_g-\frac{(\Deg\, \ffi)^2}{n}\right).
\end{equation}

We notice that the term on the right of equation~\eqref{eq:lapl_w} is always nonpositive, thus the elliptic operator
$$
{\rm L}(\,\cdot\,)\, = \, \Deg\cdot \, - \, \left(\frac{1}{u}+(n-3)u\right) \langle \na \cdot \,|\, \na \ffi\rangle.
$$
satisfies
$$
{\rm L}(w)\,\leq\, 0 \quad \mbox{on $M^*$}.
$$
Thanks to Assumption~\ref{ass:conf}, it holds $w\geq 0$ on $\pa M$. 
Let us suppose for the moment that $w\to 0$ as $\ffi\to +\infty$.
Then, recalling that, since $\ffi$ is analytic, its singular values are discrete (see~\cite{Soucek}), we can choose $s>0$ small enough and $S>0$ big enough in such a way that the level sets $\{\ffi=s\}$ and $\{\ffi=S\}$ are regular.
Thus the set $\{s\leq\ffi\leq S\}$ is a (compact) manifold and we can use the strong minimum principle to obtain
$$
\inf_{\{s\leq\ffi\leq S\}}\, w \, = \, \inf_{\pa\{s\leq\ffi\leq S\}} \, w  \, = \, \min_{\{\ffi=s\}\cup\{\ffi=S\}} \, w \,.
$$
Hence, since $\min_{\{\ffi=S\}}w\to 0$ as $S\to +\infty$, and $\min_{\{\ffi=s\}}w\to \min_{\pa M}w\geq 0$ as $s\to 0^+$, we easily find that $w\geq 0$ on $\{0\leq\ffi<+\infty\}=M^*$. This immediately gives the thesis.

It remains to prove that $\lim_{\ffi\to +\infty}w=0$. It is convenient to rewrite the limit in terms of $u,\go$. In the case $\Lambda>0$, the limit above is equivalent to $\lim_{u\to 1^-}(1-u^2-|\D u|^2)=0$, while in the case $\Lambda<0$ it is equivalent to $\lim_{u\to 1^+}(u^2-1-|\D u|^2)=0$. In both cases, since the points at which $u=1$ are extremals, we have $|\D u|\to 0$ as $u\to 1$ and so the limits above are verified.
\end{proof}

\subsection{The geometry of the level sets of $\ffi$.}
\label{sub:geom_levelsets}
In the forthcoming analysis a crucial role is played by the study of the geometry of the level sets of $\ffi$, which coincide with the level sets of $u$, by definition. 
Hence, we pass now to describe the second fundamental form and the mean curvature of the regular level sets of $\ffi$ (or equivalently of $u$) in both the original Riemannian context $(M, \go)$ and the conformally related one $(M^*, \g)$.
To this aim, we fix a regular level set $\{ \ffi = s_0\}$ and we construct a suitable set of coordinates in a neighborhood of it. Note that $\{ \ffi = s_0\}$ must be compact, by the properness of $\ffi$. In particular, there exists a real number $\delta>0$  such that in the tubular neighborhood $\mathcal{U}_\delta = \{s_0 - \delta < \ffi < s_0 + \delta \}$ we have $|\na \ffi |_\g > 0$ so that $\mathcal{U}_\delta$ is foliated by regular level sets of $\ffi$. As a consequence, $\mathcal{U}_\delta$ is diffeomorphic to $(s_0 -\delta , s_0 + \delta) \times \{ \ffi = s_0 \}$ and the function $\ffi$ can be regarded as a coordinate in $\mathcal{U}_\delta$. Thus, one can choose a local system of coordinates $\{\ffi,\vartheta^1\!,\!....,\vartheta^{n-1}\}$, where $\{\vartheta^1\!,\!...., \vartheta^{n-1} \}$ are local coordinates on $\{ \ffi = s_0 \}$. In such a system, the metric $\g$ can be written as
\begin{equation*}
\g \, = \,\frac{d\ffi \otimes d\ffi}{|\na \ffi|_\g^2} +\g_{ij}(\ffi,\vartheta^1 \!,\!...., \vartheta^{n-1})\,d\vartheta^i\!\otimes d\vartheta^j \,,
\end{equation*}
where the latin indices vary between $1$ and $n-1$. We now fix in $\mathcal{U}_\delta$ the $\go$-unit vector field 
$\nu =\D u/|\D u|=\D\ffi/|\D\ffi|$ and the $\g$-unit vector field $\nu_g =\na u/|\na u|_{\g}=\na\ffi/|\na\ffi|_g$. Accordingly, the second fundamental forms of the regular level sets of $u$ or $\ffi$ with respect to ambient metric $\go$ and the conformally-related ambient metric $g$ are respectively given by
\begin{equation*}
%\label{eq:formula_fundamentalform}
\cho_{ij}=\frac{\DD_{ij} u}{|\D u|}=\frac{\DD_{ij}\ffi}{|\D\ffi|}
\qquad\mbox{and}\qquad
\chg_{ij}=\frac{\nana_{ij} u}{|\na u|_\g}
=\frac{\nana_{ij}\ffi}{|\na\ffi|_\g}\, ,  \qquad \mbox{for}\quad i,j = 1,\!...., n-1.
\end{equation*}  
Taking the traces of the above expressions with respect to the induced metrics 
we obtain the following expressions for the mean curvatures in  the two ambients
\begin{equation}
\label{eq:formula_curvature}
\Ho=\frac{\De u}{|\D u|}-\frac{\DD u(\D u,\D u)}{|\D u|^3}\,,
\qquad\qquad
\Hg=\frac{\Deg \ffi}{|\na \ffi|_g}-\frac{\nana\ffi(\na\ffi,\na\ffi)}{|\na\ffi|_{\g}^3}\,.
\end{equation}
Taking into account expressions~\eqref{eq:defideu_D},~\eqref{eq:defideu_A} and~\eqref{eq:dedefidedeu_D},~\eqref{eq:dedefidedeu_A}, one can show that 
the second fundamental forms are related by
\begin{equation}
\label{eq:formula_h_h_g}
\chg_{ij}
\,=\,
\begin{dcases}
  \frac{1}{\sqrt{1-u^2}}\left[\,\cho_{ij}\, +\, \frac{u\,|\D u|}{1-u^2}\,  \cgo_{ij}\,\right]  & 
  \mbox{(case $\Lambda>0$)},
 \\
  \frac{1}{\sqrt{u^2-1}}\left[-\cho_{ij}\, +\, \frac{u\,|\D u|}{u^2-1}\,  \cgo_{ij}\,\right]  & 
  \mbox{(case $\Lambda<0$)}.
 \end{dcases}
\end{equation}
The analogous formula for the mean curvatures reads
\begin{align}
\label{eq:formula_H_H_g}
{\Hg}
&\,=\,
\begin{dcases}
  \sqrt{1-u^2}\left[\,\HHH \,+\, (n-1)\,\frac{u\,|\D u|}{1-u^2}\,\right]  & 
  \mbox{(case $\Lambda>0$)},
 \\
  \sqrt{u^2-1}\left[-\HHH \,+\, (n-1)\, \frac{u\,|\D u|}{u^2-1}\,\right]  & 
  \mbox{(case $\Lambda<0$)}.
 \end{dcases}
\end{align}

Concerning the nonregular level sets of $\ffi$, we first observe that $\ffi$ is analytic (see~\cite{Chr}), thus by the results in~\cite{Federer}, one has that the $(n-1)$-dimensional Hausdorff measure of the level sets of $\ffi$ is locally finite. Hence, the properness of $\ffi$ forces the level sets to have finite $(n-1)$-dimensional Hausdorff measure. 
Using the results in~\cite{Krantz}, we know that there exists a submanifold $N\subseteq{\rm Crit}(\ffi)$ such that $\mathscr{H}^{n-1}({\rm Crit}(\ffi)\setminus N)=0$. In particular, the unit normal to a level set is well-defined $\mathscr{H}^{n-1}$-almost everywhere, and so are the second fundamental form $\hg$ and the mean curvature $\Hg$. We will prove now that formul\ae~\eqref{eq:formula_h_h_g} and~\eqref{eq:formula_H_H_g} hold also at any point $y_0\in N$, and therefore they hold $\mathscr{H}^{n-1}$-almost everywhere on any level set. We do it in the case $\Lambda>0$ (the case $\Lambda<0$ is analogous). Let $\nu,\nu_g$ be the unit normal vector fields to $N$ at $y_0$ with respect to $\go,g$ respectively. Since $
|\nu_g|_g^2\,=\,1\,=\,|\nu|^2\,=\,(1-u^2)\,|\nu|_g^2$, we deduce that
$\nu_g\,=\,\sqrt{1-u^2}\,\nu$.
Let $(\pa/\pa x^1,\dots,\pa/\pa x^{n-1})$ be a basis of $T_{y_0}N$, so that in particular $(\pa/\pa x^1,\dots,\pa/\pa x^{n-1},\nu_g)$ is a basis of $T_{y_0}M$. Recalling~\eqref{eq:christoffels_D} and observing that the derivatives of $u$ in $y_0$ are all zero since $y_0\in{\rm Crit}(\ffi)$, we have
$$
\chg_{ij}\,=\,\Big\langle \na_i\frac{\pa}{\pa x^j}\,\Big|\,\nu_g\Big\rangle_g\,=\,\Gamma_{ij}^n\,=\,\Cr_{ij}^n\,=\,\Big\langle \D_i\frac{\pa}{\pa x^j}\,\Big|\,\nu_g\Big\rangle_g\,=\,\frac{1}{\sqrt{1-u^2}}\,\Big\langle \D_i\frac{\pa}{\pa x^j}\,\Big|\,\nu\Big\rangle\,=\,\frac{1}{\sqrt{1-u^2}}\,\cho_{ij}\,.
$$
This proves that formula~\eqref{eq:formula_h_h_g} holds also on $N$, and taking its trace we deduce that also~\eqref{eq:formula_H_H_g} is verified.

%We conclude this section with some considerations about the geometry of $\pa M$ in the case of null Dirichlet boundary conditions for $\ffi$. We recall that, as $\ffi$ is nonconstant in $M$ and $\pa M$ is assumed to be smooth, the Hopf Lemma implies that $|\na \ffi|_g>0$ on $\pa M$. To continue, we observe that from the first equation in~\eqref{eq:pb_conf} and from its traced version~\eqref{eq:tilde_R} it is immediate to deduce that
%\begin{align}
%\label{eq:hessg}
%\nonumber |\nana \ffi |_\g^2 & \,\, = \,\, 
%\tanh^2(\ffi) \, \bigg( \, \frac{2}{n-2} \, \Ricg (\na \ffi, \na\ffi) \, + \, |\Ricg|_g^2 \, - \, \frac{\Rg^2}{n-1} \,\bigg) \, .
%\end{align}
%In particular, we have that $\nana \ffi \equiv 0$ on $\{\ffi = 0 \}$. Hence, the boundary of $(M^*,g)$ is totally geodesic and $|\na \ffi|_g$ is constant on each connected component of $\pa M$. Notice once again that if $(M^*,g)$ is a cylinder and $\na \ffi$ is the splitting direction, one has 
%\begin{equation*}
%\Ricg (\na \ffi, \na\ffi) \, = \, 0 \qquad \hbox{and} \qquad |\Ricg|_g^2 \, = \, \frac{\Rg^2}{n-1}  \, ,
%\end{equation*}
%so that the right hand side of the above identity vanishes everywhere and $\ffi$ is an affine function in $(M^*,g)$, as expected.

\subsection{A conformal version of the Monotonicity-Rigidity Theorem.} 

\label{sub:reform}

We conclude this section by introducing the conformal analog of the functions $U_p(t)$ introduced in~\eqref{eq:Up_D} ($\Lambda > 0$) and~\eqref{eq:Up_A} ($\Lambda < 0$). To this aim,  we let $(M^*, \g, \ffi)$ be a solution to problem~\eqref{eq:pb_conf} and we define, for $p \geq 0$, the functions $\Phi_p : [0, +\infty) \longrightarrow \R$ as 
\begin{equation}
\label{eq:fip}
s \,\, \longmapsto \,\, \Phi_p(s) \,\,  =\!\!\!
\int\limits_{\{\ffi = s\}}\!\!\!
|\na \ffi|_g^p \,\,\rmd \sigma_{\!g} \, .
\end{equation}
As for the $U_p$'s, we observe that the $\Phi_p$'s are well defined. This is because the hypersurface area of the level sets is finite, due to the analyticity and properness of $\ffi$.
Before proceeding, it is worth noticing that, when $p=0$, the function 
$$
\Phi_0(s) \, = \!\!\!\int\limits_{\{\ffi = s\}}\!\!\!\!
\rmd \sigma_{\!g} \, = \, |\{ \ffi = s\}|_g ,
$$ 
coincides with the hypersurface area functional $|\{ \ffi = s\}|_g$ for the level sets of $\ffi$ inside the ambient manifold $(M^*,g)$. For future convenience, we observe that the functions $U_p$ and $\Phi_p$ and their derivatives (when defined) are related as follows

\begin{eqnarray}
\label{eq:upfip1}
U_p(t) & = &  \,\Phi_p \Big(\, \frac12 \log{\Big| \frac{1+t}{1-t} \Big|} \, \Big) \, , \\
\label{eq:upfip2}
U'_p(t) & = & \frac{1}{1-t^2} \, \,\Phi'_p  \Big(\, \frac12 \log{\Big| \frac{1+t}{1-t} \Big|} \, \Big)  \, , \\
\label{eq:upfip3}
U''_p(t) & = & \frac{1}{(1-t^2)^2} \, \left[\, 2 \, t \, \Phi'_p \Big(\, \frac12 \log{\Big| \frac{1+t}{1-t} \Big|} \, \Big)  +   \Phi''_p \Big(\, \frac12 \log{\Big| \frac{1+t}{1-t} \Big|} \, \Big)   \, \right]\, ,
%\bigg(\!\log\bigg(\frac{1+t}{1-t}\bigg)\bigg)
\end{eqnarray}

Using the above relationships, both the Monotonicity-Rigidity Theorems~\ref{thm:main_D} and~\ref{thm:main_A} can be rephrased in terms of the functions $s \mapsto \Phi_p(s)$ as follows.

\begin{theorem}[Monotonicity-Rigidity Theorem -- Conformal Version]
\label{thm:main_conf}
Let $(M^*,g,\ffi)$ be a solution to problem~\eqref{eq:pb_conf}, satisfying Assumption~\ref{ass:conf}. For every $p \geq 0$ we let $\Phi_p : [0, +\infty) \longrightarrow \R$ be the function defined  in~\eqref{eq:fip}. Then, the following properties hold true.
\begin{itemize}
\item[(i)] For every $p\geq 1$, the function $\Phi_p$ is continuous.  

\smallskip

\item[(ii)] The function $\Phi_1(s)$ is monotonically nonincreasing. Moreover, if $\Phi_1(s_1)=\Phi_1(s_2)$ for some $s_1\neq s_2$, then $(M^*,g,\ffi)$ is isometric to one half round cylinder with totally geodesic boundary.

\smallskip

\item[(iii)] For every $p \geq 3$, the function $\Phi_p$ is differentiable and the derivative satisfies, for every $s \in (0,+ \infty)$,

\begin{equation}
\label{eq:der_fip}
\ \ \ \ \Phi'_p(s)= \int\limits_{\{\ffi=s\}} \!\!\! \left[ -(p-1) |\na\ffi|_g^{p-1} \, \Hg + p\, |\na\ffi|_g^{p-2} \, \Deg\ffi \right]\rmd\sigma_g \, 
\,\leq\!\int\limits_{\{\ffi=s\}} \!\! |\na\ffi|_g^{p-2} \Deg\ffi \,\,\rmd\sigma_g\,\leq\, 0 \, .
\end{equation}
where $\Hg$ is the mean curvature of the level set $\{\ffi=s\}$. Moreover, if the first equality in~\eqref{eq:der_fip} holds, for some $s\in (0, + \infty) $ and some $p \geq 3$, then $(M^*,g,\ffi)$ is isometric to one half round cylinder with totally geodesic boundary.

\smallskip

\item[(iv)] It holds $\Phi_p'(0) = \lim_{s \to 0^+} \Phi_p'(s)= 0$, for every $p\geq 3$. In particular, setting $\Phi''_p(0) = \lim_{s \to 0^+}\Phi_p'(s)/s$, we have that for every $p \geq 3$, the following formula holds
\begin{align}
\notag
-n\int\limits_{\pa M}|\na\ffi|_g^{p-2}\left(1-|\na\ffi|_g^2\right)\rmd\sigma_g &\geq\Phi_p''(0)=\int\limits_{\pa M} \!|\na\ffi|_g^{p-2}\left[(p-1)\Ricg (\nu_g, \nu_g)-np\left(1-|\na\ffi|_g^2\right)\right]
\,\rmd\sigma_g\, , & {\rm (\Lambda>0)}
\\
\label{eq:der2_fip}
-n\int\limits_{\pa M}u^2\left(1-|\na\ffi|_g^2\right)\rmd\sigma_g &\geq\Phi_p''(0)=\int\limits_{\pa M} -\left[
\left(\frac{p-1}{n-1}\right)\Ricg (\nu_g, \nu_g)+nu^2(1-|\na\ffi|_g^2)\right]
\,\rmd\sigma_g \,, & {\rm (\Lambda<0)}
\end{align}
where $\nu_g = \na\ffi/|\na \ffi|_g$ is the inward pointing unit normal of the boundary $\pa M$.
Moreover, in the case $\Lambda>0$, if the equality is fulfilled for some $p \geq 3$, then $(M^*,g,\ffi)$ is isometric one half round cylinder with totally geodesic boundary.
\end{itemize}
\end{theorem}

\section{Proof of Theorems~\ref{thm:main_D} and~\ref{thm:main_A} after Theorem~\ref{thm:main_conf}}
\label{sec:translation}

In this section we show how to recover Theorems~\ref{thm:main_D},~\ref{thm:main_A} from Theorem~\ref{thm:main_conf}. The proof of Theorem~\ref{thm:main_conf} will be the argument of Section~\ref{sec:proofs}. 

\subsection{Case $\Lambda>0$: Theorem~\ref{thm:main_conf} implies Theorem~\ref{thm:main_D}.}

In the hypoteses of Theorem~\ref{thm:main_D}, Lemma~\ref{le:grad_ffi} is in charge, hence $|\na\ffi|_g^2\leq 1$ on the whole manifold $M$. Thus, formula~\eqref{eq:Dusign_D} is an immediate consequence of the identity
$$
|\na\ffi|_g^2\,=\,\frac{|\D u|^2}{1-u^2}\,.
$$
The equivalence between Theorem~\ref{thm:main_D}-(i),(ii) and Theorem~\ref{thm:main_conf}-(i),(ii) is also straightforward.

We pass now to prove the equivalence between~\eqref{eq:derup_D} and~\eqref{eq:der_fip}. To do this, it is enough to translate~\eqref{eq:der_fip} in terms of the conformally related quantities $u,\go$. Recalling the second equation in~\eqref{eq:pb_conf} and formul\ae~\eqref{eq:formula_H_H_g},~\eqref{eq:upfip2}, we have the following chain of equalities.
\begin{align}
\notag
U_p'(t)\,\,&= \,\,\frac{1}{1-t^2}\,\,\Phi'_p\bigg(\frac{1}{2}\log\Big(\frac{1+t}{1-t}\Big)\bigg)
\\
\notag
&=\,\,\frac{1}{1-t^2}\int\limits_{\{ u=t\}} \!\!\! |\na\ffi|_g^{p-2} \Big[ -(p-1) |\na\ffi|_g \, \Hg + p\, \Deg\ffi \,\Big]\rmd\sigma_g
\\
\notag
&=\int\limits_{\{u=t\}} \!\!\!\frac{|\D u|^{p-2}}{(1-u^2)^{p/2}} \bigg[ -(p-1) |\D u|\bigg(\HHH \,+\, (n-1)\,\frac{u\,|\D u|}{1-u^2}\bigg) -\,n\, p\, u\,\bigg(1-\frac{|\D u|^2}{1-u^2}\bigg)\, \bigg]\rmd\sigma_g
\\
\label{eq:derfip_to_derup_D}
&=\Big( \frac{1}{1-t^2} \Big)^{\frac{p}{2}}\!\!\!\int\limits_{\{u=t\}} \!\!\!|\D u|^{p-2} \bigg[ -(p-1) |\D u|\,\HHH \,+\, (n+p-1)\,\frac{u\,|\D u|^2}{1-u^2} -\,n\, p\, u\, \bigg]\rmd\sigma_g\,.
\end{align}
Now we use formula~\eqref{eq:g_D} to deduce that the volume elements $\rmd\sigma$, $\rmd\sigma_g$ are related by
$$
\rmd\sigma_g=\Big(\frac{1}{1-u^2}\Big)^{\frac{n-1}{2}}\rmd\sigma\,.
$$
Hence equality~\eqref{eq:derfip_to_derup_D} can be rewritten as 
\begin{align*}
U_p'(t)\,\,&=\,\, - \, (p-1) \,\, t \,\, \Big( \frac{1}{1-t^2} \Big)^{\!\!\frac{n+p-1}{2}} \!\!\!\!\!
\int\limits_{\{u=t\}}
\!\!\!\!
|\D u|^{p-2}
\!\left[\,\, \bigg| \frac{\D u}{u} \bigg| \, \HHH \, + \, \bigg( \!\frac{np}{p-1}\! \bigg) \, - \, \bigg(\!\frac{n+p-1}{p-1} \!\bigg) \bigg( \frac{|\D u|^2}{1-u^2} \bigg)
\, \right]  \rmd\sigma \,
\\
&= \,\, - \, (p-1) \,\, t \,\, \Big( \frac{1}{1-t^2} \Big)^{\!\!\frac{n+p-1}{2}} \!\!\!\!\!
\int\limits_{\{u=t\}}
\!\!\!\!
|\D u|^{p-2}
\!\left[\, (n-1) \, - \, \Ric (\nu, \nu) \, +   \bigg(\!\frac{n+p-1}{p-1} \!\bigg) \bigg(\!1- \frac{|\D u|^2}{1-u^2} \!\bigg)
\, \right]  \rmd\sigma \,,
\end{align*}
where in the second equality $\nu=\D u/|\D u|$ is the unit normal to $\{u=t\}$ and we have used the identity 
$$
\Ric(\nu,\nu)\,=\,-\,\bigg|\frac{\D u}{u} \bigg| \, \HHH\,,
$$
which is a consequence of the first equation in problem~\eqref{eq:pb_D}.
On the other hand, since from~\eqref{eq:der_fip} it holds
$$
\Phi'_p(s)\,\leq \!\int\limits_{\{\ffi=s\}} \!\! |\na\ffi|_g^{p-2} \Deg\ffi \,\,\rmd\sigma_g\,,
$$
we have
\begin{align}
\notag
U_p'(t)\,\,&\leq\,\, -\,\frac{n\,t}{1-t^2}\int\limits_{\{u=t\}}\!\!\Big(\frac{|\D u|}{\sqrt{1-u^2}}\Big)^{p-2}\Big(1-\frac{|\D u|^2}{1-u^2}\Big)\,\rmd\sigma_g
\\
\label{eq:derfip_to_derup_2_D}
&=\,\,
-\,n\,t\,\Big(\frac{1}{1-t^2}\Big)^{\frac{n+p-1}{2}}\!\!\!\int\limits_{\{u=t\}}\!\!|\D u|^{p-2}\Big(1-\frac{|\D u|^2}{1-u^2}\Big)\,\rmd\sigma\,.
\end{align}
This proves formula~\eqref{eq:derup_D}. Moreover, we recall from Theorem~\ref{thm:main_conf} that the equality holds in~\eqref{eq:derfip_to_derup_2_D} if and only if $(M^*,g,\ffi)$ is isometric to one half round cylinder, that is, if and only if $(M,\go,u)$ is isometric to the de Sitter solution. This proves Theorem~\ref{thm:main_D}-(iii).

It remains to prove the equivalence of Theorem~\ref{thm:main_D}-(iv) and Theorem~\ref{thm:main_conf}-(iv). 
We first observe, from formula~\eqref{eq:ricci_2_D}, that the identity
$$
\Ricg(\nu_g,\nu_g)\,=\,\Ric(\nu,\nu)\,-\,(n-1)\,|\D u|^2\,,
$$
holds on $\pa M$.
Then we apply the Gauss-Codazzi equation to obtain
$$
\Ricg(\nu_g,\nu_g)\,=\,\frac{\RRR-\RRR^{\pa M}}{2}\,-\,(n-1)|\D u|^2\,=\,\frac{(n-1)(n-2)-\RRR^{\pa M}}{2\ \ }\,+\,(n-1)\,(1-|\D u|^2)\,.
$$
Now, we recall formul\ae~\eqref{eq:upfip3} and~\eqref{eq:der2_fip} and we compute
\begin{align*}
U_p''(0)\,\,=\,\, 
\Phi''_p (0)\,\,&=\,\, 
\int\limits_{\pa M} \!|\na\ffi|_g^{p-2}\left[(p-1)\Ricg (\nu_g, \nu_g)\,-\,n\,p\left(1-|\na\ffi|_g^2\right)\right]
\,\rmd\sigma_g
\\
&=\,\,\int\limits_{\pa M} \!|\D u|^{p-2}\left[\,(p-1)\Big(\frac{(n-1)(n-2)-\RRR^{\pa M}}{2\ \ }\Big)-(n+p-1)\left(1-|\D u|^2\right)\,\right]
\,\rmd\sigma
\\
&=\,\,-(p-1)\!\int\limits_{\pa M} \!|\D u|^{p-2}\left[\,\frac{\RRR^{\pa M}-(n-1)(n-2)}{2\ \ }+\Big(\frac{n+p-1}{p-1}\Big)\left(1-|\D u|^2\right)\,\right]
\,\rmd\sigma\,.
\end{align*}
Moreover, again from formula~\eqref{eq:der2_fip}, we have
\begin{equation*}
U_p''(0)\,\,=\,\, 
\Phi''_p (0)\,\,\leq\,\, 
-n\int\limits_{\pa M} \!|\na\ffi|_g^{p-2}\left(1-|\na\ffi|_g^2\right)
\,\rmd\sigma_g
=-n\,\,\int\limits_{\pa M} \!|\D u|^{p-2}\left(1-|\D u|^2\right)
\,\rmd\sigma\,,
\end{equation*}
and the equality holds if and only if $(M,\go,u)$ is isometric to the de Sitter solution. This concludes the proof of Theorem~\ref{thm:main_D}-(iv).

\subsection{Case $\Lambda<0$: Theorem~\ref{thm:main_conf} implies Theorem~\ref{thm:main_A}.}

In the hypoteses of Theorem~\ref{thm:main_A}, Lemma~\ref{le:grad_ffi} is in charge, hence $|\na\ffi|_g^2\leq 1$ on the whole manifold $M$. Thus, formula~\eqref{eq:Dusign_A} is an immediate consequence of the identity
$$
|\na\ffi|_g^2\,=\,\frac{|\D u|^2}{u^2-1}\,.
$$
The equivalence between Theorem~\ref{thm:main_A}-(i),(ii) and Theorem~\ref{thm:main_conf}-(i),(ii) is also straightforward.

We pass now to prove the equivalence between~\eqref{eq:derup_A} and~\eqref{eq:der_fip}. To do this, it is enough to translate~\eqref{eq:der_fip} in terms of the conformally related quantities $u,\go$. Recalling the second equation in~\eqref{eq:pb_conf} and formul\ae~\eqref{eq:formula_H_H_g},~\eqref{eq:upfip2}, we have the following chain of equalities.
\begin{align}
\notag
U_p'(t)\,\,&= \,-\,\frac{1}{t^2-1}\,\,\Phi'_p\bigg(\frac{1}{2}\log\Big(\frac{1+t}{t-1}\Big)\bigg)
\\
\notag
&=\,-\,\frac{1}{t^2-1}\int\limits_{\{ u=t\}} \!\!\! |\na\ffi|_g^{p-2} \Big[ -(p-1) |\na\ffi|_g \, \Hg + p\, \Deg\ffi \,\Big]\rmd\sigma_g
\\
\notag
&=\,-\!\!\int\limits_{\{u=t\}} \!\!\!\frac{|\D u|^{p-2}}{(u^2-1)^{p/2}} \bigg[ -(p-1) |\D u|\bigg(-\HHH \,+\, (n-1)\,\frac{u\,|\D u|}{u^2-1}\bigg) -\,n\, p\, u\,\bigg(1-\frac{|\D u|^2}{u^2-1}\bigg)\, \bigg]\rmd\sigma_g
\\
\label{eq:derfip_to_derup_A}
&=\Big( \frac{1}{t^2-1} \Big)^{\frac{p}{2}}\!\!\!\int\limits_{\{u=t\}} \!\!\!|\D u|^{p-2} \bigg[ -(p-1) |\D u|\,\HHH \,-\, (n+p-1)\,\frac{u\,|\D u|^2}{u^2-1} +\,n\, p\, u\, \bigg]\rmd\sigma_g\,.
\end{align}
Now we use formula~\eqref{eq:g_A} to deduce that the volume elements $\rmd\sigma$, $\rmd\sigma_g$ are related by
$$
\rmd\sigma_g=\Big(\frac{1}{u^2-1}\Big)^{\frac{n-1}{2}}\rmd\sigma\,.
$$
Hence equality~\eqref{eq:derfip_to_derup_A} can be rewritten as 
\begin{align*}
U_p'(t)\,\,&=\,\, (p-1) \,\, t \,\, \Big( \frac{1}{t^2-1} \Big)^{\!\!\frac{n+p-1}{2}} \!\!\!\!\!
\int\limits_{\{u=t\}}
\!\!\!\!
|\D u|^{p-2}
\!\left[\,-\, \bigg| \frac{\D u}{u} \bigg| \, \HHH \, + \, \bigg( \!\frac{np}{p-1}\! \bigg) \, - \, \bigg(\!\frac{n+p-1}{p-1} \!\bigg) \bigg( \frac{|\D u|^2}{1-u^2} \bigg)
\, \right]  \rmd\sigma \,
\\
&= \,\, (p-1) \,\, t \,\, \Big( \frac{1}{t^2-1} \Big)^{\!\!\frac{n+p-1}{2}} \!\!\!\!\!
\int\limits_{\{u=t\}}
\!\!\!\!
|\D u|^{p-2}
\!\left[\, (n-1) \, + \, \Ric (\nu, \nu) \, +   \bigg(\!\frac{n+p-1}{p-1} \!\bigg) \bigg(\!1- \frac{|\D u|^2}{1-u^2} \!\bigg)
\, \right]  \rmd\sigma \,,
\end{align*}
where in the second equality $\nu=\D u/|\D u|$ is the unit normal to $\{u=t\}$ and we have used the identity 
$$
\Ric(\nu,\nu)\,=\,-\,\bigg|\frac{\D u}{u} \bigg| \, \HHH\,,
$$
which is a consequence of the first equation in problem~\eqref{eq:pb_A}.
On the other hand, since from~\eqref{eq:der_fip} it holds
$$
\Phi'_p(s)\,\leq \!\int\limits_{\{\ffi=s\}} \!\! |\na\ffi|_g^{p-2} \Deg\ffi \,\,\rmd\sigma_g\,,
$$
we have
\begin{align}
\notag
U_p'(t)\,\,&\geq\,\,\frac{n\,t}{t^2-1}\int\limits_{\{u=t\}}\!\!\Big(\frac{|\D u|}{\sqrt{u^2-1}}\Big)^{p-2}\Big(1-\frac{|\D u|^2}{u^2-1}\Big)\,\rmd\sigma_g
\\
\label{eq:derfip_to_derup_2_A}
&=\,\,
n\,t\,\Big(\frac{1}{t^2-1}\Big)^{\frac{n+p-1}{2}}\!\!\!\int\limits_{\{u=t\}}\!\!|\D u|^{p-2}\Big(1-\frac{|\D u|^2}{u^2-1}\Big)\,\rmd\sigma\,.
\end{align}
This proves formula~\eqref{eq:derup_A}. Moreover, we recall from Theorem~\ref{thm:main_conf} that the equality holds in~\eqref{eq:derfip_to_derup_2_A} if and only if $(M^*,g,\ffi)$ is isometric to one half round cylinder, that is, if and only if $(M,\go,u)$ is isometric to the anti-de Sitter solution. This proves Theorem~\ref{thm:main_A}-(iii).

It remains to prove the equivalence of Theorem~\ref{thm:main_A}-(iv) and Theorem~\ref{thm:main_conf}-(iv). 
Let $r=\sqrt{u^2-1}$ and $V_p(r)=U_p(\sqrt{1+1/r^2})$. An easy computation shows that
$$
V_p'(r)\,=\,-\,\frac{1}{r^2 \sqrt{1+r^2}}\,U_p'(\sqrt{1+1/r^2})\,.
$$
 Since $\lim_{s\to 0^+}\Phi_p'(s)=0$, we deduce from~\eqref{eq:upfip2} that $\lim_{t\to +\infty}U_p'(t)=0$, which implies $\lim_{r\to 0^+}V_p'(r)=0$. Hence we set $V_p''(0)=\lim_{r\to 0^+} V_p'(r)/r$ and we compute
\begin{align*}
V_p''(0)\,\,&=\,\lim_{r\to 0^+}-\,\frac{1}{r^3}\,U_p'(\sqrt{1+1/r^2})
\\
&=\,\lim_{t\to +\infty}-\,t^3\,U_p'(t)
\\
&=\,\lim_{t\to+\infty}t\, \Phi'_p \bigg(\, \frac12 \log{\Big( \frac{t+1}{t-1} \Big)} \, \bigg) \,,
\end{align*}
where in the last equality we have used formula~\eqref{eq:upfip2}. Recalling that $\Phi_p''(0)=\lim_{s\to 0^+}\Phi_p'(s)/s$, we conclude that $V_p''(0)=\Phi_p''(0)$.
Hence, from formula~\eqref{eq:der2_fip} we obtain
\begin{equation}
\label{eq:Vp_A}
V_p''(0)\,\,=\, 
-\int\limits_{\pa M} \!\left[\Big(\frac{p-1}{n-1}\Big)\,\Ricg (\nu_g, \nu_g)\,+\,n\,u^2\left(1-|\na\ffi|_g^2\right)\right]
\,\rmd\sigma_g\,.
\end{equation}
Recalling that $|\na\ffi|_g^2=|\D u|^2/(u^2-1)$, it is easily seen that
$$
\lim_{s\to 0^+} u^2(1-|\na\ffi|_g^2)\,\,=\,\,\lim_{t\to +\infty} \left(u^2-1-|\D u|^2\right)\,.
$$
Moreover, from the Gauss-Codazzi equation on $\pa M$ and formula~\eqref{eq:tilde_R} we find that the identity
\begin{equation*}
\Ricg(\nu_g,\nu_g)\,=\,\frac{\Rg-\Rg^{\pa M}}{2}\,=\,\frac{(n-1)(n-2)-\Rg^{\pa M}}{2}+\frac{(n-1)\,n\,u^2(1-|\na\ffi|_g^2)}{2}
\end{equation*}
holds on the conformal boundary $\pa M$.
Therefore, formula~\eqref{eq:Vp_A} rewrites as
$$
V_p''(0)=\,-\,(p-1)\!\!\int\limits_{\pa M} \!\left[\,\frac{(n-1)(n-2)-\Rg^{\pa M}}{2(n-1)}+\frac{n(p+1)}{2(p-1)}\,\left(u^2-1-|\D u|^2\right)\,\right]
\,\rmd\sigma_g\,.
$$
Finally, again from formula~\eqref{eq:der2_fip}, we have
\begin{equation*}
V_p''(0)
\,\,=\,\, 
\Phi''_p (0)
\,\,\leq\,\, 
n\int\limits_{\pa M} \!u^2\left(1-|\na\ffi|_g^2\right)
\,\rmd\sigma_g
\,\,=\,\,
n\int\limits_{\pa M} \!\left(u^2-1-|\D u|^2\right)
\,\rmd\sigma_g\,.
\end{equation*}
This concludes the proof of Theorem~\ref{thm:main_A}-(iv).

\section{Integral identities}
\label{sec:integral}

In this section, we derive some integral identities that will be used to analyze the properties of the functions $s\mapsto \Phi_p(s)$ introduced in~\eqref{eq:fip}. 

\subsection{First integral identity.}

To obtain our first identity, we are going to exploit
the equation $\Deg\ffi=-nu(1-|\na\ffi|_g^2)$ in problem~\eqref{eq:pb_conf}.

\begin{proposition}
\label{prop:byparts}
Let $(M^*, \g, \ffi)$ be a solution to problem~\eqref{eq:pb_conf}.
Then, for every $p\geq 1$ and for every $ s \in (0,+\infty)$, we have
\begin{equation}
\label{eq:id_byparts_bis}
\int\limits_{\{\ffi=s\}} \!\!
\frac{   |\na \ffi|_\g^{p} }{ \sinh^n(s)  }
\,\,\rmd\sigma_{\!g}  
\,\, =  \!\!
\int\limits_{\{\ffi> s\}} \!\!\frac{  |\na \ffi|_\g^{p-3}  
\Big( n \, \coth(\ffi) \, |\na\ffi|_g^{4} \, - \, |\na\ffi|_g^2 \, \Deg\ffi  \, 
- \,  (p-1)\,  \nana \ffi (\na\ffi, \!\na\ffi)  \Big) }{\sinh^n(\ffi)}
\,\,\rmd\mu_g \, .
\end{equation}
\end{proposition}
\begin{remark}
\label{rem:a_e1}
Arguing as in Remark~\ref{rem:uno_D}, it is easy to realize that the integral on the left hand side of~\eqref{eq:id_byparts_bis} is well defined also when $s$ is a singular value of $\ffi$.
\end{remark}
%\begin{remark}
%A formula similar to~\eqref{eq:id_byparts_bis} can be obtained for all $p\geq 0$ using the notion of weak integral, see Proposition~\ref{prop:continuity_app} in the appendix.
%\end{remark}

\smallskip

\begin{proof}
To prove identity~\eqref{eq:id_byparts_bis} when $s>0$ is a regular value of $\ffi$, we start from the formula
\begin{equation}
\label{eq:diver}
{\rm div}_g  \bigg( \frac{|\na \ffi|_g^{p-1} \, \na \ffi   }{\sinh^n(\ffi)} \bigg) \, = \, \frac{  |\na \ffi|_g^{p-3}  
\Big(     (p-1)\,  \nana \ffi (\na\ffi, \!\na\ffi) \, + \, |\na\ffi|_g^2 \, \Deg\ffi \,
- \, n \, \coth(\ffi) \, |\na\ffi|_g^{4}  \Big) }{\sinh^n(\ffi)} \, ,
\end{equation}
which follows from a direct computation. Since $\ffi$ is analytic, its singular values are discrete (see~\cite{Soucek}). In particular all the big enough values are regular. Hence, we integrate the above formula by parts using the Divergence Theorem in $\{s < \ffi < S \}$, where $S$ is large enough so that we are sure that the level set $\{\ffi=S\}$ is regular. This gives
\begin{multline}
\label{eq:byparts_fin}
 \int\limits_{\{s <\ffi < S\}} \!\!\!\!\!\!
\frac{  |\na \ffi|_g^{p-3}  
\Big(     (p-1)\,  \nana \ffi (\na\ffi, \!\na\ffi) \, + \, |\na\ffi|_g^2 \, \Deg\ffi \, 
- \,n \, \coth(\ffi) \, |\na\ffi|_g^{4}  \Big) }{\sinh^n(\ffi)}
\,\,\rmd\mu_g \,\,=  \\
 \,=\,\,
\!\!\!\!
\int\limits_{\{\ffi=S\}}
\!\!\!
\frac{   |\na\ffi|_g^{p-1} \big\langle{\na\ffi}
\,\big|\,{\rm n}_g\big\rangle_g}{  \sinh^n(\ffi)  }       \,\,\rmd\sigma_g
\,\,+
\!\!\int\limits_{\{\ffi=s\}}
\!\!\!
\frac{    |\na\ffi|^{p-1} \big \langle{\na \ffi}
\,\big|\, {\rm n}_g\big\rangle_g}{  \sinh^n(\ffi)  }       \,\,\rmd\sigma_g \, ,
\end{multline}
where ${\rm n}_g$ is the outer unit normal. In particular, one has that ${\rm n}_g = -\na\ffi/|\na\ffi|_g$ on $\{\ffi=s\}$ and ${\rm n}_g =\na\ffi/|\na\ffi|_g$
on $\{\ffi=S\}$.
Therefore, if we prove that
\begin{equation}
\label{eq:asymp_div1}
\lim_{S \to +\infty} \!\!\int\limits_{\{\ffi=S\}} \!\!\! \frac{|\na \ffi|_g^p}{\sinh^n(\ffi)} \,\,\rmd \sigma_g \, = \, 0 \,, 
\end{equation}
then the statement of the proposition will follow at once. 
Form Lemma~\ref{le:grad_ffi}, we know that $|\na\ffi|_g\leq 1$, hence it is enough to prove that
$$
\lim_{S \to +\infty} \!\!\int\limits_{\{\ffi=S\}} \!\!\! \frac{\rmd \sigma_g}{\sinh^n(\ffi)} \,\, = \, 0 \,, 
$$
Rewriting this last limit in terms of $u,\go$, we find the equalities
$$
\lim_{t\to 1} \!\!\int\limits_{\{u=t\}} \!\!\! \frac{\sqrt{1-u^2}}{u^n} \,\,\rmd \sigma \, = \, 0\quad  (\hbox{case }\Lambda>0)\,,\qquad \lim_{t\to 1} \!\!\int\limits_{\{u=t\}} \!\!\! \sqrt{u^2-1} \,\,\rmd \sigma \, = \, 0 \quad  (\hbox{case }\Lambda<0)
$$ 
which are easily verified, since the level sets $\{u=t\}$ have finite $\mathscr{H}^{n-1}$-measure (because $u$ is analytic). This proves the limit~\eqref{eq:asymp_div1} and the thesis in the case in which $s$ is a regular value. 

In the case where $s>0$ is a singular value of $\ffi$, we need to apply a refined version of the Divergence Theorem in order to perform the integration by parts which leads to identity~\eqref{eq:byparts_fin}, namely Theorem~\ref{thm:div} in the appendix. The rest of the proof is then identical to what we have done for the regular case.

According to the notations of Theorem~\ref{thm:div}, we set 
\begin{eqnarray*}
X  =  \frac{|\na\ffi|_g^{p-1}\na \ffi}{\sinh^n(\ffi)} \qquad \hbox{and} \qquad E = \{ s < \ffi < S\}\, .
\end{eqnarray*}
so that $\pa E = \{\ffi = s \} \sqcup \{ \ffi = S \}$. 
It is clear that the vector field $X$ is Lipschitz for $p\geq 1$ and, by the results of~\cite{Federer}, we know that $\mathscr{H}^{n-1}(\pa E)$ is finite.
Moreover, from~\cite{Krantz}, we know that there exists an open $(n-1)$-submanifold $N\subseteq {\rm Crit}(\ffi)$ such that $\mathscr{H}^{n-1}(\pa E\setminus N)=0$.
Set $\Sigma=\pa E\cap ({\rm Crit}(\ffi)\setminus N)$ and $\Gamma=\pa E\setminus \Sigma$. We have $\mathscr{H}^{n-1}(\Sigma)=0$ by definition, while $\Gamma$ is the union of the regular part of $\pa E$ and of $N$, which are open $(n-1)$-submanifolds.
Therefore, the hypoteses of Theorem~\ref{thm:div} are satisfied, hence we can apply it to conclude that~\eqref{eq:byparts_fin} holds also on the non regular level sets.
%$$
%|\na X|\leq
%\left|\frac{\na|\na\ffi|^{p-1}\na\ffi}{\sinh(\ffi)}\right|
%\,+\,
%\left|\frac{|\na\ffi|^{p-1}\nana\ffi}{\sinh(\ffi)}\right|
%\,+\,
%\left|\frac{|\na\ffi|^{p-1}\na\ffi\na\ffi}{\sinh(\ffi)}\right|
%$$
%The second and the third summands are clearly continuous, thus bounded on the compact $\overline E$. For the first term, it holds
%$$
%\left|\frac{\na|\na\ffi|^{p-1}\na\ffi}{\sinh(\ffi)}\right|^2=\frac{(p-1)^2\,|\na\ffi|^{2(p-1)}}{\sinh^2(\ffi)}\big|\na|\na\ffi|\big|^2
%$$
\end{proof}

\subsection{Second integral identity.}

Now we want to exploit Lemma~\ref{le:grad_ffi} in order to obtain an integral inequality analogous to~\cite[Prop. 4.2]{Ago_Maz_2}.
We rewrite equation~\eqref{eq:DegW} as
\begin{equation}
\label{eq:Deg_grad_ffi}
\Deg |\na\ffi|^2_g\, -\left(\frac{1}{u}+(n+1) u\right)\langle \na |\na\ffi|_g^2 \,|\, \na\ffi\rangle_g= \, 2\, |\nana\ffi|_g^2  \, +\, 2\, n\, u^2\, |\na\ffi|_g^2\, \left(1-|\na \ffi|_g^2\right)\, .
\end{equation}
For every $p \geq 3$, we compute 
\begin{align*}
\na |\na \ffi|_\g^{p-1} & \,= \, \Big(\frac{p-1}{2}\Big)  |\na \ffi|_\g^{p-3} \, \na |\na \ffi|_\g^2 \, ,\\
\Deg |\na \ffi|_\g^{p-1} &\, = \, \Big(\frac{p-1}{2}\Big)  |\na \ffi|_\g^{p-3} \, {\Deg |\na \ffi|^2} +  (p-1)(p-3) \, |\na \ffi|_\g^{p-3} \, \big| \na |\na \ffi |_\g   \big|_\g^2 \, .
\end{align*}
We notice {en passant} that whenever $|\na \ffi|_g>0$ the above formul\ae\ make sense for every $p \geq 0$.
These  identities, combined with~\eqref{eq:Deg_grad_ffi}, lead to
\begin{multline}
\label{eq:Deg_grad_ffi_p}
\Deg |\na\ffi|^{p-1}_g\, -\left(\frac{1}{u}+(n+1) u\right)\langle \na |\na\ffi|_g^{p-1} \,|\, \na\ffi\rangle_g =
\\
= \, (p-1) |\na\ffi|_g^{p-3} \, \Big( |\nana\ffi|_g^2 \,+\, (p-3) \left|\na|\na\ffi|_g\right|_g^2 \, +\, n\, u^2\, |\na\ffi|_g^2\, \left(1-|\na \ffi|_g^2\right)\Big)\, .
\end{multline}

Obviously, for $p=3$, the above formula coincides with~\eqref{eq:Deg_grad_ffi}. 
If we define the function 
$$
\gamma = \gamma(\ffi)=
\begin{dcases}
\frac{1}{\sinh(\ffi)\cosh^{n+1}(\ffi)} & \mbox{(case $\Lambda>0$)}\,,
\\
\frac{1}{\cosh(\ffi)\sinh^{n+1}(\ffi)} & \mbox{(case $\Lambda<0$)}\,,
\end{dcases}
$$
then the equation above can be written as
\begin{multline}
\label{eq:div_p}
\mbox{div}_g\left(\gamma(\ffi)\, \na |\na\ffi|^{p-1}_g\right) \, = 
\\ = \, (p-1) |\na\ffi|_g^{p-3} \, \gamma(\ffi)\left(|\nana\ffi|_g^2 \,+\, (p-3) \left|\na|\na\ffi|_g\right|_g^2  \, + \, n\, u^2\, |\na\ffi|_g^2\, \left(1-|\na \ffi|_g^2\right)\right)\, .
\end{multline}
Note that the term on the right is always positive, thanks to Lemma~\ref{le:grad_ffi}.

Integrating by parts identity~\eqref{eq:div_p}, we obtain the following proposition, which is the main result of this section.

\begin{proposition}
\label{prop:cyl}
Let $(M^*,g,\ffi)$ be a solution to problem~\eqref{eq:pb_conf} satisfying Assumption~\ref{ass:conf}. Then, for every $s\in [0,+\infty)$ and $p\geq 3$
\begin{multline}
\label{eq:id_byparts} 
 \, \gamma(s)\!\!\! \int\limits_{\{\ffi=s\}}\!\!\! \left(|\na \ffi|_g^{p-1} \, \Hg -\,|\na \ffi|_g^{p-2}\,\Deg\ffi\right) \rmd\sigma_g
 \, = \,
\\
\, = \,
\int\limits_{\{\ffi>s\}}\!\!\! \gamma(\ffi)\, |\na\ffi|_g^{p-3}\,\left(|\nana\ffi|^2_g 
\,+\, (p-3) \left|\na|\na\ffi|_g\right|_g^2 
\,+\, n\, u^2\, |\na\ffi|^2_g\, \left(1-|\na \ffi|^2_g\right)\right)\,\rmd\mu_g\,.
\end{multline}
Moreover, if there exists $s_0\in (0,+\infty)$ such that 
%$\{ \ffi = s_0\} \subset M$ and 
\begin{equation}
\label{eq:condition_cyl}
\int\limits_{\{\ffi=s_0\}}
\!\!\! \left(|\na \ffi|_\g^{p-1}\,\Hg - |\na\ffi|_\g^{p-2} \Deg\ffi\right) \!
\,\rmd\sigma_{\!g}  \, \leq \, 0 \, ,
\end{equation}
then the manifold $(M^*,g)$ is isometric to one half round cylinder with totally geodesic boundary.
\end{proposition}

\begin{remark}
\label{rem:a_e}
Translating Remark~\ref{rem:uno_D} in terms of the conformally related quantities, it is easy to realize that the integral on the left hand side of~\eqref{eq:id_byparts} is well defined also when $s$ is a singular value of $\ffi$.
\end{remark}
%\begin{remark}
%\label{rem:analytic}
%We observe that since the {\em static solution} $(M, g_0,u)$ to problem~\eqref{eq:pb_D} (respectively~\eqref{eq:pb_A}) is analytic (see~\cite{Chr}), the solution $(M^*,g,\ffi)$ to problem~\eqref{eq:pb_conf} coming from $(M,g_0,u)$ through~\eqref{eq:g_D} (respectively~\eqref{eq:g_A}) and~\eqref{eq:ffi_D} (respectively~\eqref{eq:ffi_A})  is analytic as well. Hence, the conclusion of the rigidity statement in Proposition~\ref{prop:cyl} can be made stronger in the sense that if $(\{ \ffi \geq s_0\} , g)$ is isometric to one half round cylinder, then the entire manifold $(M^*,g)$ must be isometric to one half round cylinder and the corresponding {\em static solution} $(M, g_0,u)$ must be rotationally symmetric and thus isometric to the {\em de Sitter solution}~\eqref{eq:D} (respectively to the {\em anti-de Sitter solution}~\eqref{eq:A}).
%\end{remark}

%\begin{remark}
%Proposition~\ref{prop:cyl} actually holds for every $p\geq 2$, but the proof relies on a more sofisticated version of the divergence theorem (see Theorem~\ref{thm:div_2} and Proposition~\ref{prop:diff_app} in the appendix). 
%\end{remark}

\smallskip

\noindent For the seek of clearness, we rewrite more explicitly Proposition~\ref{prop:cyl}, distinguishing the two cases $\Lambda>0$, $\Lambda<0$.

\begin{corollary}[Case $\Lambda>0$]
Let $(M,g_0,u)$ be a static solution to problem~\eqref{eq:pb_D} satisfying Normalization~\ref{norm:D} and Assumption~\ref{ass:D}. Let $g$ be the metric defined in~\eqref{eq:g_D} and $\ffi$ be the smooth function defined in~\eqref{eq:ffi_D}.
Then, for every $s\in [0,+\infty)$
\begin{multline*}
\!\!\! \int\limits_{\{\ffi=s\}}\!\!\! \frac{|\na \ffi|_g^{p-1} \, \Hg -\,|\na \ffi|_g^{p-2}\,\Deg\ffi}{\sinh(s)\cosh^{n+1}(s)} \rmd\sigma_g\,=
\\
\,=\!\!
\int\limits_{\{\ffi>s\}}\!\!\! |\na\ffi|_g^{p-3}\,\frac{|\nana\ffi|^2_g 
\,+\, (p-3) \left|\na|\na\ffi|_g\right|_g^2 
\,+\, n \,\tanh^2(\ffi)\, |\na\ffi|^2_g\, \left(1-|\na \ffi|^2_g\right)}{\sinh(\ffi)\cosh^{n+1}(\ffi)}\,\rmd\mu_g
\, .
\end{multline*}
Moreover, if there exists $s_0\in (0,+\infty)$ such that 
%$\{ \ffi = s_0\} \subset M$ and 
\begin{equation*}
%\label{eq:condition_cyl}
\int\limits_{\{\ffi=s_0\}}
\!\!\! \left(|\na \ffi|_\g^{p-1}\,\Hg - |\na\ffi|_\g^{p-2} \Deg\ffi\right) \!
\,\rmd\sigma_{\!g}  \, \leq \, 0 \, ,
\end{equation*}
then $(M,\go,u)$ is isometric to the de Sitter solution.
\end{corollary}

\begin{corollary}[Case $\Lambda<0$]
Let $(M,\go,u)$ be a conformally compact static solution to problem~\eqref{eq:pb_A} satisfying Normalization~\ref{norm:A} and Assumption~\ref{ass:A}. Let $g$ be the metric defined in~\eqref{eq:g_A}, and $\ffi$ be the smooth function defined in~\eqref{eq:ffi_A}. 
Then, for every $s\in [0,+\infty)$
\begin{multline*}
\!\!\! \int\limits_{\{\ffi=s\}}\!\!\! \frac{|\na \ffi|_g^{p-1} \, \Hg -\,|\na \ffi|_g^{p-2}\,\Deg\ffi}{\sinh^{n+1}(s)\cosh(s)} \rmd\sigma_g
\,=
\\
=\!\!\int\limits_{\{\ffi>s\}}\!\!\! |\na\ffi|_g^{p-3}\,\frac{|\nana\ffi|^2_g 
\,+\, (p-3) \left|\na|\na\ffi|_g\right|_g^2 
\,+\, n\, \coth^2(\ffi)\, |\na\ffi|^2_g\, \left(1-|\na \ffi|^2_g\right)}{\sinh^{n+1}(\ffi)\cosh(\ffi)}\,\rmd\mu_g 
\, .
\end{multline*}
Moreover, if there exists $s_0\in (0,+\infty)$ such that 
%$\{ \ffi = s_0\} \subset M$ and 
\begin{equation*}
%\label{eq:condition_cyl}
\int\limits_{\{\ffi=s_0\}}
\!\!\! \left(|\na \ffi|_\g^{p-1}\,\Hg - |\na\ffi|_\g^{p-2} \Deg\ffi\right) \!
\,\rmd\sigma_{\!g}  \, \leq \, 0 \, ,
\end{equation*}
then $(M,\go,u)$ is isometric to the anti-de Sitter solution.
\end{corollary}

\smallskip

\begin{proof}[Proof of Proposition~\ref{prop:cyl}]
We start by considering the case where the level set $\{\ffi=s\}$ is regular. Arguing as in the proof of Proposition~\ref{prop:byparts}, we find that we can choose $S$ large enough to be sure that $\{\ffi=S\}$ is regular. Integrating by parts identity~\eqref{eq:div_p} in $\{s<\ffi<S\}$, we obtain
\begin{multline*}
(p-1)\!\!\!\!\!\int\limits_{\{s<\ffi<S\}}\!\!\!\!\!\! \gamma(\ffi)\,|\na\ffi|_g^{p-3}\, \Big(|\nana\ffi|_g^2
\,+\, (p-3) \left|\na|\na\ffi|_g\right|_g^2
\,+ \, n\, u^2\, |\na\ffi|_g^2\, \left(1-|\na \ffi|_g^2\right)\Big)\rmd\mu_g\,
=
\\
=
\int\limits_{\{\ffi=S\}}\!\!\! \gamma(\ffi)\, \langle \na|\na\ffi|_g^{p-1}\,|\, {\rm n}_g\rangle_g\, \rmd\sigma_g
+ 
\int\limits_{\{\ffi=s\}}\!\!\! \gamma(\ffi)\, \langle \na|\na\ffi|_g^{p-1} \,|\, {\rm n}_g\rangle_g\, \rmd\sigma_g\,.
\end{multline*}  
where ${\rm n}$ is the outer $g$-unit normal
of the set $\{s \leq \ffi \leq S \}$ at its boundary. In particular, one has that ${\rm n}_g = -\na\ffi/|\na\ffi|_g$ on $\{\ffi=s\}$ and ${\rm n}_g =\na\ffi/|\na\ffi|_g$
on $\{\ffi=S\}$. On the other hand, from the second formula in~\eqref{eq:formula_curvature} it is easy to deduce that
\begin{multline*}
\langle \na|\na\ffi|_g^{p-1}|\na\ffi\rangle_g
\, = \, 
\frac{p-1}{2}|\na\ffi|_g^{p-3}\langle \na|\na\ffi|_g^2|\na\ffi\rangle_g
\, =
\\
= \, 
(p-1) \,|\na\ffi|_g^{p-3} \nana\ffi(\na\ffi,\na\ffi)
\, = \,  
(p-1) \,|\na\ffi|_g^{p-3} \left(- \,|\na\ffi|_g^3 \, \Hg \, + \, |\na\ffi|_g^2 \,\Deg \ffi\right)\, . 
\end{multline*}
Therefore, we have obtained
\begin{multline}
\label{eq:int_part_fin_2}
\int\limits_{\{s<\ffi<S\}}\!\!\!\!\!\! \gamma(\ffi)\,|\na\ffi|_g^{p-3}\,\Big(|\nana\ffi|_g^2
\,+\, (p-3) \left|\na|\na\ffi|_g\right|_g^2
+\, n\, u^2\, |\na\ffi|_g^2\, \left(1-|\na \ffi|_g^2\right)\Big)\rmd\mu_g
\,=
\\
=\,
\gamma(s)\!\!\! \int\limits_{\{\ffi=s\}}\!\!\! \Big(|\na\ffi|_g^{p-1} \, \Hg -|\na\ffi|_g^{p-2} \Deg \ffi \Big)\,\rmd\sigma_g
-
\gamma(S)\!\!\! \int\limits_{\{\ffi=S\}}\!\!\! \Big(|\na\ffi|_g^{p-1} \, \Hg -|\na\ffi|_g^{p-2} \Deg \ffi \Big)\,\rmd\sigma_g\,.
\end{multline} 

In order to obtain identity~\eqref{eq:id_byparts} it is sufficient to show that the last term on the right hand side tends to zero as $S\to +\infty$. To this end, we first compute
\begin{equation*}
\!\!\!\!\lim_{S\to +\infty}\gamma(S)\!\!\!\! \int\limits_{\{\ffi=S\}}\!\!\! \Big(|\na\ffi|_g^{p-1} \, \Hg -|\na\ffi|_g^{p-2} \Deg \ffi \Big)\,\rmd\sigma_g=\lim_{S\to +\infty}-\gamma(S)\!\!\!\! \int\limits_{\{\ffi=S\}}\!\!\!\! |\na\ffi|_g^{p-2}\nana\ffi({\rm n}_g,{\rm n}_g)\,\rmd\sigma_g\,.
\end{equation*}
Now we recall that $|\na\ffi|_g\leq 1$ thanks to Lemma~\ref{le:grad_ffi}, and we use formul\ae~\eqref{eq:dedefidedeu_D},~\eqref{eq:dedefidedeu_A} to rewrite the limit above in terms of $u,\go$. In both the cases $\Lambda>0$ and $\Lambda<0$, we find that it is enough to prove 
\begin{equation}
\label{eq:lim_S}
\lim_{t\to 1} \int\limits_{\{u=t\}}\!\!\! \sqrt{|1-u^2|}\bigg[\Big(\frac{1-u^2}{u}\Big)\DD u({\rm n},{\rm n})+|\D u|^2\bigg]\,\rmd\sigma\,=\,0\,,
\end{equation}
where ${\rm n}=\D u/|\D u|$. Note that, for $t$ near enough to 1, the vector ${\rm n}$ is well defined. In fact, since the singular values of an analytic function are discrete (see~\cite{Soucek}), it is clear that the values near enough to 1 are regular.

Since $u$ is analytic, the level set $\{u=t\}$ has finite $\mathscr{H}^{n-1}$-measure (see~\cite{Federer}), thus the equality~\eqref{eq:lim_S} is straightforward. This completes the proof of the first part of the statement in the case where
$\{ \ffi =s\}$ is regular. 

In the case where $s>0$ is a singular value of $\ffi$, we need to apply a slightly refined version of the Divergence Theorem, namely Theorem~\ref{thm:div} in the appendix, in order to perform the integration by parts which leads to identity~\eqref{eq:int_part_fin_2}. The rest of the proof is identical to what we have done for the regular case. We set 
\begin{eqnarray*}
X \, = \,  \gamma(\ffi)\,\na|\na\ffi|_g^{p-1} 
\,=\, 
\Big(\frac{p-1}{2}\Big)\, \gamma(\ffi)\,|\na\ffi|_g^{p-3}\,\na|\na\ffi|_g^2
\qquad \hbox{and} \qquad E= \{ s < \ffi < S\}\, .
\end{eqnarray*}
so that $\pa E = \{\ffi = s \} \sqcup \{ \ffi = S \}$. As we have already observed, $\ffi$ is proper and analytic, hence the $(n-1)$-dimensional Hausdorff measure of $\pa E$ is finite. 
Moreover, it is clear that $X$ is Lipschitz for $p\geq 3$. 
From the results in~\cite{Krantz}, we know that there exists an open $(n-1)$-submanifold $N\subseteq {\rm Crit}(\ffi)$ such that $\mathscr{H}^{n-1}(\pa E\setminus N)=0$.
Set $\Sigma=\pa E\cap ({\rm Crit}(\ffi)\setminus N)$ and $\Gamma=\pa E\setminus \Sigma$.
We have $\mathscr{H}^{n-1}(\Sigma)=0$ by definition, while $\Gamma$ is the union of the regular part of $\pa E$ and of $N$, which are open $(n-1)$-submanifolds.
Therefore the hypoteses of Theorem~\ref{thm:div} are satisfied, hence, taking into account Remark~\ref{rem:a_e} and expression~\eqref{eq:div_p}, we have that identity~\eqref{eq:int_part_fin_2} holds true also in the case where $s$ is a singular value of $\ffi$. 

To prove the second part of the statement, we observe that from~\eqref{eq:id_byparts} and~\eqref{eq:condition_cyl} one immediately gets $\nana \ffi \equiv 0$ in $\{ \ffi \geq s_0 \}$. Since $\ffi$ is analytic, then $\nana \ffi \equiv 0$ on the whole $M^*$. In particular, $\Deg\ffi=0$ and, from the second equation in~\eqref{eq:pb_conf}, we find $|\na\ffi|_g^2\equiv 1$ on $M^*$. 

Consider now the case $\Lambda>0$. Substituting $\nana\ffi=0$ and $|\na\ffi|_g=1$ in equality~\eqref{eq:dedefidedeu_D}, we find $\DD u=-u\,\go$ on $M^*$. Since $u$ is analytic, the set ${\rm MAX}(u)$ is negligible, hence the equality $\DD u=-u\,\go$ holds on the whole $M=M^*\cup{\rm MAX}(u)$. Therefore, using the same arguments as in~\cite{Obata}, we deduce that $(M,\go)$ is an half-sphere, and translating this back in terms of the conformally related quantities, we easily find that $(M^*,g)$ is isometric to an half round cylinder.

In the case $\Lambda<0$ we proceed in a similar way. Substituting in equality~\eqref{eq:dedefidedeu_A}, we find $\DD u=u\go$ on $M^*$ and, with the same argument as above, we deduce that the same equation holds on the whole $M=M^*\cup{\rm MIN}(u)$. Then we can use \cite[Lemma~3.3]{Qing} to conclude that $(M,\go)$ is isometric to the hyperbolic space, from which we deduce that $(M^*,g)$ is an half round cylinder.
\end{proof}

\section{Proof of Theorem~\ref{thm:main_conf}}
\label{sec:proofs}
Building on the analysis of the previous section, we are now in the position to prove Theorem~\ref{thm:main_conf}, which in turn implies Theorem~\ref{thm:main_D} and Theorem~\ref{thm:main_A}.

\subsection{Continuity.}
\label{continuity}
We claim that under the hypotheses of Theorem~\ref{thm:main_conf} the function $\Phi_p$ is continuous, for $p\geq 1$. 
%for $p\geq 0$. To avoid technicalities, here we prove the continuity for $p\geq 1$. For an extension of this proof to every $p\geq 0$, see Proposition~\ref{prop:continuity_app} in the appendix. 

We first observe that since we are assuming that the boundary $\pa M$ is a regular level set of $\ffi$, the function $s \mapsto \Phi_p(s)$ can be described in term of an integral depending on the parameter $s$, provided $s \in [0, 2\ep)$ with $\ep>0$ sufficiently small. In this case, the continuous dependence on the parameter $s$ can be easily checked using standard results from classical differential calculus. Thus, we leave the details to the interested reader and we pass to consider the case where $s \in ( \ep, + \infty)$. Thanks to Proposition~\ref{prop:byparts}
one can rewrite expression~\eqref{eq:fip} as 
\begin{equation}
\label{eq:fip_div}
\Phi_p(s)\,
= \, - \, \sinh^n(s) \!\!\!\int\limits_{\{\ffi>s\}}
\!\!
\frac{  |\na \ffi|_g^{p-3}  
\Big(     (p-1)\,  \nana \ffi (\na\ffi, \!\na\ffi) \, + \, |\na\ffi|_g^2 \, \Deg\ffi\, 
- \, n\,\coth(\ffi) \, |\na\ffi|_g^{4}  \Big) }{\sinh^n(\ffi)}
\,\,\rmd\mu_g .
\end{equation}
It is now convenient to set 
\begin{equation}
\label{eq:density}
\mu^{(p)}_{\,g} (E) \,\, =  \, \int\limits_{E} 
\frac{  |\na \ffi|_g^{p-3}  
\Big(     (p-1)\,  \nana \ffi (\na\ffi, \!\na\ffi) \, + \, |\na\ffi|_g^2 \, \Deg\ffi \, 
- \, n \, \coth(\ffi) \, |\na\ffi|_g^{4}  \Big) }{\sinh^n(\ffi)}
%{\rm div}_{\!g}  \bigg( \frac{|\na \ffi|_g^{p-1} \, \na \ffi   }{\sinh(\ffi)} \bigg) 
\,\, \rmd\mu_g \, ,
\end{equation}

for every $\mu_g$-measurable set $E \subseteq \{ \ffi >  \ep\}$. It is then clear that for $p \geq 1$ the measure $\mu_{\,g}^{(p)}$ is absolutely continuous with respect to $\mu_g$, since $|\na \ffi|_g^{p-3} \,\nana\ffi (\na \ffi, \na\ffi)\leq |\na \ffi|_g^{p-1} |\nana \ffi|_g $ and $|\nana \ffi|_g$ is bounded (this is an easy consequence of equalities~\eqref{eq:|nanaffi|_D},~\eqref{eq:|nanaffi|_A}).
%It is also worth pointing out that, under the hypotheses of Theorem~\ref{thm:refined_conf}, Proposition~\ref{prop:cyl_bis} is in force and thus the same conclusion holds for every $p \geq 0$. 

In view of~\eqref{eq:fip_div}, the function $s \mapsto \Phi_p(s)$ can be interpreted as the repartition function of the measure defined in~\eqref{eq:density}, up to the smooth factor $- \sinh^n(s)$. Thus, $s \mapsto \Phi_p(s)$ is continuous if and only if the assignment 
$$
s \longmapsto \mu^{(p)}_{\, g} (\{ \ffi > s \})
$$ 
is continuous. Thanks to~\cite[Proposition 2.6]{Amb_DaP_Men} and thanks to the fact that $\mu_{\,g}^{(p)}$ is absolutely continuous with respect to $\mu_g$, proving the continuity of the above assignment 
is equivalent to  checking that $\mu_g (\{ \ffi = s \}) = 0$ for every $s>\ep$. On the other hand, the Hausdorff dimension of the level sets of $\ffi$ is at most $n-1$, as it follows from the results in~\cite{Federer}. Hence, they are negligible with respect to the full $n$-dimensional measure. This proves the continuity of $\Phi_p$ for $p\geq 1$ under the hypotheses of Theorem~\ref{thm:main_conf}.
%and for every $p \geq 0$ under the hypotheses of Theorem~\ref{thm:refined_conf}.

\subsection{Monotonicity of $\Phi_1(s)$.}

From the second equation in problem~\eqref{eq:pb_conf} and from Lemma~\ref{le:grad_ffi}, we get
$$
\Deg\ffi\,=\,-nu\,(1-|\na\ffi|_g^2)\,\leq\, 0\,.
$$
Integrating this inequality in $\{s\leq\ffi\leq S\}$, we get
\begin{equation}
\label{eq:int_lapl}
\int\limits_{\{s\leq\ffi\leq S\}}\!\!\!\Deg\ffi\,\rmd\sigma_g\,\leq\, 0\,.
\end{equation}
Suppose that $\{\ffi=s\}$ and $\{\ffi=S\}$ are regular levels (the case in which they are singular can be handled in the same way as in the proofs of Propositions~\ref{prop:byparts} and~\ref{prop:cyl}). Then, applying the divergence theorem to inequality~\eqref{eq:int_lapl}, we easily obtain $\Phi_1(S)\leq\Phi_1(s)$, for every $s< S$. 

Moreover, if the equality holds for some values of $s,S$, then $|\na\ffi|_g\equiv 1$ on $\{s\leq\ffi\leq S\}$ and, since $\ffi$ is analytic, we have $|\na\ffi|_g\equiv 1$ on the whole $M^*$. Plugging this information inside formula~\eqref{eq:DegW}, we find that $\nana\ffi\equiv 0$ on $M^*$. With the same argument used in the proof of the rigidity statement in Proposition~\ref{prop:cyl}, we deduce that $(M^*,g,\ffi)$ is an half round cylinder.
This proves Theorem~\ref{thm:main_conf}-(ii).

\subsection{Differentiability.}
\label{sec:diff} 
We now turn our attention to the issue of the differentiability of the functions $s \mapsto \Phi_p(s)$.
As already observed in the previous subsection, we are assuming that the boundary $\pa M$ is a regular level set of $\ffi$ so that the function $s \mapsto \Phi_p(s)$ can be described in term of an integral depending on the parameter $s$, provided $s \in [0, 2\ep)$ with $\ep>0$ sufficiently small. Again, the differentiability in the parameter $s$ can be easily checked in this case, using standard results from classical differential calculus. Leaving the details to the interested reader, we pass to consider the case where $s \in (\ep, + \infty)$. We start by noticing that for every $p\geq 2$ the function
\begin{equation*}
\frac{  |\na \ffi|_g^{p-4}  
\Big(  (p-1)\,  \nana \ffi (\na\ffi, \!\na\ffi)  \, + \, |\na\ffi|_g^2 \, \Deg\ffi \, 
- \,n\, \coth(\ffi) \, |\na\ffi|_g^{4}  \Big) }{\sinh^n(\ffi)}
\end{equation*}
has finite integral in $\{ \ffi > s\}$, for every $s> \ep$. Hence, we can apply the coarea formula to expression~\eqref{eq:fip_div}, obtaining
\begin{align}
\label{eq:fip_coa}
\nonumber \Phi_p(s)\,& =\,
-\, \sinh^n(s)  \int\limits_{\{\tau>s\}}  \int\limits_{\{\ffi =\tau \}}\!\!\! \frac{(p-1)\, |\na\ffi|_g^{p-4}\, \nana \ffi (\na\ffi, \!\na\ffi)  \, + \, |\na\ffi|_g^{p-2} \, \Deg\ffi \, 
- \, n\, \coth(\ffi) \, |\na\ffi|_g^{p}}{\sinh^n(\ffi)}
\,\,\rmd\sigma_{\!g} \,\, \rmd \tau \\
\nonumber & = \, 
\sinh^n(s)  \int\limits_{\{\tau>s\}}  \int\limits_{\{\ffi =\tau \}}\!\!\! \frac{(p-1)\, |\na\ffi|_g^{p-1}\, \Hg \, - \, p \, |\na\ffi|_g^{p-2} \, \Deg\ffi \, 
+ \,n\, \coth(\ffi) \, |\na\ffi|_g^{p}}{\sinh^n(\ffi)}
\,\,\rmd\sigma_{\!g} \,\, \rmd \tau 
\\
&  = \, \sinh^n(s) \!\!\!\! \int\limits_{\{\tau>s\}} \!\!\!
\Bigg(\,\, \int\limits_{\{\ffi =\tau \}}\!\!\!\!   \frac{\, (p-1)|\na\ffi|_g^{p-1}\, \Hg \, - \, p \, |\na\ffi|_g^{p-2} \, \Deg\ffi}{\sinh^n(\tau)} \,\rmd\sigma_{\!g}  \,\, 
+ \,\, n\, \frac{\coth(\tau)}{\sinh^n(\tau)}  \, \Phi_p(\tau)
\Bigg)  \rmd \tau \,,
\end{align}
where in the second equality we have used~\eqref{eq:formula_curvature} and in the third equality we have used the definition of $\Phi_p$ given by formula~\eqref{eq:fip}. By the Fundamental Theorem of Calculus, we have that if the function
\begin{equation*}
%\label{eq:integrand}
\tau \,\, \longmapsto \,\, \int\limits_{\{\ffi =\tau \}}\!\!\!\!   \frac{\, (p-1)|\na\ffi|_g^{p-1}\, \Hg \, - \, p \, |\na\ffi|_g^{p-2} \, \Deg\ffi}{\sinh^n(\tau)} \,\rmd\sigma_{\!g}  \,\, 
+ \,\, n\, \frac{\coth(\tau)}{\sinh^n(\tau)}  \, \Phi_p(\tau)
\end{equation*}
is continuous, then $\Phi_p$ is differentiable. 
Since we have already discussed in Subsection~\ref{continuity} the continuity of $s \mapsto \Phi_p(s)$,
%is continuous for every $p \geq 1$, 
we only need to discuss
%, for every $p \geq 3$, 
the continuity of the assignment 
\begin{align}
\nonumber
\tau \,\, \longmapsto & \!\!  \int\limits_{\{\ffi =\tau \}}\!\!\!\!   \frac{\, (p-1)|\na\ffi|_g^{p-1}\, \Hg \, - \, p \, |\na\ffi|_g^{p-2} \, \Deg\ffi}{\sinh^n(\tau)} \,\rmd\sigma_{\!g} 
\\
\nonumber
\,\, = \, & \, (p-1)\frac{1}{\gamma(\tau)\sinh^n(\tau)}\!\!  \int\limits_{\{\ffi> \tau \}} \!\!\gamma(\ffi)\,
 |\na \ffi|_\g^{p-3}  
\Big(   \,  \big|\nana \ffi\big|_\g^2
 +  \, (p-3) \, \big| \na |\na \ffi |_\g   \big|_\g^2 -\, u \, |\na\ffi|_g^2\,\Deg\ffi  \, \Big)  
\,\rmd\mu_g \, +
\\
\label{eq:integrand2}
& \qquad\qquad\qquad\qquad\qquad\qquad\qquad\qquad\qquad\qquad\qquad\qquad\qquad+ \, n\,\frac{u(\tau)}{\sinh^n(\tau)}\big(\Phi_{p-2}(\tau)-\Phi_p(\tau)\big)\,.
\end{align}
We note that the above equality follows from formula $\Deg\ffi=-nu(1-|\na\ffi|_g^2)$ in problem~\eqref{eq:pb_conf} and from the integral identity~\eqref{eq:id_byparts} in Proposition~\ref{prop:cyl}, which is in force under the hypotheses of Theorem~\ref{thm:main_conf}-(iii).
%or from Proposition~\ref{prop:cyl_bis}, which is in force under the hypotheses of Theorem~\ref{thm:refined_conf}. 
In analogy with~\eqref{eq:density} it is natural to set 
\begin{equation*}
\bar\mu^{(p)}_{\,g} (E) \,\, =  \, \int\limits_{E} 
\gamma(\ffi)\, |\na \ffi|_\g^{p-3}  
\Big(   \,  \big|\nana \ffi\big|_\g^2
 +  \, (p-3) \, \big| \na |\na \ffi |_\g   \big|_\g^2  -\, u \, |\na\ffi|_g^2\,\Deg\ffi  \, \Big)   
\,\, \rmd\mu_g \, ,
\end{equation*}
for every $\mu_g$-measurable set $E \subseteq \{ \ffi > \ep\}$. It is now clear that for $p \geq 3$ the measure $\overline{\mu}_{\,g}^{(p)}$ is absolutely continuous with respect to $\mu_g$.
%, and that under the assumptions of Theorem~\ref{thm:refined_conf} the same conclusion holds for every $p \geq 0$. 
Hence, using the same reasoning as in Subsection~\ref{continuity}, we deduce that the assignment~\eqref{eq:integrand2} is continuous. In turn, we obtain the differentiability of $\Phi_p$ for $p\geq 3$, under the hypotheses of Theorem~\ref{thm:main_conf}.
%, and for every $p \geq 0$, under the hypotheses of Theorem~\ref{thm:refined_conf}. 
Finally, using~\eqref{eq:fip_coa} and~\eqref{eq:id_byparts}, a direct computation shows that 
\begin{align}
\label{eq:fip_mono}
\nonumber \Phi_p'(s) \,\, 
=   \,\, &\int\limits_{\{\ffi=s\}} \!\!\! \left( -(p-1) |\na\ffi|_g^{p-1} \, \Hg + p\, |\na\ffi|_g^{p-2} \, \Deg\ffi \right)\rmd\sigma_g 
\\
\nonumber
= \,\,&- \,(p-1) \!\! \int\limits_{\{\ffi> s\}} \!\!
\frac{\gamma(\ffi)}{\gamma(s)}  \, |\na \ffi|_\g^{p-3}  
\Big(   \,  \big|\nana \ffi\big|_\g^2
 +  \, (p-3) \, \big| \na |\na \ffi |_\g   \big|_\g^2 -\, u \, |\na\ffi|_g^2\,\Deg\ffi  \, \Big)       
\,\,\rmd\mu_g +
\\
& \qquad\qquad\qquad\qquad\qquad\qquad\qquad\qquad\qquad\qquad\qquad\quad\ \  + \int\limits_{\{\ffi=s\}} |\na\ffi|_g^{p-2} \, \Deg\ffi \,\, \rmd\sigma_g \, .
\end{align}
The monotonicity and the rigidity statements in Theorem~\ref{thm:main_conf}-(iii)
% and Theorem~\ref{thm:refined_conf}-(ii) 
are now consequences of Proposition~\ref{prop:cyl}.
%and Proposition~\ref{prop:cyl_bis}, respectively.

%We have proved the differentiability and monotonicity of $\Phi_p$ for all $p\geq 3$. For an extension of the proof to the values $2\leq p<3$, see Proposition~\ref{prop:diff_app} in the appendix. 

\subsection{The second derivative.}

\label{sub:second derivative}
To complete our analysis, we need to prove statement (iii) in Theorem~\ref{thm:main_conf}. To this aim, we observe from~\eqref{eq:formula_curvature} and the first equation of problem~\eqref{eq:pb_conf} that
\begin{align*}
%\label{eq:Phis_over_s}
\frac{\Phi_p'(s)}{s} \,\,  &= \,\, \frac{1}{s}\! \int\limits_{\{\ffi = s \}}\!\!\!\!  \left( -(p-1) |\na\ffi|_g^{p-1} \, \Hg + p\, |\na\ffi|_g^{p-2} \, \Deg\ffi \right)\rmd\sigma_g
\\
&= \,\, \frac{u(s)}{s\left[1-(n-1)\,u(s)^2\right]} \int\limits_{\{\ffi =s\}} \!\!\!\!
|\na \ffi|_g^{p-2} \Big[(p-1)\,\Ricg (\nu_g, \nu_g) - n\big(p-(n-1)u^2\big)\left(1-|\na\ffi|_g^2\right)\Big]
\,\,\rmd\sigma_g \,.
\end{align*}
Taking the limit as $s \to 0^+$ and using Assumption~\ref{ass:conf}, we obtain~\eqref{eq:der2_fip}. In the case $\Lambda<0$, one also needs to recall that $\lim_{s\to 0^+}|\na\ffi|_g=1$. This is an easy consequence of formula~\eqref{eq:tilde_R} (see also the proof of Lemma~\ref{le:tot_geod_BGH_A}-(i)).
To prove the rigidity statement in the case $\Lambda>0$, we observe that
\begin{align*}
\Phi_p''(0)\,&=\,\lim_{s\to 0^+}\frac{1}{s}\! \int\limits_{\{\ffi = s \}}\!\!\!\! \left( -(p-1) |\na\ffi|_g^{p-1} \, \Hg + p\, |\na\ffi|_g^{p-2} \, \Deg\ffi \right)\rmd\sigma_g
\\
&=\,\lim_{s\to 0^+} \left[-\left(\frac{p-1}{s}\right)\!\int\limits_{\{\ffi=s\}}\!\!\!\!  \left( |\na\ffi|_g^{p-1} \, \Hg -\, |\na\ffi|_g^{p-2} \, \Deg\ffi \right)\rmd\sigma_g +\frac{1}{s}\!\!\int\limits_{\{\ffi=s\}}\!\!\!\!   |\na\ffi|_g^{p-2} \, \Deg\ffi\,\, \rmd\sigma_g \right] 
\\
&=\,\lim_{s\to 0^+} \left[-\left(p-1\right)\!\!\!\int\limits_{\{\ffi=s\}}\!\!\!\!  \left( \frac{|\na\ffi|_g^{p-1} \, \Hg -\, |\na\ffi|_g^{p-2} \, \Deg\ffi}{\sinh(s)\cosh^{n+1}(s)} \right)\rmd\sigma_g-n\,\frac{u(s)}{s}\!\!\int\limits_{\{\ffi=s\}}\!\!\!\!   |\na\ffi|_g^{p-2} \left(1-|\na\ffi|_g^2\right)\, \rmd\sigma_g \right],
\end{align*} 
and we conclude using Proposition~\ref{prop:cyl}.

To understand why a similar rigidity statement does not hold in the case of a negative cosmological constant, we observe that a computation analogous to the one above gives
$$
\Phi_p''(0)
\,=\,
\lim_{s\to 0^+} \left[-\left(p-1\right)s^n\!\!\!\!\int\limits_{\{\ffi=s\}}\!\!\!\!  \left( \frac{|\na\ffi|_g^{p-1} \, \Hg -\, |\na\ffi|_g^{p-2} \, \Deg\ffi}{\sinh(s)^{n+1}\cosh(s)} \right)\rmd\sigma_g-n\frac{u(s)}{s}\!\!\int\limits_{\{\ffi=s\}}\!\!\!\!   |\na\ffi|_g^{p-2} \left(1-|\na\ffi|_g^2\right)\, \rmd\sigma_g \right].
$$
Therefore, if the equality holds in~\eqref{eq:der2_fip}, one inferes
$$
\lim_{s\to 0^+} s^n\!\!\!\!\int\limits_{\{\ffi=s\}}\!\!\!\!  \left( \frac{|\na\ffi|_g^{p-1} \, \Hg -\, |\na\ffi|_g^{p-2} \, \Deg\ffi}{\sinh(s)^{n+1}\cosh(s)} \right)\rmd\sigma_g
\,=\, 0\,,
$$
and this is not sufficient to use Proposition~\ref{prop:cyl} to deduce the rotational simmetry of the solution.
\underline{}

%To prove the rigidity statement in the case $\Lambda<0$, we proceed in a similar way, and we get
%\begin{align*}
%\frac{\Phi_p'(s)}{s^{n+1}}\,&=\,\lim_{s\to 0^+}\frac{1}{s^{n+1}}\! \int\limits_{\{\ffi = s \}}\!\!\!\! \left( -(p-1) |\na\ffi|_g^{p-1} \, \Hg + p\, |\na\ffi|_g^{p-2} \, \Deg\ffi \right)\rmd\sigma_g
%\\
%&=\,\lim_{s\to 0^+} \left[-\left(\frac{p-1}{s^{n+1}}\right)\!\int\limits_{\{\ffi=s\}}\!\!\!\!  \left( |\na\ffi|_g^{p-1} \, \Hg -\, |\na\ffi|_g^{p-2} \, \Deg\ffi \right)\rmd\sigma_g +\frac{1}{s^{n+1}}\!\!\int\limits_{\{\ffi=s\}}\!\!\!\!   |\na\ffi|_g^{p-2} \, \Deg\ffi\,\, \rmd\sigma_g \right] 
%\\
%&=\,\lim_{s\to 0^+} \left[-\left(p-1\right)\!\!\!\int\limits_{\{\ffi=s\}}\!\!\!\!  \left( \frac{|\na\ffi|_g^{p-1} \, \Hg -\, |\na\ffi|_g^{p-2} \, \Deg\ffi}{\sinh^{n+1}(s)\cosh(s)} \right)\rmd\sigma_g-n\frac{u(s)}{s^{n+1}}\!\!\int\limits_{\{\ffi=s\}}\!\!\!\!   |\na\ffi|_g^{p-2} \left(1-|\na\ffi|_g^2\right)\, \rmd\sigma_g \right]
%\\
%&\leq \,\lim_{s\to 0^+} \left[-n\frac{1}{s^{n+2}}\!\!\int\limits_{\{\ffi=s\}}\!\!\!\!   |\na\ffi|_g^{p-2} \left(1-|\na\ffi|_g^2\right)\, \rmd\sigma_g \right]\, ,
%\end{align*} 
%
%where again we have used Proposition~\ref{prop:cyl}.

\section*{Appendix}  
% use *-form to suppress numbering  

\renewcommand{\thesection}{A}
\renewcommand{\thesubsection}{A}
\numberwithin{equation}{subsection}
\renewcommand{\theequation}{\thesubsection-\arabic{equation}}

\subsection{Technical results}

This appendix will be dedicated to the proof of the technical results that we have used in our work.
Specifically, we will give a complete proof of Theorems~\ref{thm:estimate_D},~\ref{thm:estimate_A} (for the ease of reference, we have restated them here as Theorem~\ref{thm:MAX} and Theorem~\ref{thm:MIN}), we will prove an estimate on the static solution near the conformal boundary in the case $\Lambda<0$, and we will state the version of the divergence theorem that we have used in the proofs of Propositions~\ref{prop:byparts},~\ref{prop:cyl}.

\begin{theoremapp}
\label{thm:MAX}
Let $(M,\go,u)$ be a solution of \eqref{eq:pb_D} satisfying Assumption~\ref{ass:D}.
The set ${\rm MAX}(u)$ is discrete (and finite) and 
\begin{equation}
\label{eq:limitbehavior_Up}
\liminf_{t\to 1^-} \,U_p(t)\,\,\geq\,\, |{\rm MAX}(u)|\,|\Sph^{n-1}|\,,
\end{equation}
for every $0\leq p\leq n-1$.
\end{theoremapp}

In the proof of this theorem, we will need the following result, that will be proven later.

\begin{propositionapp}
\label{prop:lim_estimate}
Let $(M,\go,u)$ be a solution of \eqref{eq:pb_D} and let $y_0\in {\rm MAX}(u)$. Then for every $d>0$ it holds
\begin{equation}
\label{eq:cond_thesis}
\liminf_{t\to 1^-} \bigg(\frac{1}{1-t^2}\bigg)^{n-1}\int_{\{u=t\}\cap B_{d}(y_0)} |\D u|^{n-1} \,\rmd\sigma
\,\,\geq\,\,
|\Sph^{n-1}|
\end{equation}
\end{propositionapp}

We first show how to use this result to prove Theorem~\ref{thm:MAX}.

\begin{proof}[Proof of Theorem~\ref{thm:MAX}]
First we notice that the functions $U_p(t)$ can be written as follows
$$
U_p(t)
\,=\,
\bigg(\frac{1}{1-t^2}\bigg)^{\frac{n-1}{2}}\int_{\{u=t\}} \bigg[\frac{|\D u|^2}{1-u^2}\bigg]^{\frac{p}{2}} \,\rmd\sigma\,.
$$
From formula~\eqref{eq:Dusign_D} in Theorem~\ref{thm:main_D}, we have that the term in square braquet is less or equal to 1. Thus, for every $p\leq n-1$, we have
$$
\bigg[\frac{|\D u|^2}{1-u^2}\bigg]^{\frac{p}{2}} 
\,\geq\,
\bigg[\frac{|\D u|^2}{1-u^2}\bigg]^{\frac{n-1}{2}}\,,
$$
hence $U_p(t)\geq U_{n-1}(t)$ and, in particular
\begin{equation}
\label{eq:p>n-1}
\liminf_{t\to 1^-}\, U_{p}(t)
\,\geq\,
\liminf_{t\to 1^-}\, U_{n-1}(t)\,,
\end{equation}
so it is enough to prove the inequality~\eqref{eq:limitbehavior_Up} for $p=n-1$.

Now we pass to analyze the set ${\rm MAX}(u)$. Suppose that it contains an infinite number of points. Then for each $k\in\N$ we can consider $k$ points in ${\rm MAX}(u)$. Let $2 d$ be the minimum of the distances between our points. Applying Proposition~\ref{prop:lim_estimate} in a neighborhood of radius $d$ of each of these points, we obtain
$$
\liminf_{t\to 1^-} \, U_{n-1}(t)
\,\geq\,
k\,|\Sph^{n-1}|\,.
$$
Since this is true for every $k\in\N$, we conclude that $\lim_{t\to 1^-} U_{n-1}(t)=+\infty$ and, using \eqref{eq:p>n-1}, we find that $\lim_{t\to 1^-} U_1(t)=+\infty$. But this is impossible, since from the monotonicity of $U_1(t)$ (stated in Theorem~\ref{thm:main_D}-(ii)) we know that
$$
\lim_{t\to 1^-} U_1(t)\,\leq\, U_1(0)\,=\, |\pa M|\,.
$$ 

Therefore ${\rm MAX}(u)$ contains only a finite number of points. 
Repeating the argument above with $k=|{\rm MAX}(u)|$ we obtain the inequality in the thesis.
\end{proof}

Now we turn to the proof of Proposition~\ref{prop:lim_estimate}, that will be done in various steps. Our strategy consists in choosing a suitable neighborhood of the point $y_0$ where we are able to control the quantities in our limit, and then proceed to estimate them. 

\begin{notation}
Here and throughout the paper, we agree that for $f \in {\mathscr C}^{\infty}(M)$, $\tau \in \R$ and $k \in \N$ it holds
\begin{equation*}
f \, = \, o_k(|x|^{-\tau}) \quad \iff \quad \sum_{|J|\leq k} |x|^{\tau + |J|} \, \big|\pa^J \!f\big| \, = \, o(1) \, ,\quad\quad \hbox{as $|x| \to + \infty$} \,,
\end{equation*}
where the $J$'s are multi-indexes.
\end{notation}

Consider a normal set of coordinates $(x^1,\dots,x^n)$ in $B_d(y_0)$, that diagonalize the hessian in $y_0$. 
Note that, since $y_0$ is a maximum of $u$, the derivatives $\pa^2_{\alpha}u_{|_{y_0}}$ are non positive numbers for all $\alpha=1,\dots,n$, hence it makes sense to introduce the quantities $\lambda_\a^2=-\pa^2_\a u_{|_{y_0}}$ for $\alpha=1,\dots,n$. Since $\De u=-nu$, we have $\sum_{\a=1}^n \lambda_\a^2=n$. In particular, at least one of the $\lambda_\a$'s is different from zero. 
We have the following Taylor expansion of $u$ in a neighborhood of $y_0$
\begin{equation}
\label{eq:u_taylor}
u\,=\,1\,-\,\frac{1}{2}\,\sum_{\a=1}^n \big[\,\lambda_\a^2 \cdot\,|x^\a|^2\,\big]
\,+\,o_2\!\left(|x|^2\right)\,,
\end{equation}

From \eqref{eq:u_taylor} we easily compute
\begin{equation}
\label{eq:Du_taylor}
|\D u|^2\,=\,\sum_{i=1}^n \big[\lambda_\a^4 \cdot\,|x^\a|^2\,\big]+o_1\!\left(|x|^2\right)\,.
\end{equation}
Now we consider polar coordinates $(r,\theta^1,\dots,\theta^{n-1})$, where $\theta=(\theta^1,\dots,\theta^{n-1})\in\R^{n-1}$ are stereographic coordinates on $\Sph^{n-1}\setminus\{north\ pole\}$. 

\begin{lemmaapp}
With respect to the coordinates $(r,\theta)=(r,\theta^1,\dots,\theta^{n-1})$, the metric $\go$ writes as
\begin{equation}
\label{eq:gr_polar}
\go\,=\,dr\otimes dr+r^2 g_{\Sph^{n-1}}+\sigma\, dr\otimes dr+\sigma_i\left(dr \otimes d\theta^i+d\theta^i \otimes dr\right)+\sigma_{ij}\, d\theta^i\otimes d\theta^j
\end{equation} 
where $\sigma=o_2(r)$, $\sigma_i=o_2(r^2)$, $\sigma_{ij}=o_2(r^3)$, as $r\to 0^+$.
\end{lemmaapp}

\begin{proof}
To ease the notation, in this proof we use the Einstein summation convention.
It is known that, with respect to the normal coordinates $(x^1,\dots,x^n)$ the metric $\go$ writes as
$$
\go=(\delta_{\alpha\beta}+\eta_{\alpha\beta})d x^\alpha\otimes d x^\beta
$$
where $\delta_{\alpha\beta}$ is the Kronecker delta and $\eta_{\alpha\beta}=o_2(r)$ (actually, the term $\eta_{\alpha\beta}$ can be estimated better, but this is enough for our purposes).

Moreover, it is easy to check that the quantities $\phi^\alpha\,=\,x^{\alpha}/r$ are smooth functions of the coordinates $(\theta^1,\dots,\theta^{n-1})$ only, and that
$$
\delta_{\alpha\beta}\,\frac{\pa\phi^{\alpha}}{\pa\theta^i}\, \frac{\pa\phi^{\beta}}{\pa\theta^j} \, d\theta^i \otimes d\theta^j=g_{\Sph^{n-1}} \, .
$$
From $r^2=\delta_{\alpha\beta}\,x^{\alpha}x^{\beta}$ one also finds the equality $\delta_{\alpha\beta}\,\phi^{\alpha}\phi^{\beta}=1$. Deriving it with respect to $\theta^i$ we get 
$$
\delta_{\alpha\beta}\,\frac{\pa\phi^{\alpha}}{\pa\theta^i}\,\phi^{\beta}\,=\,0\, ,\qquad  \mbox{for all } i=1,\dots,n-1
$$
We are now ready to compute
\begin{align}
\nonumber\go &=(\delta_{\alpha\beta}+\eta_{\alpha\beta})\, dx^{\alpha}\otimes dx^{\beta} 
\\
\nonumber&=(1+\eta_{\alpha\beta}\phi^{\alpha}\phi^{\beta})dr\otimes dr+
(\delta_{\alpha\beta}+\eta_{\alpha\beta})r^2\,\frac{\pa\phi^{\alpha}}{\pa\theta^i}\, \frac{\pa\phi^{\beta}}{\pa\theta^j} \, d\theta^i \otimes d\theta^j+\eta_{\alpha\beta}\phi^{\alpha}\frac{\pa\phi^{\beta}}{\pa\theta^i}\,r
\left(dr \otimes d\theta^i+d\theta^i \otimes dr\right)
\\
%\label{eq:go_polarexp}
\nonumber
&=dr\otimes dr+r^2 g_{\Sph^{n-1}}+\sigma\, dr\otimes dr+\sigma_i\left(dr \otimes d\theta^i+d\theta^i \otimes dr\right)+\sigma_{ij}\, d\theta^i\otimes d\theta^j \, ,
\end{align}
where $\sigma,\sigma_i,\sigma_{ij}$ are infinitesimals of the wished order.
\end{proof}

We can rewrite formul\ae~\eqref{eq:u_taylor},~\eqref{eq:Du_taylor} in terms of $(r,\theta)$ as
\begin{align}
\label{eq:u_taylor2}
u(r,\theta)\,&=\,1\,-\,\frac{r^2}{2}\,\sum_{\a=1}^n \big[\,\lambda_\a^2\, |\phi^\a|^2(\theta)\,\big]
\,+\,w(r,\theta)\,,
\\
\label{eq:Du_taylor2}
|\D u|^2(r,\theta)\,&=\,r^2\sum_{\a=1}^n \big[\,\lambda_\a^4\, |\phi^\a|^2(\theta)\,\big]+h(r,\theta)\,.
\end{align}
where $w(r,\theta)=o\!\left(r^2\right)$, $h(r,\theta)=o\!\left(r^2\right)$. Moreover, since we know from~\eqref{eq:u_taylor} that $\pa w/\pa x^\a=o(r)$ for any $\a$, we have the following estimates on the order of the derivatives of $w$ with respect to $(r,\theta)$
\begin{align}
\label{eq:est_dewder}
\frac{\pa w}{\pa r}(r,\theta)\,&=\,
\sum_{\a=1}^n\Big[\frac{\pa w}{\pa x^\a}(r,\theta)\,\frac{\pa x^\a}{\pa r}(r,\theta)\Big]\,=\,
\sum_{\a=1}^n\Big[\frac{\pa w}{\pa x^\a}(r,\theta)\,\phi^\a(\theta)\Big]\,=\,o(r)\,,  &\hbox{as }r\to 0^+
\\
\label{eq:est_dewdetheta}
\frac{\pa w}{\pa \theta^j}(r,\theta)\,&=\,
\sum_{\a=1}^n\Big[\frac{\pa w}{\pa x^\a}(r,\theta)\,\frac{\pa x^\a}{\pa\theta^j}(r,\theta)\Big]\,=\,
\sum_{\a=1}^n\Big[\frac{\pa w}{\pa x^\a}(r,\theta)\,r\,\frac{\pa\phi^\a}{\pa\theta^j}(\theta)\Big]\,=\,o(r^2)\,, &\hbox{as }r\to 0^+
\end{align}

To estimate the limit in~\eqref{eq:cond_thesis}, we need to rewrite the set $\{u=t\}\cap B_d(y_0)$ and the density $\sqrt{\det({\go}_{|_{\{u=t\}\cap B_d(y_0)}})}$ as functions of our coordinates.  
In order to do so, it will prove useful to restrict our neighborhood $B_d(y_0)$ to a smaller domain where we have a better characterization of the level set $\{u=t\}$.
In this regard, it is convenient, for any $\ep>0$, to define the set
$$
C_\ep\,=\,\Big\{\theta=(\theta^1,\dots,\theta^{n-1})\in\R^{n-1}\,:\,\sum_{\a=1}^n\big[\lambda_\a^2\, |\phi^\a|^2(\theta)\big]>\ep\Big\}\,.
$$ 
The following result shows that, for $t$ small enough, the level set $\{u=t\}$, is a graph over $C_\ep$.

\begin{lemmaapp}
\label{le:aux_1}
For any $0<\ep<1$, there exists $\eta=\eta(\ep)>0$ such that
\begin{itemize}
\item[(i)] the estimates $|w|(r,\theta)<\frac{\ep^2}{4}\,r^2$, $|\pa w/\pa r|(r,\theta)<\frac{\ep}{2}\,r$, $|h|(r,\theta)<\ep^2\,r^2$ holds on the whole $B_\eta(y_0)$.
\item[(ii)] it holds $\frac{\pa u}{\pa r}(r,\theta)<0$ in $(0,\eta)\times C_\ep$.
\item[(iii)] for every $0<\delta<\eta$, there exists $\tau=\tau(\delta,\ep)$ such that for any $\tau<t<1$, there exists a smooth function $r_t:C_\ep\rightarrow (0,\delta)$ such that
$$
\{u=t\}\cap B_\delta(y_0)\cap C_\ep\,=\,\{(r_t(\theta),\theta)\,:\,\theta\in C_\ep\}\,.
$$
\end{itemize}
\end{lemmaapp}

\begin{proof}
Since the functions $w,h$ in~\eqref{eq:u_taylor2},~\eqref{eq:Du_taylor2} are $o(r^2)$, while $\pa w/\pa r$ is $o(r)$ thanks to~\eqref{eq:est_dewder}, it is clear that statement (i) is true for some $\eta$ small enough.
Moreover, from expansion~\eqref{eq:u_taylor2} we compute
$$
\frac{\pa u}{\pa r}(r,\theta)\,=\,-\,r\,\sum_{\a=1}^n \big[\,\lambda_\a^2\, |\phi^\a|^2(\theta)\,\big]
\,+\,\frac{\pa w}{\pa r}(r,\theta)\,<\,-\ep\,r\,+\,\frac{\ep}{2}\,r\,=\,-\frac{\ep}{2}\,r\,.
$$
This proves point (ii).
To prove (iii), fix $t\in(0,1)$ and consider the function $u(r,\theta)-t$. Since $u(r,\theta)\to 1^-$ as $r\to 0^+$, we have $u(r,\theta)-t> 0$ for small values of $r$. 

On the other hand, from expansion~\eqref{eq:u_taylor2} we find
$$
u(\delta,\theta)-t\,=\,(1-t)-\frac{\delta^2}{2}\,\sum_{\a=1}^n \big[\,\lambda_\a^2\, |\phi^\a|^2(\theta)\,\big]
\,+\,w(\delta,\theta)
\,<\,(1-t)\,-\,\frac{\ep}{2}\,\delta^2\,+\,w(\delta,\theta)\,<\,(1-t)-\frac{\ep}{4}\,\delta^2\,,
$$
and the quantity on the right is negative for any $t>\tau=1-\frac{\ep}{4}\delta^2$.

Therefore, fixed a $\theta\in C_\ep$ the function $r\mapsto u(r,\theta)-t$ is positive for small values of $r$ and negative for $r=\delta$. Moreover from point (ii) we have that $\frac{\pa u}{\pa r}(r,\theta)<0$ for any $(r,\theta)\in (0,\delta)\times C_\ep$, hence for any $\theta\in C_\ep$, there exists one and only one value $0<r_t(\theta)<\delta$ such that $(r_t(\theta),\theta)\in\{u=t\}$. The smoothness of the function $r_t(\theta)$ is a consequence of the Implicit Function Theorem applied to the function $u(r,\theta)$.
\end{proof}

As anticipated, Lemma~\ref{le:aux_1} will now be used to estimate the density of the restriction of the metric $\go$ on $\{u=t\}\cap ((0,\delta)\times C_\ep)$. 

\begin{lemmaapp}
\label{le:aux_2}
There exists $0<\delta<\eta(\ep)$ such that it holds
$$
\sqrt{\det({\go}_{|_{\{u=t\}}})}\,(r_t(\theta),\theta)
\,\,>\,\,
(1-\ep)\,r_t^{n-1}(\theta)\,\sqrt{\det(g_{\Sph^{n-1}})}\,,
$$
for every $\theta\in C_\ep$, $\tau(\delta,\ep)<t<1$.
\end{lemmaapp}

\begin{proof}
Let $r_t$ be the function introduced in Lemma~\ref{le:aux_1}. Taking the total derivative of $u(r_t(\theta),\theta)=t$, we find, for any $\theta\in C_\ep$
$$
dr_t\,=\,-\Big[\frac{\pa u}{\pa r}(r_t(\theta),\theta)\Big]^{\!-1}\,\sum_{j=1}^{n-1}\frac{\pa u}{\pa \theta^j}(r_t(\theta),\theta)\, d\theta^j
\,=\,-\,r_t(\theta)\sum_{j=1}^{n-1}\xi_j(\theta)d\theta^j\,,
$$
where 
$$
\xi_j(r_t(\theta),\theta)\,=\,\frac{1}{r_t(\theta)}\,\frac{\frac{\pa u}{\pa \theta^j}}{\frac{\pa u}{\pa r}}(r_t(\theta),\theta)\,.
$$

To ease the notation, in the rest of the proof we avoid to explicitate the dependence of the functions on the variables $r_t(\theta),\theta$.
In order to compute the restriction of the metric on $\{u=t\}\cap ((0,\delta)\times C_\ep)$, we substitute the term $dr$ in formula~\eqref{eq:gr_polar} with the formula for $dr_t$ computed above. We obtain
$$
{\go}_{|_{\{u=t\}}}\,=\,r_t^2\,
\Big[\xi_i\,\xi_j(1+\sigma)+g_{ij}^{\Sph^{n-1}}-\frac{\sigma_i}{r_t}\xi_j-\frac{\sigma_j}{r_t}\xi_i+\sigma_{ij}\Big]\,d\theta^i\otimes d\theta^j\,.
$$

Set $\xi=\sum_{j=1}^{n-1}\xi_j d\theta^j$. We have the following 
\begin{align}
\notag
\sqrt{\det\Big[\Big(\xi_i\,\xi_j+g_{ij}^{\Sph^{n-1}}\Big)\,d\theta^i\otimes d\theta^j\Big]}\,&=\,\sqrt{\det\Big(\xi\,\otimes\,\xi\,+\,g_{\Sph^{n-1}}\Big)}
\\
\notag
&=\,\sqrt{(1+|\xi|_{g_{\Sph^{n-1}}}^2)\,\det (g_{\Sph^{n-1}})}
\\
\label{eq:det_go_1}
&\geq\,\sqrt{\det (g_{\Sph^{n-1}})}\,,
\end{align}
where in the second equality we have used the Matrix Determinant Lemma.

On the other hand, since $\sigma_i=o(r^2)$ and $\sigma_{ij}=o(r^3)$, we deduce that
$$
\sigma\,\xi_i\,\xi_j-\frac{\sigma_i}{r_t}\xi_j-\frac{\sigma_j}{r_t}\xi_i+\sigma_{ij}=o(r_t)\,,
$$
hence
\begin{equation}
\label{eq:det_go_2}
\sqrt{\det({\go}_{|_{\{u=t\}}})}\,=\,(1+\omega)\,r_t^{n-1}\,\sqrt{\det\Big[\Big(\xi_i\,\xi_j+g_{ij}^{\Sph^{n-1}}\Big)\,d\theta^i\otimes d\theta^j\Big]}\,,
\end{equation}
with $\omega=o(1)$ as $r\to 0^+$. In particular, we can choose $\delta$ small enough so that $|\omega|<\ep$ on $(0,\delta)\times C_\ep$.
Combining equations~\eqref{eq:det_go_1} and~\eqref{eq:det_go_2} we have the thesis.
\end{proof}

We also need an estimate of the integrand in~\eqref{eq:cond_thesis}, which is provided by the following lemma.

\begin{lemmaapp}
\label{le:aux_3}
We can choose $0<\delta<\eta(\ep)$ such that
$$
\frac{|\D u|^2}{1-u^2}(r,\theta)
\,>\,
(1-\ep)\,\sum_{\a=1}^n \big[\,\lambda_\a^2\, |\phi^\a|^2(\theta)\,\big]
$$
for every $(r,\theta)\in (0,\delta)\times C_\ep$.
\end{lemmaapp}

\begin{proof}
To ease the notation, in this proof we avoid to explicitate the dependence of the functions on the coordinates $r,\theta$.
From expansions~\eqref{eq:u_taylor2} and~\eqref{eq:Du_taylor2} we deduce
\begin{align*}
\frac{|\D u|^2}{2(1-u)}\,&=\,\frac{\sum_{\a=1}^n \big(\,\lambda_\a^4\, |\phi^\a|^2\,\big)+h}{\sum_{\a=1}^n \big(\,\lambda_\a^2\, |\phi^\a|^2\,\big)-2\,w}
\\
&=\,\frac{\sum_{\a=1}^n \big(\,\lambda_\a^4\, |\phi^\a|^2\,\big)}{\sum_{\a=1}^n \big(\,\lambda_\a^2\, |\phi^\a|^2\,\big)}\,\left[1+\frac{h}{\sum_{\a=1}^n (\,\lambda_\a^4\, |\phi^\a|^2\,)}\right]\left[1-\frac{2\,w}{\sum_{\a=1}^n (\,\lambda_\a^2\, |\phi^\a|^2\,)}\right]^{-1}
\!\!\!.
\end{align*}
Using the Cauchy-Schwarz Inequality we have
\begin{equation}
\label{eq:CS_ine}
\sum_{\a=1}^n \big(\lambda_\a^2\,|\phi^\a|^2\,\big)
\,\leq\,
\bigg[\sum_{\a=1}^n \big(\lambda_\a^4\,|\phi^\a|^2\,\big)\bigg]^{\frac{1}{2}}\,\cdot\,\bigg[\sum_{\a=1}^n |\phi^\a|^2\bigg]^{\frac{1}{2}}
\,=\,
\bigg[\sum_{\a=1}^n \big(\lambda_\a^4\,|\phi^\a|^2\,\big)\bigg]^{\frac{1}{2}}\,,
\end{equation}
hence, recalling that $\sum_{\a=1}^n (\lambda_\a^2 |\phi^\a|^2)>\ep$ on $C_\ep$, we have also $\sum_{\a=1}^n (\lambda_\a^4 |\phi^\a|^2)>\ep^2$. Therefore, from Lemma~\ref{le:aux_1}-(i) we easily compute
$$
\bigg|\left(1+\frac{h}{\sum_{\a=1}^n (\,\lambda_\a^4\, |\phi^\a|^2\,)}\right)\left(1-\frac{2\,w}{\sum_{\a=1}^n (\,\lambda_\a^2\, |\phi^\a|^2\,)}\right)^{-1}\bigg|\,>\,\frac{1-\delta^2}{1+\frac{\ep\delta^2}{2}}\,.
$$
In particular, we can choose $\delta$ small enough so that the right hand side of the inequality above is greater than $1-\ep$. Hence, we get
\begin{equation*}
%\label{eq:grad_near_1}
\frac{|\D u|^2}{1-u^2}
\,\geq\,
\frac{|\D u|^2}{2(1-u)}
\,\geq\,(1-\ep)\,\frac{\sum_{\a=1}^n \big(\,\lambda_\a^4\, |\phi^\a|^2\,\big)}{\sum_{\a=1}^n \big(\,\lambda_\a^2\, |\phi^\a|^2\,\big)}
\,\geq\,
(1-\ep)\,\sum_{\a=1}^n \big[\,\lambda_\a^2\, |\phi^\a|^2\,\big]
\,,
\end{equation*}
where in the first inequality we have used that $u\leq 1$ on $M$ and in the latter inequality we have used~\eqref{eq:CS_ine}.
\end{proof}

Now we are finally able to prove our proposition.

\begin{proof}[Proof of Proposition~\ref{prop:lim_estimate}]
For every $\ep>0$ and $0<\delta\leq d$, we have the following estimate of the left hand side of condition~\eqref{eq:cond_thesis}
\begin{align}
\notag
\bigg(\frac{1}{1-t^2}\bigg)^{n-1}\!\!\!\!\!\!\!\!\!\!\int\limits_{\{u=t\}\cap B_d(y_0)} \!\!\!\!\!\!|\D u|^{n-1}\, \rmd\sigma
\,\,&\geq\,\,
\bigg(\frac{1}{1-t^2}\bigg)^{n-1}\!\!\!\!\!\!\!\!\!\!\int\limits_{\{u=t\}\cap B_\delta(y_0)} \!\!\!\!\!\!|\D u|^{n-1} \,\rmd\sigma
%\\
%\notag
%&=\,\,
%\bigg[\frac{1}{2(\umax-t)}\bigg]^{\frac{n-1}{2}}\int_{E_{\bar\delta,\ep}(t)} \bigg[\frac{|\D u|^2}{2(\umax-t)}\bigg]^{\frac{n-1}{2}} \sqrt{\det({\gr}_{|_{W_{\bar\delta,\ep}(t)}})}\,\rmd\sigma_{\Sph^{n-1}}
\\
\label{eq:int_in1}
&=\!\!\!\! 
\int\limits_{\{u=t\}\cap B_\delta(y_0)} \!\!\!\!\!\!\bigg(\frac{|\D u|^2}{1-u^2}\bigg)^{\frac{n-1}{2}}\,\bigg(\frac{2}{1+u}\bigg)^{\frac{n-1}{2}}\,\bigg[\frac{1}{2(1-u)}\bigg]^{\frac{n-1}{2}}\,\rmd\sigma
\end{align}
Since $u\leq 1$, we have $2/(1+u)\geq 1$. Moreover, from~\eqref{eq:u_taylor2} and Lemma~\ref{le:aux_1}-(i), we obtain
$$
[2(1-u)](r,\theta)\,=\,r^2\,\sum_{\a=1}^n\big[\lambda_\a^2|\phi^\a|^2(\theta)\big]\,\Big(1-\frac{w(r,\theta)}{r^2\sum_{\a=1}^n\lambda_\a|\phi^\a|^2(\theta)}\Big)
\,<\,
(1+\ep)\,r^2\,\sum_{\a=1}^n\big[\lambda_\a^2|\phi^\a|^2(\theta)\big]\,.
$$
Now fix a $\delta$ small enough so that Lemmas~\ref{le:aux_2},~\ref{le:aux_3} are in charge. Taking the limit of integrand~\eqref{eq:int_in1} as $t\to 1^-$ we obtain the estimate
\begin{align}
\!\!\!\liminf_{t\to 1^-}\bigg(\frac{1}{1-t^2}\bigg)^{n-1}\!\!\!\!\!\!\!\!\!\!\int\limits_{\{u=t\}\cap B_d(y_0)} \!\!\!\!\!\!|\D u|^{n-1}\, \rmd\sigma
\,\,
\notag
&>\,\, 
\int\limits_{C_\ep}\,(1-\ep)\,\Big(\frac{1-\ep}{1+\ep}\Big)^{\frac{n-1}{2}}\,\sqrt{\det(g_{\Sph^{n-1}})}\,d\theta^1\cdots d\theta^{n-1}
\\
\label{eq:int_in}
&= 
\int\limits_{\R^{n-1}}\!\!{\chi}^{\phantom{\frac{1}{2}}}_{C_\ep}(\theta)\,\frac{(1-\ep)^{\frac{n+1}{2}}}{(1+\ep)^{\frac{n-1}{2}}}\,\sqrt{\det(g_{\Sph^{n-1}})}\,d\theta^1\cdots d\theta^{n-1}\,.
\end{align}

It is clear that the functions ${\chi}^{\phantom{\frac{1}{2}}}_{C_\ep}$ converge to ${\chi}^{\phantom{\frac{1}{2}}}_{C_0}$ as $\ep\to 0^+$, where
$$
C_0\,=\,\Big\{\theta=(\theta^1,\dots,\theta^{n-1})\in\R^{n-1}\,:\,\sum_{\a=1}^n\big[\lambda_\a^2\, |\phi^\a|^2(\theta)\big]\neq 0\Big\}\,\subseteq \,\Sph^{n-1}\,.
$$ 
Therefore, taking the limit of~\eqref{eq:int_in} as $\ep\to 0^+$ and using the Monotone Convergence Theorem, we find
\begin{equation*}
\liminf_{t\to 1^-} \bigg(\frac{1}{1-t^2}\bigg)^{n-1}\!\!\!\!\!\!\!\!\!\!\int\limits_{\{u=t\}\cap B_d(y_0)} \!\!\!\!\!\!|\D u|^{n-1}\, \rmd\sigma
\,\,\geq\,\, 
\int\limits_{\Sph^{n-1}} {\chi}^{\phantom{\frac{1}{2}}}_{C_0}\,\rmd\sigma_{\Sph^{n-1}}\,.
\end{equation*}
To end the proof, it is enough to show that the set $\Sph^{n-1}\setminus C_0$ is negligible. But this is clear. In fact, since $\sum_{\a=1}^n\lambda_\a^2=n$, there exists at least one integer $\b$ such that $\lambda_\b\neq 0$. Thus $\Sph^{n-1}\setminus C_0$ is contained in the hypersurface $ \{\phi^\b=0\}$, hence its $n$-measure is zero.
This proves inequality~\eqref{eq:cond_thesis} and the thesis. 
\end{proof}

This concludes the proof for the de Sitter case.
In the anti-de Sitter case we can prove the following analogue of Theorem~\ref{thm:MAX}.

\begin{theoremapp}
\label{thm:MIN}
Let $(M,\go,u)$ be a conformally compact static solution of problem~\ref{eq:pb_A} satisfying Assumption~\ref{ass:A}. 
Then the set ${\rm MIN}(u)$ is discrete (and finite) and 
$$
\liminf_{t\to 1^+} \,U_p(t)\,\,\geq\,\, |{\rm MIN}(u)|\,|\Sph^{n-1}|\,,
$$
for every $p\leq n-1$.
\end{theoremapp}

The proof follows the exact same scheme as the de Sitter case, the only small modifications being in the proof of Lemma~\ref{le:aux_2} and in the computation~\eqref{eq:int_in}, where we have used the fact that $u\leq 1$. This is not true anymore, however, since we are working around a minimum point, we can suppose $u< 1+\kappa$, where $\kappa$ is an infinitesimal quantity that can be chosen to be as small as necessary. Aside from this little expedient, the proof is virtually the same as the de Sitter case, thus we omit it.

We pass now to the proof of some other results that we have used in our work.
The next lemma is useful in order to study the behavior of the static solutions of problem~\eqref{eq:pb_A} near the conformal boundary.

\begin{lemmaapp}
\label{le:tot_geod_BGH_A} 
Let $(M,\go,u)$ be a conformally compact static solution to problem~\eqref{eq:pb_A}. Suppose that $1/\sqrt{u^2-1}$ is a defining function, so that the metric $g=\go/(u^2-1)$ extends to the conformal boundary $\pa M$. Then
\begin{itemize}
\item[(i)]
$\lim_{x\to \bar{x}}(u^2-1-|\D u|^2)$ is well-definite and finite for every $\bar x\in\pa M$,
\item[(ii)] $\pa M$ is a totally geodesic hypersurface in $(\overline{M},g)$.
\end{itemize}
\end{lemmaapp}

\begin{proof}
For the proof of this result, it is convenient to use the notations introduced in Section~\ref{sec:conf_reform}. 
Let $\ffi$ be the function defined by~\eqref{eq:ffi_A}.
By hypotesis, $M$ is the interior of a compact manifold $\overline{M}$ and the metric $g$ is well defined on the whole $\overline{M}$. In particular, the scalar curvature $\Rg$ is a smooth finite function at $\pa M$. Therefore, from equation~\eqref{eq:tilde_R} we easily deduce that $\lim_{x\to\bar x}u^2(1-|\na\ffi|_g^2)$ is well-definite and finite for every $\bar x\in\pa M$. Since
$$
|\na\ffi|_g^2=\frac{|\D u|^2}{u^2-1}\,,
$$
this proves point (i).

To prove statement (ii), we first observe that, since $|\na\ffi|_g=1$ at $\pa M$ (as it follows immediately from point (i)), there exists $\delta>0$ such that $|\na\ffi|_g\neq 0$ on the whole collar $\mathcal{U}_\delta=\{\ffi<\delta\}$. 
%In particular, $\ffi$ can be regarded as a coordinate on $\mathcal{U}_\delta$ and $\mathcal{U}_{\delta}$ is diffeomorfic to the product $[0,\delta]\times\pa M$. Therefore, we can find coordinates $\{\ffi,\vartheta^1,\dots,\vartheta^{n-1}\}$ on $\mathcal{U}_\delta$, where $\vartheta^1,\dots,\vartheta^{n-1}$ are coordinates on $\pa M$. 
Therefore, proceeding as in Subsection~\ref{sub:geom_levelsets}, we find a set of coordinates $\{\ffi,\vartheta^1,\dots,\vartheta^{n-1}\}$ on $\mathcal{U}_{\delta}$, such that the metric $g$ writes as
$$
g=\frac{d\ffi\otimes d\ffi}{|\na\ffi|_g^2}+g_{ij}(\ffi,\theta^i,\dots,\theta^{n-1})d\theta^i\otimes d\theta^j\,.
$$
With respect to these coordinates, the second fundamental form of the boundary $\pa M=\{\ffi=0\}$ is
$$
\chg_{ij}\,=\,\frac{\nana_{ij}\ffi}{|\na\ffi|_g}=\nana_{ij}\ffi\,,\qquad\hbox{for } i,j=1,\dots,n-1\,.
$$
On the other hand, from the first equation of problem~\eqref{eq:pb_conf}, we easily deduce that $\nana\ffi=0$ on $\pa M$. This concludes the proof of point (ii).
\end{proof}

Finally, in order to prove the integral identities in Section~\ref{sec:integral}, we need an extension of the classical Divergence Theorem to the case of open domains whose boundary has a (not too big) nonsmooth portion. Note that~\cite[Theorem~A.1]{Ago_Maz_2} is not enough for our purposes, because hypotesis (ii) is not necessarily fulfilled. To avoid problems, we state the following generalization, due to De Giorgi and Federer.

\begin{theoremapp}[\cite{DeGiorgi,Federer2}]
\label{thm:div}
Let $(M,g)$ be a $n$-dimensional Riemannian manifold, 
with $n\geq 2$,
let $E \subset M$ be a bounded open 
subset of $M$ with compact boundary $\pa E$ of finite $(n-1)$-dimensional Hausdorff measure, 
and suppose that $\pa E = \Gamma \sqcup \Sigma$, where the subsets 
$\Gamma$ and $\Sigma$ have the following properties:
\begin{itemize}
\item[(i)] For every $x \in \Gamma$, there exists an open neighborhood $U_x$ of $x$ in $M$ such that $\Gamma \cap U_x$ is a smooth regular hypersurface. 
\item[(ii)] 
The subset $\Sigma$ is compact and $\mathscr{H}^{n-1}(\Sigma) = 0$.
\end{itemize}
If $X$ is a Lipschitz vector field defined in a neighborhood of $\overline{E}$
then the following identity holds true
\begin{equation}
\label{eq:thm_div}
\int\limits_E {\rm div} X \, \rmd \mu \, = \, \int\limits_{\Gamma} \langle  X \,|\, {\rm n}  \rangle \, \rmd \sigma  ,
\end{equation}
where ${\rm n}$ denotes the exterior unit normal vector field.
\end{theoremapp}

\renewcommand{\theequation}{B-\arabic{equation}}
\renewcommand{\thesection}{B}
\renewcommand{\thesubsection}{B}

\subsection{Boucher-Gibbons-Horowitz method}
\label{sec:app_BGH}
In this section we discuss an alternative approach to the study of the rigidity of the de Sitter and anti-de Sitter spacetime, which does not require the machinery of Section~\ref{sec:conf_reform}. Without the need of any assumptions, this method will allow to derive results that are comparable to Theorems~\ref{thm:main2_geom_D},~\ref{thm:rig_gen_D} (case $\Lambda>0$) and Theorems~\ref{thm:main2_geom_A},~\ref{thm:rig_gen_A} (case $\Lambda<0$). In the case $\Lambda>0$, the computations that we are going to show are quite classical (see~\cite{Bou_Gib_Hor,Chr2}). However, to the author's knowledge, the analogous calculations in the case $\Lambda<0$ are new.

As usual, we start with the case $\Lambda>0$. Recalling the Bochner formula and the equations in~\eqref{eq:pb_D} we compute

\begin{align}
\notag
\De|\D u|^2\,&=\,2\,|\DD u|^2\,+\,2\,\Ric(\D u,\D u)\,+\,2\langle\D\De u\,|\,\D u\rangle
\\
\notag
&=\,2\,|\DD u|^2\,+\,2\,\Big[\,\frac{1}{u}\,\DD u(\D u,\D u)\,+\,n\,|\D u|^2\,\Big]\,-\,2n\,|\D u|^2
\\
\label{eq:boch_BGH_D}
&=\,2\,|\DD u|^2\,+\,\frac{1}{u}\,\langle\D|\D u|^2\,|\,\D u\rangle\,.
\end{align}

Now, if we consider the field 
$$
Y\,=\,\D|\D u|^2\,-\,\frac{2}{n}\,\De u\,\D u
$$
we can compute its divergence using~\eqref{eq:boch_BGH_D}.
\begin{align*}
{\rm div}(Y)\,&=\, \De|\D u|^2\,-\,\frac{2}{n}\langle\D\De u\,|\,\D u\rangle\,-\,\frac{2}{n}(\De u)^2
\\
&=\,2\,\Big[\,|\DD u|^2-\frac{(\De u)^2}{n}\,\Big]\,+\,\frac{1}{u}\,\langle\D|\D u|^2\,|\,\D u\rangle\,+\,2\,|\D u|^2\,.
\end{align*}

More generally, for every nonzero $\mathscr{C}^1$ function $\alpha=\alpha(u)$:
\begin{align*}
\frac{{\rm div} (\alpha\, Y)}{\alpha}\,&=\,{\rm div}(Y)\,+\,\frac{\dot\alpha}{\alpha}\,\langle Y\,|\,\D u\rangle
\\
&=\,2\,\Big[\,|\DD u|^2-\frac{(\De u)^2}{n}\,\Big]\,+\,\Big(\frac{\dot\alpha}{\alpha}+\frac{1}{u}\Big)\Big(\langle\D|\D u|^2\,|\,\D u\rangle\,+\,2\,u\,|\D u|^2\Big)\,.
\end{align*}
where $\dot\alpha$ is the derivative of $\alpha$ with respect to $u$.
The computation above suggests us to choose
$$
\alpha(u)\,=\,\frac{1}{u}\,.
$$
With this choice of $\alpha$, we have
\begin{equation}
\label{eq:div_BGH_D}
{\rm div}\Big(\,\frac{1}{u}\,Y\,\Big)\,=\,\frac{2}{u}\,\Big[\,|\DD u|^2-\frac{(\De u)^2}{n}\,\Big]\,.
\end{equation}

\begin{propositionapp}
\label{prop:cyl_BGH_D}
Let $(M,\go,u)$ be a static solution to problem~\eqref{eq:pb_D}. Then, for every $t\in [0,1)$ it holds
\begin{equation}
\label{eq:id_byparts_BGH_D}
\int\limits_{\{u=t\}}\!\!\! \frac{1}{u} \left(|\D u|^2 \, \HHH -\, \frac{n-1}{n}\,|\D u|\,\De u\right) \rmd\sigma
\,\,=\,\,
\int\limits_{\{u>t\}}\!\!\! \frac{1}{u}\, \Big[\,|\DD u|^2-\frac{(\De u)^2}{n}\,\Big]\,\rmd\mu
\,\,\geq\,\,0 
\,.
\end{equation}
Moreover, if there exists $t_0\in (0,1)$ such that 
\begin{equation}
\label{eq:condition_cyl_BGH_D}
\int\limits_{\{u=t_0\}}
\!\!\! \left(|\D u|^2\,\HHH \,-\, \frac{n-1}{n} \,|\D u|\, \De u\right) \!
\,\rmd\sigma  \, \leq \, 0 \, ,
\end{equation}
then the triple $(M,\go,u)$ is isometric to the de Sitter solution.
\end{propositionapp}

\begin{remarkapp}
\label{rem:a_e_BGH_D}
Recalling Remark~\ref{rem:uno_D}, it is easy to realize that the integral on the left hand side of~\eqref{eq:id_byparts_BGH_D} is well defined also when $t$ is a singular value of $u$.
\end{remarkapp}

\begin{remarkapp}
Note that the right hand side of inequality~\eqref{eq:id_byparts_BGH_D} is always nonnegative, as opposed to formula~\eqref{eq:id_byparts}, where we needed to suppose Assumption~\ref{ass:conf} to achieve the same result. This is one of the reasons why this approach works without the need to suppose any assumption.
\end{remarkapp}

\smallskip

\begin{proof}[Proof of Proposition~\ref{prop:cyl_BGH_D}]
Suppose for the moment that $\{u=t\}$ is a regular level set. Integrating by parts identity~\eqref{eq:div_BGH_D}, we obtain
\begin{equation}
\label{eq:int_part_fin_BGH_D}
\int\limits_{\{u>t\}}\!\!\! \frac{2}{u}\, \Big[\,|\DD u|^2-\frac{(\De u)^2}{n}\,\Big]\,\rmd\mu
\,=\,
\int\limits_{\{u=t\}}\!\!\! \frac{1}{u}\, \langle Y\,|\, \mathrm{n}\rangle\, \rmd\sigma\,,
\end{equation}  
where ${\rm n} = -\D u/|\D u|$ is the outer $g$-unit normal
of the set $\{ u \geq t \}$ at its boundary. On the other hand, from the first formula in~\eqref{eq:formula_curvature} it is easy to deduce that
$$
\langle Y\,|\,\D u\rangle
\, = \, 2\left( \DD u(\D u,\D u)\,-\,\frac{|\D u|^2\De u}{n} \right)\, = \,  - 2\left(|\D u|^3 \, \HHH -\, \frac{n-1}{n}\,|\D u|^2\,\De u\right) . 
$$
Substituting in~\eqref{eq:int_part_fin_BGH_D} proves formula~\eqref{eq:id_byparts_BGH_D} in the case where $\{u=t\}$ is a regular level set.

In the case where $t>0$ is a singular value of $u$, we need to apply a slightly refined version of the Divergence Theorem, namely Theorem~\ref{thm:div} in the appendix, in order to perform the integration by parts which leads to identity~\eqref{eq:int_part_fin_BGH_D}. The rest of the proof is identical to what we have done for the regular case. We set 
\begin{eqnarray*}
X  =  \frac{1}{u}\, Y \qquad \hbox{and} \qquad E= \{ u>t \}\, .
\end{eqnarray*}
so that $\pa E = \{ u=t \}$. 

As usual, we denote by ${\rm Crit}(u) = \{  x\in M \, | \, \D u(x)= 0  \}$ the set of the critical points of $u$, From~\cite{Krantz}, we know that there exists an open $(n-1)$-dimensional submanifold $N\subseteq {\rm Crit}(u)$ such that $\mathscr{H}^{n-1}({\rm Crit}(u)\setminus N)=0$.
Set $\Sigma = \pa E \cap ({\rm Crit}(u)\setminus N)$ and $\Gamma = \pa E \setminus \Sigma$, so that $\pa E$ can be written as the disjoint union of $\Sigma$ and $\Gamma$.
We have $\mathscr{H}^{n-1}(\Sigma)=0$ by definition, while $\Gamma$ is the union of the regular part of $\pa E$ and of $N$, which are open $(n-1)$-submanifolds.
Therefore the hypoteses of Theorem~\ref{thm:div} are met, and we can apply it to conclude that equation~\eqref{eq:int_part_fin_BGH_D} holds true also when $t$ is a singular value of $u$.

To prove the second part, we observe that from~\eqref{eq:id_byparts_BGH_D} and~\eqref{eq:condition_cyl_BGH_D} one immediately gets $\DD u = ({\De u}/{n})\, \go$ in $\{ u \geq t_0 \}$. Since $u$ is analytic, the same equality holds on the whole manifold $M$.
Now we can use the results in~\cite{Lafontaine} to conclude that $(M,\go,u)$ is the de Sitter solution.
\end{proof}

The proposition above is particularly interesting when applied at the boundary $\pa M=\{u=0\}$. 

\begin{corollaryapp}
\label{cor:bound_BGH_D}
Let $(M,\go,u)$ be a static solution to problem~\eqref{eq:pb_D}. Then it holds
\begin{equation}
\label{eq:bound_BGH_D}
\int\limits_{\pa M}\!\! |\D u|\big[\,\RRR^{\pa M} -\, (n-1)(n-2)\,\big] \,\rmd\sigma\,\geq 0\,.
\end{equation}
Moreover, if the equality holds then the triple $(M,\go,u)$ is isometric to the de Sitter solution.
\end{corollaryapp}

\begin{proof}
First we compute from the equations in~\eqref{eq:pb_D} and formula~\eqref{eq:formula_curvature}, that
$$
\frac{\HHH\,|\D u|}{u}\,=\,-\Ric(\nu,\nu)\,,
$$
where $\nu=\D u/|\D u|$ as usual. In particular, we have $\HHH=0$ on $\pa M$. Hence we can use the Gauss-Codazzi identity to find
$$
\frac{\HHH\,|\D u|}{u}\,=\,\frac{\RRR^{\pa M}-\RRR}{2}\,=\,\frac{1}{2}\,\big[\,\RRR^{\pa M}-\,n(n-1)\,\big]\,.
$$
Substituting $t=0$ in formula~\eqref{eq:id_byparts_BGH_D} and applying Proposition~\ref{prop:cyl_BGH_D}, we have the thesis.
\end{proof}

Now we turn our attention to the case $\Lambda<0$.
Mimicking the computations done in the case $\Lambda>0$, but using the equations in~\eqref{eq:pb_A} instead of the ones in~\eqref{eq:pb_D} we obtain
\begin{equation}
\label{eq:div_BGH_A}
{\rm div}\Big(\,\frac{1}{u}\,Y\,\Big)\,=\,\frac{2}{u}\,\Big[\,|\DD u|^2-\frac{(\De u)^2}{n}\,\Big]\,.
\end{equation}
Incidentally, we notice that this equation coincides with the analogous formula~\eqref{eq:div_BGH_D} in the case $\Lambda>0$. We are now ready to state the analogous of Proposition~\ref{prop:cyl_BGH_D}.

\begin{propositionapp}
\label{prop:cyl_BGH_A}
Let $(M,\go,u)$ be a static solution to problem~\eqref{eq:pb_A}. Then, for every $t\in (1,+\infty)$ it holds
\begin{equation}
\label{eq:id_byparts_BGH_A}
\int\limits_{\{u=t\}}\!\!\! \frac{1}{u} \left(|\D u|^2 \, \HHH -\, \frac{n-1}{n}\,|\D u|\,\De u\right) \rmd\sigma
\,\,=\,\,
-\!\!\int\limits_{\{u<t\}}\!\!\! \frac{1}{u}\, \Big[\,|\DD u|^2-\frac{(\De u)^2}{n}\,\Big]\,\rmd\mu
\,\,\leq\,\,0 
\,.
\end{equation}
Moreover, if there exists $t_0\in (1,+\infty)$ such that 
\begin{equation}
\label{eq:condition_cyl_BGH_A}
\int\limits_{\{u=t_0\}}
\!\!\! \left(|\D u|^2\,\HHH \,-\, \frac{n-1}{n} \,|\D u|\, \De u\right) \!
\,\rmd\sigma  \, \geq \, 0 \, ,
\end{equation}
then the triple $(M,\go,u)$ is isometric to the anti-de Sitter solution.
\end{propositionapp}

\begin{remarkapp}
\label{rem:a_e_BGH_A}
Recalling Remark~\ref{rem:uno_D}, it is easy to realize that the integral on the left hand side of~\eqref{eq:id_byparts_BGH_A} is well defined also when $t$ is a singular value of $u$.
\end{remarkapp}

\smallskip

\begin{proof}[Proof of Proposition~\ref{prop:cyl_BGH_A}]
The proof is almost identical to the proof of Proposition~\ref{prop:cyl_BGH_A}. The only change is that, when we apply the divergence theorem, we need the outer $g$-unit normal of the set $\{u\leq t\}$, that is ${\rm n} = \D u/|\D u|$ instead of $-\D u/|\D u|$. This is the reason of the different signs in formul\ae~\eqref{eq:id_byparts_BGH_D},~\eqref{eq:id_byparts_BGH_A}.
\end{proof}

Now suppose that the manifold $M$ is conformally compact. We would like to use Proposition~\ref{prop:cyl_BGH_A} to study the behavior of a static solution at the conformal boundary $\pa M$. In order to simplify the computations and to emphasize the analogy with the case $\Lambda>0$, it will prove useful to suppose that Assumption~\ref{ass:A2} holds. Therefore, from now on we suppose that $1/\sqrt{u^2-1}$ is a defining function, and that $\lim_{x\to \bar x}(u^2-1-|\D u|^2)=0$ for every $\bar x\in\pa M$.
We are now ready to prove the analogous of Corollary~\ref{cor:bound_BGH_D} in the case of a negative cosmological constant.

\begin{corollaryapp}
\label{cor:bound_BGH_A}
Let $(M,\go,u)$ be a conformally compact static solution to problem~\eqref{eq:pb_A} satisfying assumption~\ref{ass:A2}, and let $g=\go/(u^2-1)$. Then it holds
\begin{equation}
\label{eq:bound_BGH_A}
\int\limits_{\pa M}\! \Big[\,(n-1)(n-2)-\Rg^{\pa M}
\,\Big] \,\rmd\sigma_g
\,\,\geq\,\,0 
\,.
\end{equation}
Moreover, if
\begin{equation}
\label{eq:bound_rig_BGH_A}
\lim_{t\to+\infty} \,t^{n-1}\!\!\!\int\limits_{\{u=t\}}\! \Ricg(\nu_g,\nu_g)\, \rmd\sigma_g\,
\,\,=\,\,0 
\,,
\end{equation}
where $\nu_g=\D u/|\D u|_g$,
then the triple $(M,\go,u)$ is isometric to the anti-de Sitter solution. 
\end{corollaryapp}

\begin{proof}
First we compute from the equations in~\eqref{eq:pb_A} and formula~\eqref{eq:formula_curvature}, that
$$
\frac{\HHH\,|\D u|}{u}\,=\,-\Ric(\nu,\nu)\,,
$$
where $\nu=\D u/|\D u|$ as usual. Therefore, we can rewrite formula~\eqref{eq:id_byparts_BGH_A} as
\begin{equation}
\label{eq:Ric_Ricg_BGH_A}
\int\limits_{\{u=t\}}\!\!\! |\D u|\,\big[\,\Ric(\nu,\nu) \,+\,(n-1)\,\big] \rmd\sigma
\,\,\geq\,\,0 
\,.
\end{equation}
Now we use equation~\eqref{eq:ricci_2_A} in order to rewrite the term in the square brackets in the following way
$$
\big[\,(n-1)u^2-1\,\big]\,\big[\,\Ric(\nu,\nu)\,+\,(n-1)\,\big]
\,=\,
\Ricg(\nu_g,\nu_g)-\bigg(\frac{(n-1)u^2+1}{u^2-1}\bigg)\left(u^2-1-|\D u|^2\right)\,.
$$
Now it is easy to obtain from inequality~\eqref{eq:Ric_Ricg_BGH_A} the following formula
$$
\int\limits_{\{u=t\}}\!\!\! \bigg(\frac{|\D u|}{\sqrt{u^2-1}}\bigg)\,\bigg[\,\Ric_g(\nu_g,\nu_g) -\bigg(\frac{(n-1)u^2+1}{u^2-1}\bigg)\left(u^2-1-|\D u|^2\right)\,\bigg] \rmd\sigma_g
\,\,\geq\,\,0 
\,.
$$
Since $\lim_{x\to \bar x} (u^2-1-|\D u|^2)=0$, in particular $|\D u|/\sqrt{u^2-1}$ goes to zero as $t\to +\infty$. Therefore, taking the limit as $t\to +\infty$ of the formula above, we obtain
\begin{equation}
\label{eq:aux_BGH_A}
\int\limits_{\pa M}\!\Ric_g(\nu_g,\nu_g) \, \rmd\sigma_g
\,\,\geq\,\,0 
\,,
\end{equation}
where $\nu_g=\D u/|\D u|_g$.
Since $\pa M$ is a totally geodesic hypersurface by Lemma~\ref{le:tot_geod_BGH_A}-(ii), from the Gauss-Codazzi equation and formula~\eqref{eq:tilde_R} we obtain
$$
2\,\Ricg(\nu_g,\nu_g)\,=\,\Rg-\Rg^{\pa M}=\,(n-1)(n-2)-\Rg^{\pa M}\,.
$$
Substituting in equation~\eqref{eq:aux_BGH_A} we obtain formula~\eqref{eq:bound_BGH_A}.

To prove the rigidity statement, we observe that we can rewrite formula~\eqref{eq:bound_rig_BGH_A} as
\begin{align*}
0\,&=\,
\lim_{t\to+\infty}\int\limits_{\{u=t\}}\! u^{n-1}\Ricg(\nu_g,\nu_g)\, \rmd\sigma_g
\\
&=\, 
\lim_{t\to+\infty}\int\limits_{\{u=t\}}\! u^{n-1}\Big(\frac{|\D u|}{\sqrt{u^2-1}}\Big)\big[(n-1)u^2+1\big]\bigg[\,\Ric(\nu,\nu)+(n-1)+1-\frac{|\D u|^2}{u^2-1}\,\bigg] \bigg[\frac{\rmd\sigma^{\phantom{^{n-1}}}}{(u^2-1)^{\frac{n-1}{2}}}\bigg]
\\
&=\, 
\lim_{t\to+\infty}\int\limits_{\{u=t\}}\! (n-1)\,\Big[u\,|\D u|\big(\,\Ric(\nu,\nu)+(n-1)\,\big)+(u^2-1-|\D u|^2)\Big]\,\rmd\sigma\,
\\
&=\, 
\lim_{t\to+\infty}\int\limits_{\{u=t\}}\! (n-1)\,u\,|\D u|\Big[\,\Ric(\nu,\nu)+(n-1)\,\Big]\,\rmd\sigma\,.
\end{align*}
Now we recall that $u\,[\Ric(\nu,\nu)+(n-1)]=-\HHH|\D u|\,-\,(n-1)\De u/n$ and we conclude using Proposition~\ref{prop:cyl_BGH_A}.
\end{proof}

%%%%%%%%%%%%%%%%%%%%%%%%%%%%%%%%%%%%%%%%%%%%%%%
%%%%%%%%%%%%%%%%%%%%%%%%%%%%%%%%%%%%%%%%%%%%%%%

\subsection*{Acknowledgements}
{\em The authors would like to thank P. T. Chru\'sciel for his interest in our work and for stimulating discussions during the preparation of the manuscript. L.~M. has been partially supported by the Italian project FIRB 2012 ``Geometria Differenziale e Teoria Geometrica delle Funzioni''. The authors are members of the Gruppo Nazionale per l'Analisi Matematica, la Probabilit\`a e le loro Applicazioni (GNAMPA) of the Istituto Nazionale di Alta Matematica (INdAM) and are partially founded by the GNAMPA Project ``Principi di fattorizzazione, formule di monotonia e disuguaglianze geometriche''.
}

%%%%%%%%%%%%%%%%%%%%%%%%%%%%%%%%%%%%%%%%%%%%%%%
%%%%%%%%%%%%%%%%%%%%%%%%%%%%%%%%%%%%%%%%%%%%%%%

\bibliographystyle{plain}

\end{document}